\documentclass[a4paper,10pt,reqno]{amsart}
\usepackage[latin1]{inputenc}
\usepackage{dsfont}
\usepackage{amsxtra}
\usepackage{amsmath}
\usepackage{amssymb}
\usepackage{amsfonts}
\usepackage[all]{xy}
\usepackage{mathrsfs}
\usepackage{amsthm}
\usepackage{color}

\theoremstyle{plain}
\newtheorem{thm}{Theorem}[section]
\newtheorem{lemma}[thm]{Lemma}
\newtheorem{claim}[thm]{Claim}
\newtheorem{prop}[thm]{Proposition}
\newtheorem{prop-df}[thm]{Proposition-Definition}
\newtheorem{conj}[thm]{Conjecture}
\newtheorem{cor}[thm]{Corollary}
\newtheorem{df}[thm]{Definition}

\theoremstyle{definition}
\theoremstyle{remark}
\newtheorem{rk}[thm]{Remark}

\numberwithin{equation}{section}

\def\bbA{\mathbb{A}}

\def\bbC{\mathbb{C}}

\def\bbH{\mathbb{H}}

\def\bbN{\mathbb{N}}

\def\bbP{\mathbb{P}}

\def\bbV{\mathbb{V}}
\def\bbW{\mathbb{W}}

\def\bbZ{\mathbb{Z}}

\def\scrA{\mathscr{A}}

\def\scrC{\mathscr{C}}

\def\scrP{\mathscr{P}}

\def\fraks{\mathfrak{s}}
\def\frakl{\mathfrak{l}}

\def\frakS{\mathfrak{S}}

\def\calC{\mathcal{C}}

\def\calE{\mathcal{E}}
\def\calF{\mathcal{F}}
\def\calG{\mathcal{G}}
\def\calH{\mathcal{H}}

\def\calL{\mathcal{L}}
\def\calM{\mathcal{M}}
\def\calN{\mathcal{N}}
\def\calO{\mathcal{O}}
\def\calP{\mathcal{P}}

\def\calR{\mathcal{R}}
\def\calS{\mathcal{S}}

\def\calU{\mathcal{U}}

\def\calY{\mathcal{Y}}

\def\frakb{\mathfrak{b}}

\def\frakg{\mathfrak{g}}

\def\frakl{\mathfrak{l}}
\def\frakm{\mathfrak{m}}

\def\frakt{\mathfrak{t}}

\def\bfone{\mathbf{1}}
\def\bfb{\mathbf{b}}
\def\bfp{\mathbf{p}}

\def\bfg{\mathbf{g}}

\def\bfK{\mathbf{K}}

\def\bfZ{\mathbf{Z}}
\def\bfO{\mathbf{O}}

\def\bft{\mathbf{t}}
\def\bfA{\mathbf{A}}

\def\bfC{\mathbf{C}}

\def\bfD{\mathbf{D}}
\def\bfF{\mathbf{F}}
\def\bfP{\mathbf{P}}

\def\oo{{{\rm{o}}}}
\def\iso{{\buildrel\sim\over\to}}
\def\dd{{{\rm{d}}}}

\def\ghom{\operatorname{hom}}
\def\homo{\operatorname{\it \mathscr{H}\kern-.25em om}}
\def\ext{\operatorname{\it \mathscr{E}\kern-.25em xt}}
\def\edo{\operatorname{\it \mathscr{E}\kern-.25em nd}}
\def\der{\operatorname{\it \mathscr{D}\kern-.25em er}}

\def\triv{\operatorname{triv}\nolimits}
\def\sgn{\operatorname{sgn}\nolimits}
\def\Inj{\operatorname{Inj}\nolimits}
\def\Proj{\operatorname{Proj}\nolimits}
\def\Tilt{\operatorname{Tilt}\nolimits}
\def\codim{\operatorname{codim}\nolimits}

\def\hom{\operatorname{hom}\nolimits}
\def\ext{\operatorname{ext}\nolimits}
\def\Hom{\operatorname{Hom}\nolimits}
\def\RHom{\operatorname{RHom}\nolimits}
\def\Spec{\operatorname{Spec}\nolimits}
\def\End{\operatorname{End}\nolimits}
\def\max{{\operatorname{max}\nolimits}}
\def\min{{\operatorname{min}\nolimits}}
\def\supp{\operatorname{supp}\nolimits}
\def\wt{\operatorname{wt}\nolimits}

\def\top{{\operatorname{top}\nolimits}}

\def\op{{{\operatorname{{op}}\nolimits}}}

\def\Ext{\operatorname{Ext}\nolimits}

\def\Ker{\operatorname{Ker}\nolimits}

\def\Im{\operatorname{Im}\nolimits}

\def\Irr{\operatorname{Irr}\nolimits}

\def\pro{{\lim\limits_{\longleftarrow}}}
\def\ind{{\lim\limits_{\longrightarrow}}}

\def\bfmod{\operatorname{\!-\mathbf{mod}}\nolimits}
\def\bfgmod{\operatorname{\!-\mathbf{gmod}}\nolimits}
\def\bfMod{\operatorname{\!-\mathbf{Mod}}\nolimits}

\author{P. Shan, M. Varagnolo, E. Vasserot}
\email{peng.shan@unicaen.fr}
\email{michela.varagnolo@u-cergy.fr}
\email{vasserot@math.jussieu.fr}

\title
[Koszul duality of affine Kac-Moody algebras]
{Koszul duality of affine Kac-Moody algebras
and cyclotomic rational double affine Hecke algebras}

\begin{document}
\begin{abstract}
We give a proof of the parabolic/singular Koszul duality for the
category O of affine Kac-Moody algebras. The main new tool is a relation 
between moment graphs and finite codimensional affine Schubert varieties.
We apply this duality to $q$-Schur algebras and to cyclotomic rational
double affine Hecke algebras. This yields a proof of a conjecture of Chuang-Miyachi
relating the level-rank duality with the Ringel-Koszul duality of cyclotomic rational
double affine Hecke algebras.
\end{abstract}
\thanks{This research was partially supported by the ANR grant number ANR-10-BLAN-0110}

\maketitle

\setcounter{tocdepth}{2}

\tableofcontents

\section{Introduction}

The purpose of this paper is 
to give a proof of the parabolic/singular Koszul duality for the
category O of affine Kac-Moody algebras. 
The main motivation for this is the conjecture in \cite{VV} (proved in \cite{RSVV}) relating
the parabolic affine category O and the category O of cyclotomic rational
double affine Hecke algebras (CRDAHA for short).
Using the present work,
we deduce from this conjecture the main conjecture of \cite{CM} which claims that the category O of 
CRDAHA's is Koszul and that the Koszul equivalence is related to the level-rank duality on the Fock space.

There are several possible approaches to Koszul duality for affine Kac-Moody
algebras. In \cite{BY}, a geometric analogue of the composition
of the Koszul and the Ringel duality is given, which involves Whittaker sheaves
on the affine flag variety. Our principal motivation  comes from representation theory of CRDAHA's.
For this, we need to prove a Koszul duality for the category O 
itself rather than for its geometric analogues.

One difficulty of the Kac-Moody case comes from the fact that,
at a positive level, the category O has no tilting modules, while at a negative
level it has no projective modules. One way to overcome this is 
to use a different
category of modules than the usual category O, as the Whittaker 
category in loc.~cit. or a category of linear complexes as in
\cite{MOS}. To remedy this, we use a truncated version of the (affine parabolic) category O.
Under truncation any singular block of an affine parabolic category O at a non-critical level
yields a finite highest-weight category which contains both
tilting and projective objects. We prove that these highest weight categories 
are Koszul and are Koszul dual to each other.

Note that the affine category O is related to two different
types of geometry.
In negative level it is related to the affine flag ind-scheme and to 
finite dimensional affine Schubert varieties.
In positive level it is related to Kashiwara's affine flag manifold 
and to finite codimensional affine Schubert varieties.
In negative level, a localization theorem (from the category O to perverse sheaves on the affine flag ind-scheme) has been worked out by
Beilinson-Drinfeld and Frenkel-Gaitsgory in \cite{BD}, \cite{FG2}.
A difficulty in the proof of the main theorem comes from the absence of a localization theorem (from the category O to perverse sheaves on
Kashiwara's affine flag manifold)
at the positive level. To overcome this we use standard Koszul duality and Ringel duality to relate the positive and negative level.

Our general argument is similar to the one in \cite{B}, \cite{BGS}. 
For this, we need an affine analogue of the Soergel functor on the category O.
We use the functor introduced by Fiebig in \cite{F2}.
To define it, we must introduce the deformed category O, which is a highest weight category
over a localization of a polynomial ring, and some category of sheaves over a moment graph.

By the work of Fiebig, sheaves over moment graphs give an algebraic analogue of
equivariant perverse sheaves associated with finite dimensional affine Schubert varieties.
An important new tool in our work is a relation 
between sheaves over some moment graph and equivariant perverse sheaves
associated with finite codimensional affine Schubert varieties, see Appendix \ref{app:B}.
This relation is of independent interest.

Let us now explain the structure of the paper.
Section 2 contains generalities on highest weight categories and standard Koszul duality.
In Section \ref{sec:3} we introduce the affine parabolic category O and its truncated version.
Section \ref{sec:momentgraph} contains some genreralities on moment graphs and the relation with the deformed affine category O. 
The section \ref{sec:5} is technical and contains the proof of the main theorem.
Next, we apply the Koszul duality to CRDAHA's and $q$-Schur algebras in Section \ref{app:A}.

The Kazhdan-Lusztig equivalence \cite{KL} implies
that the module category of the $q$-Schur algebra is equivalent to
a highest weight subcategory of the affine category O of $GL_n$ 
at a negative level.
Thus, our result implies that the $q$-Schur algebra is Morita equivalent to a Koszul algebra
(and also to a standard Koszul algebra),
see Section 
\ref{sec:schurkoszul}
\footnote{After our paper was written, we received a copy of \cite{CM}
where a similar result is obtained by different methods}.
To our knowledge, this was not proved so far.
There are different possible 
approachs for proving that the $q$-Schur algebra is Koszul.
Some are completely algebraic, see e.g., \cite{PS}.
Some use analogues of the Bezrukavnikov-Mirkovic 
modular localization theorem, see e.g., \cite{Rc}.
Our approach has the advantage that it yields
an explicit description of the Koszul dual of the $q$-Schur algebra.

Finally, we apply the Koszul duality of the category O to CRDAHA's.
More precisely, in \cite{VV} some higher analogue
of the $q$-Schur algebra has been introduced. 
It is a highest-weight subcategory 
of the affine parabolic category O. 
Since the category is standard Koszul, these higher $q$-Schur algebras are also Koszul
and their are Koszul dual to each other.
Next, it was conjectured in loc.~cit. (and proved in \cite{RSVV}) that
these higher $q$-Schur algebras are
equivalent to the category O of the CRDAHA.
Thus, our result also implies that the CRDAHA's are Koszul. Using this,
we prove the level-rank conjecture
for CRDAHA's in \cite{CM}.

\vspace{1cm}

\section{Preliminaries on Koszul rings and highest weight categories}

\subsection{Categories}

For an object $M$ of a category $\bfC$
let ${\bf 1}_M$ be the identity endomorphism of $M$.
Let $\bfC^{\op}$ be the category opposite to $\bfC$.

A functor of additive categories is always assumed to be additive.
If $\bfC$ is an exact category then $\bfC^\op$ is 
equipped with the exact structure such that
$0\to M'\to M\to M''\to 0$ is exact in $\bfC^\op$ if and only if 
$0\to M''\to M\to M'\to 0$ is exact in $\bfC$.
An {\it exact functor} of exact categories is a 
functor which takes short exact sequences
to short exact sequences.
Unless specified otherwise, a functor will always be a covariant functor.
A contravariant functor $F:\bfC\to\bfC'$ is exact if and only if the functor
$F:\bfC^\op\to\bfC'$ is exact.

Given an abelian category $\bfC$,
let $\Irr(\bfC)$ be the set of isomorphism classes of simple objects.
Let $\Proj(\bfC)$ and $\Inj(\bfC)$ be the sets of 
isomorphism classes of indecomposable projective, injective objects respectively.
For an object $M$ of  $\bfC$ we abbreviate 
$\Ext_{\bfC}(M)=\Ext_{\bfC}(M,M),$
where $\Ext_\bfC$ stands for the direct sum of all $\Ext^i_\bfC$'s.

Let $R$ be a commutative, noetherian, integral domain.
An {\it $R$-category} is an additive
category enriched over the tensor category of $R$-modules.
Unless mentioned otherwise, a functor of 
$R$-categories is always assumed to be $R$-linear.
A {\it hom-finite $R$-category} is an $R$-category 
whose Hom spaces are finitely
generated over $R$. 

An additive category  is {\it Krull-Schmidt} if any object  has a decomposition
such that each summand is indecomposable with local
endomorphism ring.
A full additive subcategory of a Krull-Schmidt category  is again Krull-Schmidt
if every idempotent splits (i.e., if it is closed under direct summands).
A hom-finite $k$-category over a field $k$ is Krull-Schmidt if 
every idempotent splits (e.g., if it is abelian).

%If the category $\bfC$ is Krull-Schmidt, then taking the projective cover yields a bijection
%$\Irr(\bfC)\to\Proj(\bfC)$. If $\bfC=\bfmod(R)$, with $R$ a ring, let $1_x\in R$ be the primitive idemp

We call {\it finite abelian $k$-category} a $k$-category which is 
equivalent to the category of finite dimensional modules over a finite 
dimensional $k$-algebra.
 
For any abelian category $\bfC$, let $\bfD^b(\bfC)$ be the corresponding
bounded derived category.

\vspace{2mm}

\subsection{Graded rings}\label{sec:gradedrings}
For a ring $A$ let $A^\op$ be the opposite ring. Let
$A\bfMod$ be the category of left $A$-modules and let
$A\bfmod$ be the subcategory of the finitely generated ones.
We abbreviate 
$\Irr(A)=\Irr(A\bfmod)$, $\Proj(A)=\Proj(A\bfmod)$ and $\Inj(A)=\Inj(A\bfmod).$

By a graded ring $\bar A$ we'll always mean a $\bbZ$-graded ring.
Let $\bar A\bfgmod$ be 
the category of finitely generated graded $\bar A$-modules.
We abbreviate $\bfD^b(\bar A)=\bfD^b(\bar A\bfgmod),$
$\Irr(\bar A)=\Irr(\bar A\bfgmod),$ $\Inj(\bar A)=\Inj(\bar A\bfgmod)$ and
$\Proj(\bar A)=\Proj(\bar A\bfgmod).$
Given a graded $\bar A$-module $M$ and an integer $j$, 
let $M\langle j\rangle$ be  the graded $\bar A$-module obtained from $\bar M$
by shifting the grading by $j$, i.e., such that $(M\langle j\rangle)^i=M^{i-j}$.
Given $M,N\in\bar A\bfgmod$ let $\hom_{\bar A}(M,N)$ and 
$\ext_{\bar A}(M,N)$ be the morphisms and extensions in the category of graded modules.

We say that the ring $\bar A$ is {\it positively graded} if $\bar A^{<0}=0$ and if
$\bar A^{0}$ is a semisimple ring.
A finite dimensional graded algebra $\bar A$ over a field
$k$ is positively graded if $\bar A^{<0}=0$ and $\bar A^0$ is 
semisimple as an $\bar A$-module.
Here $\bar A^0$ is identified with $\bar A/\bar A^{>0}$.

A graded module $M$ is called \emph{pure of weight
$i$} if it is concentrated in degree $-i$, i.e., $M = M^{-i}$.
Suppose $\bar A$ is a positively graded ring. Then any
simple graded module is pure, and any pure graded module is semisimple.

Assume that $k$ is a field and that
$\bar A$ is a positively graded $k$-algebra.
We say that $\bar A$ is {\it basic} if 
$\bar A^0$ isomorphic to a finite product of copies of $k$
as a $k$-algebra.

Assume that $\bar A$ is finite dimensional.
Let $\{1_x\,;\,x\in\Irr(\bar A^0)\}$ be a complete system of 
primitive orthogonal idempotents of $\bar A^0$.
The {\it Hilbert polynomial} of $\bar A$ is the 
matrix $P(\bar A,t)$ with entries $\bbN[t]$ given by
$P(\bar A,t)_{x,x'}=\sum_it^i\dim\bigl(1_x\bar A^i1_{x'}\bigr)$
for each $x,x'$.
Assume further that $\bar A$ is positively graded.
We have canonical bijections
$\Irr(A)=\Irr(\bar A^0)=\{\bar A^01_x\}$
such that $x$ is the isomorphism class of $\bar A^01_x$.
Since $A\bfmod$ is Krull-Schmidt,
there is a canonical bijection
$\top:\Proj(A)\to\Irr(A)$, $P\mapsto\top(P)$.
We have 
$\top(A1_x)=\bar A^01_x=x.$
The set $\{1_x\,;\,x\in\Irr(\bar A^0)\}$
is a complete system of 
primitive orthogonal idempotents of $A$.

Given a graded commutative, noetherian, integral domain $R$,
a {\it graded $R$-category} is an additive category 
enriched over the monoidal category of graded $R$-modules.

\vspace{2mm}

\subsection{Koszul duality}
\label{sec:koszul}
Let $k$ be a field and
$\bar A$ be a  graded $k$-algebra.
Let $A$ be the (non graded) $k$-algebra underlying $\bar A$.

The {\it Koszul dual} of $\bar A$ is the graded $k$-algebra 
$E(\bar A)=\Ext_{A}(\bar A^0).$
Forgetting the grading of $E(\bar A)$, we get a $k$-algebra $E(A)$.
It is finite dimensional if $A$ is finite dimensional and has finite global dimension.

A \emph{linear projective resolution} of a graded $\bar A$-module $M$ is a 
graded projective resolution $\dots\to P_1\to P_0\to M\to 0$ such that, for each $i$, the projective graded 
$\bar A$-module $P_i$ is finitely generated by a set of homogeneous elements of degree $i$.

Assume that $\bar A$ is positively graded (in the sense of Section \ref{sec:gradedrings}).
We say that $\bar A$ is \emph{Koszul}
if each simple graded module which is pure of weight 0 has a linear
projective resolution.
If $\bar A$ is Koszul, then we say that $A$ 
{\it has a Koszul grading}.
If $\bar A$ is finite dimensional then this grading 
is unique up to isomorphism of graded $k$-algebras, 
see \cite[cor.~2.5.2]{BGS}.

Assume that $\bar A$ is finite dimensional, has finite global dimension and is  Koszul.
Then $E(\bar A)$ is Koszul and $E^2(\bar A)=\bar A$ canonically, see \cite[thm.~1.2.5]{BGS}.
Put $\bar A^!=E(\bar A)^\op$ and $A^!=E(A)^\op$. Note that $\bar A^!$ is also Koszul by \cite[prop.~2.2.1]{BGS}.

For each $x\in\Irr(A)=\Irr(\bar A^0)$ the idempotent $1_x\in\bar A^0$ yields an idempotent
$1_x^!=E(1_x)\in E(A)$ in the obvious way.
The set $\{1_x^!\,;\,x\in\Irr(A)\}$
is a complete system of 
primitive orthogonal idempotents of $E(A)$.

By \cite[thm.~2.12.5, 2.12.6]{BGS}, there is an equivalence of triangulated categories
$E:\bfD^b(\bar A)\to \bfD^b(\bar A^!)$
such that $E(M\langle i\rangle)=E(M)[-i]\langle-i\rangle$ and
$E(\bar A^01_x)=\bar A^! 1_x^!$ for each $x\in\Irr(A)$.
We'll call it the \emph{Koszul equivalence}.
By forgetting the grading, we get a
bijection $\Irr(A)\to\Proj(A^!)$, $x\mapsto A^!1_x^!$.
It induces a bijection $(\bullet)^!=\top\circ E:\Irr(A)\to\Irr(A^!)$, which will be called the {\it natural bijection}
between $\Irr(A)$ and $\Irr(A^!).$

Let $\bfC$ be a finite abelian $k$-category.
We say that {\it $\bfC$ has a Koszul grading} if 
there is a projective generator $P$ such that
the ring $A=\End_\bfC(P)^{\op}$ has a Koszul grading $\bar A$. 
We may also simply say that \textit{$\bfC$ is Koszul} and
we abbreviate $\bfC^!=A^!\bfmod.$

The following lemmas are well-known.

\vspace{2mm}

\begin{lemma}
\label{lem:1.1}
Let $P$, $A$, $\bfC$ be as above.
If $\bar A$ is a positively graded 
then  
$E(\bar A)=\Ext_{\bfC}(L)$, where $L=\top(P)$.
\end{lemma}

\vspace{.5mm}

\begin{proof}
The equivalence 
$\bfC\to A\bfmod$, $X\mapsto\Hom_\bfC(P,X)$
takes $L$ to $\bar A/\bar A^{>0}$.
Therefore
$E(\bar A)=
\Ext_{\bfC}(L).$
\end{proof}

\vspace{2mm}

\begin{lemma}
\label{lem:1.2}
Let $\bfC$ be a finite abelian $k$-category with
a Koszul grading. If
$\bfD$ is a Serre subcategory and the inclusion
$\bfD\subset\bfC$ induces injections on extensions, then 
$\bfD$ has also a Koszul grading.
\end{lemma}

\vspace{.5mm}

\begin{proof}
Let $P_L$ be the projective cover of $L\in\Irr(\bfC)$. The set
$\Irr(\bfC)$ is finite and  $P=\bigoplus_LP_L$ is a minimal projective generator.
Set $A=\End_\bfC(P)^{\op}$. 

Let $A_I$ be the quotient of $A$ by the two-sided ideal $I$
generated by $\{{\bf 1}_{P_L}\,;\,L\notin \Irr(\bfD)\}.$
The pull-back by the ring homomorphism $A\to A_I$ identifies
$A_I\bfmod$ with the full subcategory of $A\bfmod$ consisting of the modules killed by $I$.

An object $M\in\bfC$ belongs to $\bfD$ if and only if $\Hom_\bfC(P_L,M)=0$
whenever $L\notin\Irr(\bfD)$.
Therefore, the equivalence 
$\Hom_\bfC(P,\bullet):\bfC\to A\bfmod$
identifies $\bfD$ with the full subcategory of $A\bfmod$ consisting of the modules killed by $I$.
We deduce that $\bfD$ is equivalent to $A_I\bfmod$.

Now, let $\bar A$ be the Koszul grading on $A$.
Since $\bar A$ is positively graded,
the idempotent ${\bf 1}_{P_L}$ has degree 0. 
Thus $I$ is an homogeneous ideal of $\bar A$.
Hence $\bar A$ yields a grading $\bar A_I$ on $A_I$ such that $\bar A_I^{<0}=0$ 
and $\bar A_I^0$ is semi-simple
as an $\bar A_I$-module. Thus $\bar A_I$ is also positively graded.

Fix a graded lift $\bar L$ of $L\in\Irr(A)$, see Section \ref{sec:standardkoszul}.
The graded $\bar A$-module $\bar L$ is  pure.
Let $d_L$ be its degree.
By \cite[prop.~2.1.3]{BGS} we have
$\ext_{\bar A}^i(\bar L,\bar L')=0$ 
unless $i=d_{L'}-d_L$ for
$L,L'\in \Irr(\bfC)$. We must check that
$\ext_{\bar A_I}^i(\bar L,\bar L')=0$ unless $i=d_{L'}-d_L$
if 
$L,L'\in \Irr(\bfD)$.
This is obvious because the inclusion
$\bar A_I\bfgmod\subset\bar A\bfgmod$
induces injections on extensions. 
\end{proof}

\iffalse%%%%%%%%%%%%%%%%%%%%
\vspace{2mm}

\begin{rk} Let us sketch the construction of the equivalence
$E:\bfD^b(\bar R\bfgmod)\to \bfD^b(E(\bar R)\bfgmod),$ following \cite[sec.~2]{BGS}.
See also \cite[sec.~3]{RH}. 

We can view $\bar R$ as a 
dg-algebra with 0 differential concentrated in the place 0 of the complex.
Fix a linear projective resolution $K_\bullet$ of $\bar R^0$ as a graded $\bar R$-module. 
It can be viewed as a  dg-module over $\bar R$, which is K-projective in the sense of
Spaltenstein. 

The endomorphism ring $\End_{R}(K_\bullet)$ has a natural structure of 
dg-algebra such that $K_\bullet$ is a dg-module.
Thus, we have the functor $M\to \RHom_{R}(K_\bullet,M)$ from the bounded derived category
of dg-modules over $\bar R$ to the bounded derived category
of dg-modules over $\End_{R}(K_\bullet).$

Now, taking into account the gradings on $\bar R$ and $K_\bullet$, this functor yields
a functor $E'$ from the bounded derived category
of finitely generated, graded dg-modules over $\bar R$ to the bounded derived category
of finitely generated, graded dg-modules over $\End_{R}(K_\bullet).$

We view $E(\bar R)$ as a dg-algebra with 0 differential such that $\Ext^i_R(\bar R^0)$ 
is at the $i$-th place in the complex. 
Since $\bar R$ is Koszul, the dg-algebras $\End_{R}(K_\bullet)$ and $E(\bar R)$
are quasi-isomorphic. Thus their bounded derived categories are equivalent.
The functor $E$ is induced by the functor $E'$.

\end{rk}
\fi%%%%%%%%%%%%%%%%%%%%%%%%%%%

\vspace{2mm}

\subsection{Highest weight categories}
\label{sec:5}
Let $R$ be a commutative, noetherian ring with 1
which is a local ring with residue field $k$.
%Let $K$ be its fraction field.

Let $\bfC$ be an $R$-category which is equivalent to the
category of finitely generated modules over a finite projective $R$-algebra $A$.

The category $\bfC$ is a \textit{highest weight $R$-category} 
if it is equipped with a poset of isomorphisms of objects $(\Delta(\bfC),\leqslant)$ called the standard objects satisfying the following conditions:
\begin{itemize}
\item the objects of $\Delta(\bfC)$ are projective over $R$
\item given $M\in\bfC$ such that $\Hom_{\bfC}(D,M)=0$ for all $D\in\Delta(\bfC)$,
we have $M=0$
\item given $D\in\Delta(\bfC)$, there is a projective object $P\in\bfC$ and a surjection 
$f:P\twoheadrightarrow D$ such that $\ker f$ has a (finite) filtration whose successive
quotients are objects $D'\in\Delta$ with $D'>D$.
\item given $D\in\Delta$, we have $\End_{\bfC}(D)=R$
\item given $D_1,D_2\in\Delta$ with $\Hom_{\bfC}(D_1,D_2)\not=0$, we have $D_1\le D_2$.
\end{itemize}
See \cite[def.~4.11]{Ro}. Note that since $R$ is local, any finitely generated projective $R$-module is free, 
hence the set $\tilde\Delta$ in loc.~cit.~is the set of finite direct sums of objects in $\Delta$.
The partial order $\leqslant$ is called the \emph{highest weight order} on $\bfC$. We write $\Delta(\bfC)=\{\Delta(\lambda)\}_{\lambda\in\Lambda}$ for $\Lambda$ an indexing poset. 

\smallskip

\begin{lemma}\label{lem:hwtbasic}
Let $\bfC$ be a highest weight $R$-category. Given $\lambda\in\Lambda$, there
is a unique (up to isomorphism) indecomposable projective (resp. injective, tilting, costandard) object associate with $\lambda$, denoted by 
 $P(\lambda)$ (rep. $I(\lambda)$,
 $T(\lambda)$, $\nabla(\lambda)$) such that
\begin{itemize}
\item[{\small($\nabla$)}] $\Hom_{\bfC}(\Delta(\mu),\nabla(\lambda))\simeq \delta_{\lambda\mu}R$
and $\Ext^1_{\bfC}(\Delta(\mu),\nabla(\lambda))=0$ for all $\mu\in\Lambda$,
\item[{\small($P$)}] there is a surjection $f:P(\lambda)\twoheadrightarrow\Delta(\lambda)$ 
such that $\ker f$ has a filtration whose successive
quotients are $\Delta(\mu)$'s with $\mu>\lambda$,
\item[{\small($I$)}] there is an injection $f:\nabla(\lambda)\hookrightarrow I(\lambda)$
such that $\mathrm{coker} f$ has a filtration whose successive
quotients are $\nabla(\mu)$'s with $\mu>\lambda$,
\item[{\small($T$)}] there is an injection $f:\Delta(\lambda)\hookrightarrow T(\lambda)$
and a surjection $g:T(\lambda)\twoheadrightarrow \nabla(\lambda)$ such that
$\mathrm{coker} f$ (resp. $\ker g$) has a 
filtration whose successive
quotients are $\Delta(\mu)$'s (resp. $\nabla(\mu)$'s) with $\mu<\lambda$.
\end{itemize}
\end{lemma}
See e.g. \cite[prop~.2.1]{RSVV}. The objects
$\nabla(\lambda)$, $\Delta(\lambda)$, $P(\lambda)$, $I(\lambda)$
and $T(\lambda)$ are projective over $R$.
We have $\Proj(\bfC)=\{P(\lambda)\}_{\lambda\in\Lambda}$, $\Inj(\bfC)=\{I(\lambda)\}_{\lambda\in\Lambda}$. The set $\Tilt(\bfC)=\{T(\lambda)\}_{\lambda\in\Lambda}$ is the set of isomorphism classes of indecomposable tilting
objects in $\bfC$. Let $\nabla(\bfC)=\{\nabla(\lambda)\}_{\lambda\in\Lambda}$. Note that $\Delta(\lambda)$ has a unique simple quotient $L(\lambda)$. The set of isomorphism classes of simple objects in $\bfC$ is given by $\Irr(\bfC)=\{L(\lambda)\}_{\lambda\in\Lambda}$.

Let $\bfC^\Delta$, $\bfC^\nabla$ be the full 
subcategories of $\bfC$ consisting of the 
{\it $\Delta$-filtered} and {\it $\nabla$-filtered} objects, i.e., 
the objects having a finite
filtration whose successive  quotients are standard, costandard respectively.
These categories are exact. Recall that a tilting object is by definition an object that is both $\Delta$-filtered and $\nabla$-filtered.

The opposite of $\bfC$ 
is a highest weight $R$-category such that
$\Delta(\bfC^\op)=\nabla(\bfC)$  with the opposite highest weight order.

Given a commutative local $R$-algebra $S$ and an $R$-module $M$, we write
$SM=M\otimes_RS$. Let $S\bfC=SA\bfmod.$ We have the following, see e.g. \cite[prop.~2.1, 2.4, 2.5]{RSVV}.

%\smallskip
%\noindent
%$\bullet\ $
%$\bfC^*$ is a highest weight $R$-category on the poset
%$\Lambda$ with standard objects $\Delta^*(\lambda)=\nabla(\lambda)^*$ and
%with $P^*(\lambda)=I(\lambda)^*$, $I^*(\lambda)=P(\lambda)^*$,
%$\nabla^*(\lambda)=\Delta(\lambda)^*$ and $T^*(\lambda)=T(\lambda)^*$.

%The {\it reduction to $k$} is the functor 
%$\bfC\to k\bfC,$ $M\mapsto kM.$ 
%The category $k\bfC$ is a highest weight $k$-category by \cite[thm.~4.15]{Ro}.
%The reduction to $k$ yields bijections
%$\Delta(\bfC)\to\Delta(k\bfC)$ and $\nabla(\bfC)\to\nabla(k\bfC).$
%It gives  
%also a bijection $\Irr(\bfC)\to\Irr(k\bfC)$ by the Nakayama lemma.
%So, the canonical bijections
%$\Delta(k\bfC)=\nabla(k\bfC)=\Irr(k\bfC)$
%yield canonical bijections
%$\Delta(\bfC)=\nabla(\bfC)=\Irr(\bfC).$

\vspace{2mm}

\begin{prop} \label{prop:2.3} 
Let $\bfC$ be a highest weight $R$-category, and let $S$ be a commutative local $R$-algebra
with 1.
For any $M,N\in\bfC$ the following holds :

(a) if $S$ is $R$-flat then $S\Ext^d_\bfC(M,N)=\Ext^d_{S\bfC}(SM,SN)$ for all $d\geqslant 0,$

(b) if either $M\in\bfC$ is projective or ($M\in\bfC^\Delta$ and $N\in\bfC^\nabla$), then we have
$S\Hom_\bfC(M,N)=\Hom_{S\bfC}(SM,SN)$,

(c) if $M$ is $R$-projective then $M$ is projective in $\bfC$ 
(resp.~ $M$ is tilting in $\bfC$, $M\in\bfC^\Delta$)
if and only if $k M$ is projective in $k\bfC$
(resp.~ $k M$ is tilting in $k\bfC,$ $k M\in k\bfC^\Delta$),

(d) if either ($M$ is projective in $\bfC$ and $N$ is $R$-projective)
or ($M\in\bfC^\Delta$ and $N\in\bfC^\nabla$) 
then $\Hom_\bfC(M,N)$ is $R$-projective,

(e) the category $S\bfC$ is a highest weight 
$S$-category on the poset $\Lambda$ with standard objects
$S\Delta(\lambda)$
and costandard objects $S\nabla(\lambda)$. The projective, injective and tilting objects associated with $\lambda$ are $SP(\lambda)$, $SI(\lambda)$ and $ST(\lambda)$
\qed
\end{prop}

\vspace{2mm}
\iffalse%%%%%%%%%%%%%%%%%%%%%%%%%%%%%%%%%%%%%%%%%%%%%
In particular, 
the reduction gives   bijections
$\Proj(\bfC)\to\Proj(k\bfC)$ and $\Tilt(\bfC)\to\Tilt(k\bfC).$
Hence, there are bijections
\begin{equation*}
%\label{bijection}
\Proj(\bfC)\to\Delta(\bfC)\to\Tilt(\bfC)\to\nabla(\bfC),\quad
P\mapsto\Delta\mapsto T\mapsto\nabla
\end{equation*}
such that there is a surjection $P\to\Delta$ whose kernel is filtered by standard modules 
which are $>\Delta$.
Further, there is an injection $\Delta\to T$ whose cokernel is filtered by standard modules which are $<\Delta$
and there is a surjection $T\to\nabla$ whose kernel is filtered by costandard modules 
which are $<\nabla$.
Hence,  any of the sets
$\Tilt(\bfC)$, $\Proj(\bfC)$, $\Delta(\bfC)$, $\nabla(\bfC)$, $\Irr(\bfC)$ 
can be regarded as a poset for the highest weight order.

\vspace{2mm}
\fi%%%%%%%%%%%%%%%%%%%%%%%%%%%%%%%%%%%%%%%%%%%%%%%%%

\begin{rk}
\label{rk:2.3}
For any subset $\Sigma\subset\Lambda$,
let $\bfC[\Sigma]$ be the Serre subcategory generated by all the $L(\lambda)$ with $\lambda\in\Sigma$ and 
let $\bfC(\Sigma)$ be the Serre quotient
$\bfC/\bfC[\Irr(\bfC)\setminus \Sigma]$.

An {\it ideal} in the poset $\Lambda$
is a subset of the form $I=\bigcup_{i\in I}\{\leqslant i\}$.

A {\it coideal}  is the complement of an ideal, i.e., a subset of the form
$J=\bigcup_{j\in J}\{\geqslant j\}$.

Now, assume that $\bfC$ is a highest weight category over a field $k$, 
and that $I$, $J$ are respectively an ideal and a coideal of $\Lambda$.
Then $\bfC[I]$, $\bfC(J)$ are highest weight categories and
the inclusion $\bfC[I]\subset\bfC$ induces injections on extensions by
\cite[thm.~3.9]{CPS2}, \cite[prop.~A.3.3]{Do}.
\end{rk}

\vspace{2mm}

\subsection{Ringel duality}
\label{sec:R}
Let $R$ be a commutative, noetherian ring with 1
which is a local ring with residue field $k$.

Let $\bfC$ be a highest-weight $R$-category 
which is equivalent to $A\bfmod$ for a finite projective $R$-algebra $A$. 

We call $T=\bigoplus_{\lambda\in\Lambda}T(\lambda)$ the \emph{characteristic tilting module}. 
Set $D(A)=\End_\bfC(T)$ and $A^\diamond=D(A)^\op.$
The {\it Ringel dual} of
$A$ is the $R$-algebra $A^\diamond,$
the  Ringel dual of $\bfC$ is the category $\bfC^\diamond=A^\diamond\bfmod.$

%with the set of standard objects
%$\Delta(D(\bfC))=\{\Hom_\bfC(\Delta,T)\,;\,\Delta\in\Delta(\bfC)\}.$
%The order on $\Delta(D(\bfC))$ is the opposite of the order on $\Delta(\bfC),$
%Note that $D(\bfC)$ is independent of the choice of $R$ or $T$, 
%up to equivalence of highest weight categories
%see \cite{Ro}.

The category $\bfC^\diamond$ is a highest-weight $R$-category on the poset $\Lambda^\op$.
We have an equivalence of triangulated categories
$(\bullet)^\diamond:\ \bfD^b(\bfC)\to \bfD^b(\bfC^\diamond)$ 
called the {\it Ringel equivalence}. 
It restricts to an equivalence of exact categories
$(\bullet)^\diamond:\ \bfC^\Delta\to (\bfC^\diamond)^\nabla$ such that $M\mapsto\RHom_\bfC(M,T)^*$. Here $(\bullet)^*$ is the dual as a $k$-vector space.
We have $\Delta(\lambda)^\diamond=\nabla^\diamond(\lambda)$, $P(\lambda)^\diamond=T^\diamond(\lambda)$ and 
$T(\lambda)^\diamond=I^\diamond(\lambda)$,
 for each $\lambda\in\Lambda$, see \cite[prop.~4.26]{Ro}. 

The ring $(A^\diamond)^\diamond$ is Morita equivalent to $A$,
see loc.~cit. and \cite[sec.~A.4]{Do}.

Now, assume that $R=k$.
For each primitive idempotent $e\in A,$ there is a unique $\lambda\in\Lambda$ 
such that $Ae=P(\lambda)$. We define $e^\diamond\in A^\diamond$ to be the primitive idempotent
such that $A^\diamond e^\diamond=P^\diamond(\lambda)$.
The bijection $(\bullet)^\diamond:\Irr(\bfC)\to\Irr(\bfC^\diamond)$ given by $L(\lambda)\mapsto L^\diamond(\lambda)$ is called the {\it natural bijection}
between $\Irr(\bfC)$ and $\Irr(\bfC^\diamond).$

The following is well-known, see, e.g., \cite[prop.~A.4.9]{Do}.

\begin{lemma}\label{lem:ringeltroncation}
A subset $I\subset \Lambda$ is an ideal if and only if it is a coideal in $\Lambda^\op$. We have 
$\bfC[I]^\diamond=\bfC^\diamond(I),$
and the Ringel equivalence
factors to an equivalence of categories
$\bfC[I]^\Delta\to \bfC^\diamond(I)^\nabla$.
\end{lemma}

\vspace{2mm}

\subsection{Standard Koszul duality}
\label{sec:standardkoszul}

Let $k$ be a field and $\bfC$ be a highest weight $k$-category.
Assume that $\bfC$ is equivalent to the
category of finitely generated left modules over a finite dimensional $k$-algebra $A$.

Let $\bar A$ be a graded $k$-algebra which is
isomorphic to $A$ as an $k$-algebra. 
We call $\bar A$ a {\it graded lift} of $A$. 
A {\it graded lift} of an object $M\in\bfC$ is a graded
$\bar A$-module which is isomorphic to $M$ as an $A$-module. 

Assume that $\bar A$ is positively graded (in the sense of Section \ref{sec:gradedrings}). We have the following
\vspace{2mm}

\begin{prop}
Given $\lambda\in\Lambda$ there exists unique graded lifts $\bar L(\lambda)$, $\bar P(\lambda)$, $\bar I(\lambda)$, $\bar \Delta(\lambda)$, $\bar \nabla(\lambda)$, $\bar T(\lambda)$ such that 
\begin{itemize}
\item $\bar L(\lambda)$ is pure of degree zero,
\item the surjection $\bar P(\lambda)\twoheadrightarrow\bar L(\lambda)$ is homogeneous of degree zero,
\item the injection $\bar L(\lambda)\hookrightarrow\bar I(\lambda)$ is homogeneous of degree zero,
\item the surjection $f:\bar P(\lambda)\twoheadrightarrow\bar \Delta(\lambda)$ in $(P)$ is homogeneous of degree zero,
\item the injection $f:\bar\nabla(\lambda)\hookrightarrow\bar I(\lambda)$ in $(I)$ is homogeneous of degree zero,
\item both the injection $f:\bar\Delta(\lambda)\hookrightarrow \bar T(\lambda)$ and the surjection $g:\bar T(\lambda)\twoheadrightarrow\bar\nabla(\lambda)$ in $(T)$ are homogeneous of degree zero.
\end{itemize}

\end{prop}

\vspace{.5mm}

\begin{proof}
The existence of the graded lifts is proved in \cite[cor.~4,5]{MO}. They are unique up to isomorphisms
because, by \cite[lem.~2.5.3]{BGS}, the graded lift of an indecomposable 
object of $k\bfC$ is unique up to a graded 
$\bar A$-module isomorphism and up to a shift of the grading.
\end{proof}

\vspace{2mm}

\iffalse%%%%%%%%%%%%%%%%%%%%%%%%
\begin{lemma}
Let $M$ be an indecomposable object of $\bfC$ which is projective as a $A$-module.
Assume that $kM$ is indecomposable in $k\bfC$.
If $M$ has a graded lift then this lift is unique up to a graded 
$\bar R$-module isomorphism and up to a shift of the grading.
\end{lemma}
\begin{proof}
If $A=k$ a proof is given in \cite[lem.~2.5.3]{BGS}.
For the general case, using the same argument as in loc.~cit., 
we are reduced to check that $\End_\bfC(M)$ is a local ring.
Since $M$ is finite and projective over $A$, an element $x\in\End_A(M)$ is invertible
if and only if is reduction to $k$ is invertible in $\End_{k\bfC}(kM)$ by the
Nakayama lemma.
$x\in\End_\bfC(M)$ 
\end{proof}
\fi%%%%%%%%%%%%%%%%%%%%%%%%%%

The gradings above will be called the {\it natural gradings}.
In particular, let $\bar T=\bigoplus_{\lambda\in\Lambda}\bar T(\lambda)$. 
The grading on $\bar T$ induces a grading on the $k$-algebra $A^\diamond$ given by
$\bar A^\diamond=\End_\bfC(\bar T)^\op$ . It is called the
\emph{natural grading}.

A chain complex of projective (resp. ~injective, tilting) modules $\dots \to M_i\to M_{i-1}\to\dots$
is called \emph{linear}
provided that for every $i\in\bbZ$ all indecomposable direct summands of $M_i\langle -i\rangle$ have the natural grading.

Following \cite{ADL}, we say that $\bar A$ is {\it standard Koszul} provided that
all standard modules have linear projective resolutions and all costandard modules have linear injective 
coresolutions.
By \cite[thm.~1]{ADL}, a standard Koszul graded algebra is Koszul.

We will identify $\Lambda=\{L(\lambda)\}_{\lambda\in\Lambda}=\{1_x\,;\,x\in \Irr(A)\}$ and equip the latter with the partial order $\leqslant$.
Assume that $\bar A$ is Koszul. By \cite[thm.~3]{ADL}, 
the graded $k$-algebra $\bar A$ is standard Koszul
if and only if
$\bar A^!$ is quasi-hereditary relatively to the poset
$\{1_x^!\,;\,x\in \Irr(A)\}$ such that 
$1_x^!\leqslant 1_y^!\iff 1_x\geqslant 1_y.$

Following \cite{Ma}, we say that
$\bar A$ is {\it balanced} provided that
all standard modules have linear tilting coresolutions and all costandard modules have linear tilting
resolutions.
By \cite[thm.~7]{MO}, the graded $k$-algebra
$\bar A$ is balanced
if and only if it is standard Koszul and if the graded
$k$-algebra $D(\bar A)$ is positive.

If $\bar A$ is balanced, then the following hold \cite[thm.~1]{Ma}
\begin{itemize}
\item $\bar A$, $\bar A^\diamond$, $\bar A^!$ and $(\bar A^!)^\diamond$ 
are positively graded, quasi-hereditary, Koszul, standard Koszul and balanced,
\item $(\bar A^!)^\diamond=(\bar A^\diamond)^!$ as graded quasi-hereditary $k$-algebras,
\item the natural bijection $\Irr(A)\to\Irr(A^!)$ takes the 
highest weight order on $\Irr(A)$ to the opposite of the highest weight order
on $\Irr(A^!)$.
\end{itemize}
Note that the notation $E(\bar A)$ in \cite{Ma} corresponds to our notation $\bar A^!$.

%Let $\bfC$ be a highest weight category over $k$.
We'll say that \textit{$\bfC$ is standard Koszul or balanced}
if we can choose the graded $k$-algebra $\bar A$ in such a way that it is 
standard Koszul or balanced respectively.

\vspace{2mm}

\begin{rk} \label{rk:KS}
Assume that the graded $k$-algebra $\bar A$ is standard Koszul.
In particular, it is finite dimensional, Koszul and with finite global dimension.
The Koszul equivalence
is given by $E=\RHom_A(\bar A^0,\bullet)$, up to the grading. 
See \cite[sec.~2]{BGS}, \cite[sec.~3]{RH} for details. 
It takes costandard (resp. injective, simple) modules to standard (resp. simple, projective) ones,
by \cite[prop.~2.7]{ADL}, \cite[thm.~2.12.5]{BGS}.
\end{rk}

\vspace{2mm}

\begin{rk}
Assume that $\bar A$, $\bar A^\diamond$ are both positively graded (in the sense of Section \ref{sec:gradedrings}).
The functor $\Hom_\bfC(\bullet,\bar T)^*$
takes the natural graded indecomposable tilting objects in $\bar A\bfgmod$
to the natural graded indecomposable projective ones in
$\bar A^\diamond\bfgmod$. It takes also
the natural graded indecomposable injective objects in $\bar A\bfgmod$
to natural graded indecomposable tilting ones in $\bar A^\diamond\bfgmod$.
\end{rk}

\vspace{2mm}

Let $\bfC$ be a highest weight category over a field $k$ and
$I\subset\Irr(\bfC)$ be an ideal. 
Put $J=\Irr(\bfC)\setminus I$. 
Assume that $\bfC$ is standard Koszul.

The category $\bfC[I]$ has a Koszul grading by 
Lemma \ref{lem:1.2} and Remark \ref{rk:2.3}.
The natural bijection $\Irr(\bfC)\to\Irr(\bfC^!)$ is an anti-isomorphism of posets. 
Thus, the images $I^!,$ $J^!$ of $I,$ $J$ are respectively a coideal and an ideal of $\Irr(\bfC^!)$.

\vspace{2mm}

\begin{lemma} \label{lem:2.6:E} 
For each ideal $I\subset\Irr(\bfC)$ we have $\bfC[I]^!=\bfC^!(I^!).$
\end{lemma}

\vspace{.5mm}

\begin{proof}
Set $L_I=\bigoplus_{L\in I}L$.
%By Lemmas \ref{lem:1.1}, \ref{lem:1.2}
%We have $E(\bfC)=\bfmod(\Ext_\bfC(L_{\Irr(\bfC)}))$.
By Lemma \ref{lem:1.1} and Remark \ref{rk:2.3}, we have 
$\bfC^!=\Ext_\bfC(L)\bfmod$ and
$\bfC[I]^!=\Ext_\bfC(L_I)\bfmod.$ 
Let $e\in\End_\bfC(L)$ be the projection from $L$ to $L_I$.
We have $\Ext_\bfC(L_I)=e\,\Ext_\bfC(L)\,e$.

The full subcategory $K_I\subset \bfC^!$ consisting of the modules killed by $e$
is a Serre subcategory and there is an equivalence
$\bfC^!/K_I\to \bfC[I]^!,$ $M\mapsto eM$.
In other words, the restriction with respect to the obvious inclusion
$\Ext_\bfC(L_I)\subset\Ext_\bfC(L)$, yields  a quotient functor
$\bfC^!\to \bfC[I]^!$ whose kernel is $K_I$.

Since $\bfC$ is standard Koszul, its Koszul dual $\bfC^!$ is a highest weight category.
We have $\bfC^!(I^!)=\bfC^!/\bfC^![J^!],$ because $J^!=\Irr(\bfC^!)\setminus I^!$.
Since  $\bfC^![J^!]$ is the Serre subcategory of $\bfC^!$ generated by 
$J^!$, it consists of the $\Ext_\bfC(L)$-modules killed by $e$.
This implies the lemma.
\end{proof}

\vspace{3mm}

\section{Affine Lie algebras and the parabolic category $\bfO$}
\label{sec:3}

\vspace{2mm}

\subsection{Lie algebras}
Let $\frakg$ be a simple Lie $\bbC$-algebra and let $G$ 
be a connected simple group over $\bbC$ with Lie algebra $\frakg$.
Let  $T\subset G$ be a maximal tori and let $\frakt\subset\frakg$   be its Lie algebra.
Let $\frakb\subset\frakg$ 
be a Borel subalgebra containing $\frakt$.

The elements of $\frakt$,
$\frakt^ *$ are called {\it coweights} and {\it weights} respectively. 
Given a root $\alpha\in\frakt^*$ let $\check\alpha\in\frakt$ 
denote the corresponding coroot.
Let $\Pi\subset\frakt^*$ be the set of roots of $\frakg$,
$\Pi^+\subset\Pi$ the set of roots of $\frakb$,
and $\bbZ\Pi$ be the root lattice. 
Let $\rho$ be half the sum of the positive roots. 
Let $\Phi=\{\alpha_i\,;\,i\in I\}$ be the set of simple roots in $\Pi^+$.
Let $W$ be the Weyl group. 
Let $N$ be the dual Coxeter number of $\frakg$.

\vspace{2mm}

\subsection{Affine Lie algebras}
\label{sec:affine}
Let $\bfg$ be the affine Lie algebra associated with $\frakg$.
Recall that $\bfg=\bbC\partial\oplus\widehat{L\frakg},$
where $\widehat{L\frakg}$ is a central extension of 
$L\frakg=\frakg\otimes\bbC[t,t^{-1}]$ and
$\partial=t\partial_t$ is a derivation of 
$\widehat{L\frakg}$ acting trivially on
the canonical central element $\bfone$ of $\widehat{L\frakg}$. 
Consider the Lie subalgebras 
$\bfb=\frakb\oplus\frakg\otimes t\bbC[t]\oplus\bbC\partial\oplus \bbC\bfone$
and
$\bft=\bbC\partial\oplus\frakt\oplus\bbC\bfone.$

Let $\widehat\Pi,$ $\widehat\Pi^+$
be the sets of roots of $\bfg$, $\bfb$ respectively. We'll call an element of $\widehat\Pi$ an
{\it affine root}.  The set of simple roots in $\widehat\Pi^+$ is 
$\widehat\Phi=\{\alpha_i; i\in \{0\}\cup I\}.$

The elements of $\bft$,
$\bft^*$ are called {\it affine coweights} and {\it affine weights} 
respectively. Let $(\bullet:\bullet):\bft^*\times\bft\to\bbC$ 
be the canonical pairing.
Let $\delta$, $\Lambda_0$, $\hat\rho$ be the affine weights given by
$(\delta:\partial)=(\Lambda_0:\bfone)=1,$
$(\Lambda_0:\bbC\partial\oplus\frakt)=(\delta:\frakt\oplus\bbC\bfone)=0$
and
$\hat\rho=\rho+N\Lambda_0.$

An element of $\bft^*/\bbC\delta$ is called
a {\it classical affine weight}.
Let $cl:\bft^*\to\bft^*/\bbC\delta$ denote the obvious projection. 

Let $e$ be an integer $\neq 0$.
We set 
$\bft^*_e=\{\lambda\in\bft^*;(\lambda:\bfone)=-e-N\}.$
We'll use the identification
$\bft^*=\bbC\times\frakt^*\times\bbC$
such that
$\alpha_i\mapsto(0,\alpha_i,0)$ if $i\neq 0$, 
$\Lambda_0\mapsto (0,0,1)$ and
$\delta\mapsto(1,0,0)$. 

Let $\check\alpha\in\bft$ be the affine coroot associated 
with the real affine root $\alpha$.
Let $\langle\bullet :\bullet \rangle$ be the non-degenerate 
symmetric bilinear form
on $\bft^*$ such that 
$(\lambda:\check\alpha_i)=
2\langle\lambda:\alpha_i\rangle/\langle\alpha_i:\alpha_i\rangle$
and
$(\lambda:\bfone)=\langle\lambda:\delta\rangle.$
Using $\langle\bullet :\bullet \rangle$ we identify $\check\alpha$
with an element of $\bft^*$ 
for any real affine root $\alpha$.

Let
$\widehat W=W\ltimes\bbZ\Pi$ be the affine Weyl group and
let $\calS=\{s_i=s_{\alpha_i}\,;\,\alpha_i\in\widehat\Phi\}$ be the set of simple affine reflections.
The group 
$\widehat W$ acts on $\bft^*$. For
$w\in W$, $\tau\in\bbZ\Pi$ we have
$w(\Lambda_0)=\Lambda_0,$
$w(\delta)=\delta,$
$\tau(\delta)=\delta,$
$\tau(\lambda)=\lambda-\langle\tau:\lambda\rangle\delta$
and
$\tau(\Lambda_0)=\tau+\Lambda_0-\langle\tau:\tau\rangle\delta/2$.

The {\it $\bullet$-action} on $\bft^*$ is given by 
$w\bullet\lambda=w(\lambda+\hat\rho)-\hat\rho$.
It factors to a $\widehat W$-action on $\bft^*/\bbC\delta$.
Two (classical)
affine weights $\lambda$, $\mu$ are {\it linked} if they belong to the same
orbit of the $\bullet$-action, and we write $\lambda\sim\mu$.
Let $W_\lambda$ be the stabilizer of an affine weight $\lambda$ under the $\bullet$-action. 
We say that $\lambda$ is {\it regular} if $W_\lambda=\{1\}$.

Set
$\calC^\pm=\{\lambda\in\bft^*\,;\,
\langle\lambda+\hat\rho:\alpha\rangle\geqslant 0,\,
\alpha\in\widehat\Pi^\pm\}.$
%\calC^-_e=\calC^-\cap\bft^*_e,\quad\text{for}\ e>0,$$
%The {\it dominant alcove} is given by
%$$\calC^+=\{\lambda\in\bft^*\,;\,
%\langle\lambda+\hat\rho:\alpha\rangle\geqslant 0,\,
%\alpha\in\widehat\Pi^+\}.$$
%\calC^+_e=\calC^+\cap\bft^*_e,\quad\text{for}\ e<0.$$
An element of $\calC^-$ (resp.~of $\calC^+$)
is called an {\it antidominant affine weight} (resp.~
a {\it dominant affine weight}).
We write again $\calC^\pm$ for $cl(\calC^\pm)$.
We have the following basic fact, see e.g.,  \cite[lem.~2.10]{KT}.

\vspace{2mm}

\begin{lemma}
Let $\lambda$ be an integral affine weight of level $-e-N$. We have

(a) $\sharp(\widehat W\bullet\lambda\cap\calC^-)=1$ and
$\sharp(\widehat W\bullet\lambda\cap\calC^+)=0$
if  $e>0$,

(b) 
$\sharp(\widehat W\bullet\lambda\cap\calC^+)=1$ and
$\sharp(\widehat W\bullet\lambda\cap\calC^-)=0$
if  $e<0$.
\end{lemma}

\vspace{.5mm}

We say that
$\lambda$ is {\it negative} in the first case,
and {\it positive} in the second case.

For $\lambda\in\calC^\pm$ the subgroup
$W_\lambda$ of $\widehat W$ is finite and is a standard
parabolic subgroup.  It is isomorphic to the Weyl group
of the root system 
$
\{\alpha\in\widehat\Pi\,;\,\langle\lambda+\hat\rho:\alpha\rangle=0\}.$

\vspace{2mm}

\subsection{The parabolic category $\bfO$} 
\label{sec:2.9}
Let $\calP$ be the set of proper subsets of $\widehat\Phi$.
An element of $\calP$ is called a {\it parabolic type}.
Fix a parabolic type $\nu$.
If $\nu$ is the empty set we say that {\it $\nu$ is regular},
and we write $\nu=\phi$.

Let $\bfp_\nu\subset\bfg$ be the unique parabolic subalgebra containing $\bfb$ whose set of roots
is generated by $\widehat\Phi\cup(-\nu)$. 
Let $\Pi_\nu$ be the root system of a levi subalgebra of $\bfp_\nu$.
Set $\Pi^+_\nu=\Pi^+\cap\Pi_\nu$.
Let $W_\nu\subset\widehat W$ be the Weyl group of $\Pi_\nu$.
Let $w_\nu$ be the longest element in $W_\nu$. 
Let $\tilde\bfO^{\nu}$ be the category of
all $\bfg$-modules $M$ such that
$M=\bigoplus_{\lambda\in\bft^*}M_\lambda$ with
$M_\lambda=\{m\in M;xm=\lambda(x)m,\,x\in\bft\}$ 
and $U(\bfp_\nu)\,m$ is finite dimensional for each $m\in M$.
An affine weight $\lambda$ is {\it $\nu$-dominant } if
$(\lambda:\check\alpha)\in\bbN$ for all
$\alpha\in\Pi_\nu^+.$
For any $\nu$-dominant affine weight $\lambda,$ let 
$V^\nu(\lambda)$, $L(\lambda)$ 
be the parabolic Verma module with the highest weight
$\lambda$ and its simple top.

Recall that $e$ is an integer $\neq 0$.
For any weight $\lambda\in\frakt^*$, let
$\lambda_e=\lambda-(e+N)\Lambda_0$ and
$z_\lambda=\langle\lambda:2\rho+\lambda\rangle/2e.$
Let $\bfO^{\nu}\subset\tilde\bfO^{\nu}$ 
be the full subcategory of the modules such that 
the highest weight of any of its simple subquotients is of the form
$\tilde\lambda_e=\lambda_e+z_\lambda\,\delta$, where $\lambda$ is a $\nu$-dominant integral weight.
%Note that $\lambda_e=w\bullet\mu_e$ if and only if
%$\tilde\lambda_e=w\bullet\tilde\mu$. 
We'll abbreviate
$V^\nu(\lambda_e)=V^\nu(\tilde\lambda_e)$ and
$L(\lambda_e)=L(\tilde\lambda_e).$

Now, fix $\mu\in\calP$ and assume that $e>0$. We use the following notation
\begin{itemize}
\item $\oo_{\mu,-}$ is an antidominant 
integral classical affine weight of level $-e-N$
whose stabilizer for the $\bullet$-action of $\widehat W$ is equal to $W_\mu$, 
\item $\oo_{\mu,+}=-\oo_{\mu,-}-2\hat\rho$
is a dominant integral classical affine weight of level $e-N$
whose stabilizer for the $\bullet$-action of $\widehat W$ is 
equal to $W_\mu$,
\item $\bfO^{\nu}_{\mu,\pm}\subset\bfO^{\nu}$
is the full subcategory
consisting of the 
modules such that the highest weight of any of its simple subquotients
is linked to $\oo_{\mu,\pm}$.
\end{itemize}

Let $I_\mu^\min, I^\max_\mu\subset\widehat W$ be the sets of minimal 
and maximal length
representatives of the left cosets in $\widehat W/W_\mu$.
Let $\leqslant$ be the Bruhat order.
We'll consider the posets
\begin{itemize}
\item
$I_{\mu,-}=(I_\mu^\max,\preccurlyeq)$ where
$\preccurlyeq=\leqslant$,  and $I_{\mu,-}^\nu=\{x\in I_{\mu,-}\,;\,x\bullet\oo_{\mu,-}\ 
\text{is}\  \nu\text{-dominant}\}$,
\item
$I_{\mu,+}=(I_\mu^\min,\preccurlyeq)$ where
$\preccurlyeq=\geqslant$, and 
$I_{\mu,+}^\nu=\{x\in I_{\mu,+}\,;\,x\bullet\oo_{\mu,+}\ 
\text{is}\  \nu\text{-dominant}\}.$
\end{itemize}
They have the following properties.

\vspace{2mm}

\begin{lemma} 
\label{lem:C}
For any $\mu$, $\nu\in\calP$ we have

(a) $x\in I_{\phi,+}^\mu\iff x^{-1}\in I_{\mu,+}$,

(b) $x\in I_{\phi,-}^\mu\iff x^{-1}\in I_{\mu,-}$,

(c) $I_{\phi,\pm}^\mu\cap I_{\nu,\mp}=\{xw_\nu\,;\,x\in I_{\nu,\pm}^\mu\}$, 

(d) $x\in I^\mu_{\nu,+}\iff w_\mu xw_\nu\in I_{\nu,-}^\mu$.
\end{lemma}

\vspace{.5mm}

\begin{proof}
To prove $(a)$ note that, since $\oo_{\phi,+}$ is dominant regular, we have
$$\aligned
I_{\phi,+}^\mu
&=\{x\in\widehat W\,;\,x\bullet\oo_{\phi,+}\ 
\text{is}\ \mu\text{-dominant}\}\cr
&=\{x\in \widehat W\,;\, 
\langle \oo_{\phi,+}+\hat\rho:x^{-1}(\check\alpha)\rangle\geqslant 0,\,\forall
\alpha\in\Pi^+_\mu\}\\
&=\{x\in\widehat W\,;\,x^{-1}(\Pi_\mu^+)\subset\widehat\Pi^+\}\cr
&=\{x\in\widehat W\,;\,x^{-1}\in I_{\mu,+}\}.
\endaligned$$
The proof of $(b)$ is  similar and is left to the reader.
Now we prove part $(c)$.
Choose positive integers $d,f$ such that $\pm(f-d)>-N$.
Then, the translation functor $T_{\phi,\nu}:{}^z\bfO_{\phi,\pm}\to{}^z\bfO_{\nu,\pm}$ 
in Proposition \ref{prop:translation2} is well-defined for any $z\in I_\nu^\max$.
By Proposition \ref{prop:translation2}$(d),(e)$, we have
$$\aligned
I^\mu_{\nu,\pm}&=\{x\;;\;xw_\nu\in I^\mu_{\phi,\pm},\,x\in I_{\nu,\pm}\}
=\{x\;;\;xw_\nu\in I^\mu_{\phi,\pm}\cap I_{\nu,\mp}\}.
\endaligned $$
Part $(d)$ is standard. More precisely, by \cite[prop.~2.7.5]{Ca}, if
$x\in I^\mu_{\nu,+}$ then we have $W_\mu\cap x(W_\nu)=\emptyset$ and
$x\in(I^\min_\mu)^{-1}\cap I_\nu^\min$. Then, we also have
$w_\mu x\in(I^\max_\mu)^{-1}\cap I_\nu^\min,$ from which we deduce that
$w_\mu x w_\nu\in I^\mu_{\nu,-}$. The lemma is proved.
\end{proof}

\vspace{2mm}

For each $x\in I_{\mu,\mp}^\nu$, we write $x_\pm=w_\nu xw_\mu\in I_{\mu,\pm}^\nu$. 
From Lemma \ref{lem:C} and its proof, we get the following.

\vspace{2mm}

\begin{cor}
\label{lem:C2}
(a) We have
$I^\mu_{\nu,+}=\{x\;;\;xw_\nu\in(I^\min_\mu)^{-1}\cap I_\nu^\max\}$ and
$I^\mu_{\nu,-}=\{x\;;\;xw_\nu\in(I^\max_\mu)^{-1}\cap I_\nu^\min\}.$

(b) The map
$I_{\mu,\pm}^\nu\to I_{\nu,\pm}^\mu$, $x\mapsto x^{-1}$ is an isomorphism of posets.

(c) The map
$I_{\mu,\mp}^\nu\to I_{\mu,\pm}^\nu,$ $x\mapsto x_\pm$
is an anti-isomorphism of posets.
\end{cor}

%\vspace{.5mm}

%\begin{rk} 
%Set $I_{\mu,\nu}^\min=(I_\mu^\min)^{-1}\cap I_\nu^\min$
%and $I_{\mu,\nu}^\max=(I_\mu^\max)^{-1}\cap I_\nu^\max$.
%We have the inclusions 
%$I^\mu_{\nu,+}\subset I^\min_{\mu,\nu}$ and
%$I^\mu_{\nu,-}\subset I^\max_{\mu,\nu},$
%see e.g., \cite[sec.~2.7]{Ca}.
%\end{rk}

\vspace{2mm}

The  category
$\bfO^\nu_{\mu,\pm}$ is a direct summand in $\bfO^\nu$ by
the linkage principle, see \cite[thm.~6.1]{So}, and
we have
$\Irr(\bfO^{\nu}_{\mu,\pm})=
\{L(x\bullet\oo_{\mu,\pm})\,;\,x\in I_{\mu,\pm}^\nu\}$.
Further, for each $x\in I^\nu_{\mu,+}$, the simple module
$L(x\bullet\oo_{\mu,+})$ has a projective cover $P^\nu(x\bullet\oo_{\mu,+})$ in $\bfO_{\mu,+}^\nu$. 
For each $x\in I^\nu_{\mu,-},$ there is a tilting module $T^\nu(x\bullet\oo_{\mu,-})$ in $\bfO_{\mu,-}^\nu$ with
 highest weight $x\bullet\oo_{\mu,-}$, see e.g., \cite{So}.
Note that $\bfO_{\mu,+}^\nu$ does not have tilting objects and 
$\bfO_{\mu,-}^\nu$ does not have projective objects.
Let $\bfO^{\nu,\Delta}_{\mu,\pm}$ be the full subcategory of $\bfO^{\nu,\Delta}_{\mu,\pm}$ consisting of objects with a finite filtration by the parabolic Verma modules.
The following is a version of Ringel equivalence for the categories $\bfO_{\mu,\pm}^\nu$.

\vspace{2mm}

\begin{lemma}\label{lem:ringel-soergel}
There is an equivalence of exact categories
$D: \bfO^{\nu,\Delta}_{\mu,+}\iso\,(\bfO^{\nu,\Delta}_{\mu,-})^\op$
which maps $V^\nu(x\bullet\oo_{\mu,+})$
to $V^\nu(x_-\bullet\oo_{\mu,-}),$
and $P^\nu(x\bullet\oo_{\mu,+})$
to $T^\nu(x_-\bullet\oo_{\mu,-}).$
\end{lemma}

\begin{proof}
Note that $x_-\bullet\oo_{\mu,-}=w_\nu x\bullet\oo_{\mu,-}=-w_\nu(x\bullet\oo_{\mu,+}+\hat\rho)-\hat\rho$. 
So the lemma is given by \cite[thm.~6.6, proof of cor.~7.6]{So}.
\end{proof}

\smallskip
\begin{rk}\label{rk:ringel-soergel}
Recall that the BGG-duality on $\bfO^{\nu}_{\mu,-}$ is an equivalence $\dd: \bfO^{\nu}_{\mu,-}\iso \bfO^{\nu,\op}_{\mu,-}$ which maps a simple module to itself, maps a parabolic Verma module to a dual parabolic Verma module, and maps a tilting module to itself. Let $\bfO^{\nu,\nabla}_{\mu,-}$ be the full subcategory of $\bfO^{\nu}_{\mu,-}$ consisting of objects with a finite filtration by dual parabolic Verma modules. Then we have an equivalence
$$\dd\circ D:\bfO^{\nu,\Delta}_{\mu,+}\iso\,\bfO^{\nu,\nabla}_{\mu,-}$$
which maps $V^\nu(x\bullet\oo_{\mu,+})$
to $\dd(V^\nu(x_-\bullet\oo_{\mu,-})),$
and $P^\nu(x\bullet\oo_{\mu,+})$
to $T^\nu(x_-\bullet\oo_{\mu,-}).$
\end{rk}
\vspace{2mm}

\subsection{The truncated category $\bfO$}
\label{sec:truncation}
Fix parabolic types $\mu,\nu\in\calP$ and an integer $e>0$.
Fix an element $w\in\widehat W$. We
define the \emph{truncated parabolic category} $\bfO$ at the negative/positive level
in the following way.

\smallskip

First, set 
${}^w\!I^\nu_{\mu,-}=\{x\in I^\nu_{\mu,-}\,;\,x\preccurlyeq\! w\}$ and
let ${}^w\bfO^{\nu}_{\mu,-}$ be the full subcategory of $\bfO^{\nu}_{\mu,-}[{}^w\!I^\nu_{\mu,-}]$ consisting of the
finitely generated modules. We have
$\Irr({}^w\bfO^{\nu}_{\mu,-})=
\{L(x\bullet\oo_{\mu,-})\,;\,x\in {}^w\!I_{\mu,-}^\nu\}$.
For each $x\in {}^w\!I^\nu_{\mu,-},$ the modules $V^\nu(x\bullet\oo_{\mu,-})$ and $T^\nu(x\bullet\oo_{\mu,-})$ belong to ${}^w\bfO^{\nu}_{\mu,-}$. We may 
write $T^\nu(x\bullet\oo_{\mu,-})={}^w\!T^\nu(x\bullet\oo_{\mu,-})$. 
The module $L(x\bullet\oo_{\mu,-})$ has a projective cover in ${}^w\bfO^{\nu}_{\mu,-}$. We denoted it by ${}^w\!P^\nu(x\bullet\oo_{\mu,-})$. Let ${}^w\!P^\nu_{\mu,-}$ be the minimal projective generator, ${}^w\!T^\nu_{\mu,-}$ the characteristic tilting 
module, and let ${}^w\!L^\nu_{\mu,-}$ be the sum of simple modules.
\smallskip

\begin{df}\label{df:3.6}
Set
${}^w\!\bar A^\nu_{\mu,-}=
\Ext_{{}^{w}\bfO^\nu_{\mu,-}}({}^{w}\!L^\nu_{\mu,-})^\op$ and
${}^w\!A^\nu_{\mu,-}=\End_{{}^w\bfO^\nu_{\mu,-}}({}^w\!P^\nu_{\mu,-})^{\op}.$
\end{df}

\smallskip

\begin{lemma}
The category ${}^w\bfO^{\nu}_{\mu,-}\simeq {}^w\!A^\nu_{\mu,-}\bfmod$ is a highest weight category 
with the set of standard objects
$\Delta({}^w\bfO^{\nu}_{\mu,-})=
\{V^\nu(x\bullet\oo_{\mu,-})\,;\,x\in{}^w\!I^{\nu}_{\mu,-}\}$
and the highest weight order given by the
partial order $\preccurlyeq$ on ${}^w\!I_{\mu,-}^\nu$.
\end{lemma}

\smallskip

\begin{proof}
First, the category $\bfO^{\nu}_{\mu,-}$ is a highest weight category in the sense of 
\cite[def.~3.1]{CPS2}, i.e., it satisfies the axioms in Section \ref{sec:5} although it is not equivalent to the module category of a finite dimensional algebra. By \cite[thm.~3.5]{CPS2} the subcategory ${}^w\bfO^{\nu}_{\mu,-}$ is also highest weight. Since ${}^w\bfO^{\nu}_{\mu,-}\simeq {}^w\!A^\nu_{\mu,-}\bfmod$ and ${}^w\!A^\nu_{\mu,-}$ is finite dimensional, the category ${}^w\bfO^{\nu}_{\mu,-}$ is a highest weight category in the sense of
Section \ref{sec:5}.
\end{proof}

\smallskip

Next, set
${}^w\!I_{\mu,+}^\nu=\{x\in I_{\mu,+}^\nu\,;\,x\succcurlyeq\! w\}$ and let
${}^w\bfO^{\nu}_{\mu,+}$ be the full subcategory of $\bfO^{\nu}_{\mu,+}({}^w\!I_{\mu,+}^\nu)$ consisting of the finitely 
generated objects. We have
$\Irr({}^w\bfO^{\nu}_{\mu,+})=
\{L(x\bullet\oo_{\mu,+})\,;\,x\in {}^w\!I_{\mu,+}^\nu\}$.
Let
${}^wL^\nu_{\mu,+}=\bigoplus_{x\in {}^w\!I^\nu_{\mu,+}}L(x\bullet\oo_{\mu,+})$
and let ${}^w\!P^\nu_{\mu,+}=\bigoplus_{x\in {}^w\!I^\nu_{\mu,+}}P^\nu(x\bullet\oo_{\mu,+})$.

\smallskip
\begin{df}
Set
${}^w\!\bar A^\nu_{\mu,+}=
\Ext_{{}^{w}\bfO^\nu_{\mu,+}}({}^{w}\!L^\nu_{\mu,+})^\op$ and
${}^w\!A^\nu_{\mu,+}=\End_{{}^w\bfO^\nu_{\mu,-}}({}^w\!P^\nu_{\mu,+})^{\op}.$
\end{df}
\smallskip

Consider the quotient functor
$$F=\Hom_{\bfO^{\nu}_{\mu,+}}({}^w\!P^\nu_{\mu,+},\bullet): \bfO^{\nu}_{\mu,+}\to 
{}^w\!A^\nu_{\mu,+}\bfMod.$$
Its kernel is the full subcategory generated by the modules $L(x\bullet\oo_{\mu,+})$ with $x\not\in {}^v\!I^\nu_{\mu,+}$. 
So, we have an equivalence of categories 
$\bfO^{\nu}_{\mu,+}({}^v\!I^\nu_{\mu,+})\simeq 
{}^w\!A^\nu_{\mu,+}\bfMod$.
It restricts to an equivalence
$${}^w\bfO^{\nu}_{\mu,+}\simeq {}^w\!A^\nu_{\mu,+}\bfmod.$$
For each $x\in{}^w\!I_{\mu,+}^\nu$, we'll view $V^\nu(x\bullet\oo_{\mu,+})$ as an object in ${}^w\bfO^{\nu}_{\mu,+}$ by identifying it with the ${}^w\!A^\nu_{\mu,+}$-module $F(V^\nu(x\bullet\oo_{\mu,+}))$. We denote
${}^w\!P^\nu(x\bullet\oo_{\mu,+})=F(P^\nu(x\bullet\oo_{\mu,+}))$.

\smallskip

For each $x\in{}^w\!I_{\mu,\pm}^\nu$, 
let $1_x$ be the (obvious) idempotents in the $\bbC$-algebras ${}^w\!\bar A^\nu_{\mu,\pm}$ and
${}^w\!A^\nu_{\mu,\pm},$
associated with the modules $L(x\bullet\oo_{\mu,\pm})$
and ${}^w\!P^\nu(x\bullet\oo_{\mu,\pm})$.
We'll abbreviate $1_x^\diamond=(1_x)^\diamond$, ${}^w\!A^{\nu,\diamond}_{\mu,-}=({}^w\!A^\nu_{\mu,-})^\diamond$ and
${}^v\bfO^{\nu,\diamond}_{\mu,-}=({}^v\bfO^{\nu}_{\mu,-})^\diamond$.

\smallskip

\begin{prop}
\label{prop:ringel}
Assume that $w\in I^\nu_{\mu,+}$ and $v=w_-\in I^\nu_{\mu,-}$.

(a)
The category ${}^w\bfO^{\nu}_{\mu,+}\simeq {}^w\!A^\nu_{\mu,+}\bfmod$ is a highest weight category 
with the set of standard objects
$\Delta({}^w\bfO^{\nu}_{\mu,+})=
\{V^\nu(x\bullet\oo_{\mu,+})\,;\,x\in{}^w\!I^{\nu}_{\mu,+}\}$
and the highest weight order given by the
partial order $\preccurlyeq$ on ${}^w\!I_{\mu,+}^\nu$.

(b)
There is an equivalence of highest weight categories
${}^v\bfO^{\nu,\diamond}_{\mu,-}\simeq{}^{w}\bfO^{\nu}_{\mu,+}$ 
which takes
$V^\nu(x\bullet\oo_{\mu,-})^\diamond$
to $V^\nu(x_+\bullet\oo_{\mu,+})$ for any $x\in{}^v\!I_{\mu,-}^\nu.$

(c) There is a $\bbC$-algebra isomorphism 
${}^w\!A^{\nu,\diamond}_{\mu,-}\simeq{}^{v}\!A^\nu_{\mu,+}$
such that $1_x^\diamond\mapsto 1_{(x_+)}$ for any $x\in{}^v\!I_{\mu,-}^\nu.$
\end{prop}
\begin{proof}
First, note that the anti-isomorphism of posets $I_{\mu,+}^\nu\to I_{\mu,-}^\nu,$ $x\mapsto x_-,$
in Corollary \ref{lem:C2} restricts to an anti-isomorphism of posets 
\begin{equation*}%\label{eq:antiiso}
{}^w\!I_{\mu,+}^\nu\iso\, {}^v\!I_{\mu,-}^\nu.
\end{equation*} 
Therefore by Remark \ref{rk:ringel-soergel} the functor $\dd\circ D$ maps ${}^w\!P^\nu_{\mu,+}$ to ${}^v\!T^\nu_{\mu,-}$, and yields an algebra isomorphism
$\End_{\bfO^{\nu}_{\mu,+}}({}^w\!P^\nu_{\mu,+})^\op\simeq\End_{\bfO^{\nu}_{\mu,-}}({}^v\!T^\nu_{\mu,-})^\op.$ So, we have ${}^w\bfO^{\nu}_{\mu,+}=({}^v\bfO^{\nu}_{\mu,-})^\diamond$. 
In particular ${}^w\bfO^{\nu}_{\mu,+}$ is a highest weight category because it is the Ringel dual of a highest weight category. The rest of the proposition follows from the generalities on Ringel duality in Section \ref{sec:R}.
\end{proof}

\smallskip

We'll denote by ${}^w\! T^\nu(x\bullet\oo_{\mu,+})$ the tilting object associated with $V^\nu(x\bullet\oo_{\mu,+})$ in ${}^w\bfO^{\nu}_{\mu,+}$
and by ${}^w\! T^\nu_{\mu,+}$ the characteristic tilting module.
If $\nu=\phi,$ we also abbreviate 
${}^w\bfO_{\mu,\pm}={}^w\bfO_{\mu,\pm}^\phi,$
suppressing $\phi$ everywhere in the notation.

\vspace{2mm}

\begin{rk}
\label{rk:level}
The highest weight category 
${}^w\bfO^\nu_{\mu,\pm}$ does not depend on the choice of 
$\oo_{\mu,\pm}$ and $e$
but only on $\mu$, $\nu$ and on the sign of the level, see \cite[thm.~11]{F2}. 
\end{rk}

\vspace{2mm}

\begin{rk}
Write $D=\dd\circ(\bullet)^\diamond$, where $\dd$ is the BGG duality.
Then, $D$ restrict to an equivalence of categories
${}^w\bfO^{\nu,\Delta}_{\mu,+}\to({}^v\bfO^{\nu,\Delta}_{\mu,-})^\op$.
\end{rk}

\vspace{3mm}

\subsection{Parabolic inclusion and truncation}
\label{sec:tau}

Let $i=i_{\nu,\phi}$ be the canonical inclusion ${}^w\bfO^\nu_{\mu,\pm}\subset{}^w\bfO_{\mu,\pm}$.
It is a fully faithful functor and it admits a left adjoint $\tau=\tau_{\phi,\nu}$ which
takes an object of ${}^w\bfO_{\mu,\pm}$
to its largest quotient which lies in ${}^w\bfO_{\mu,\pm}^\nu$. 
We'll call $i$ the {\it parabolic inclusion functor} 
and $\tau$ the {\it parabolic truncation functor}.

Since $i$ is exact, the functor $\tau$ takes projectives to projectives.
We have
\begin{enumerate}
\item[(a)]
$i(L(x\bullet \oo_{\mu,\pm}))=L(x\bullet \oo_{\mu,\pm})$
for $x\in{}^w\!I_{\mu,\pm}^\nu$,
\item[(b)]
$\tau({}^w\!P(x\bullet \oo_{\mu,\pm}))={}^w\!P^\nu(x\bullet \oo_{\mu,\pm})$
for $x\in{}^w\!I_{\mu,\pm}^\nu$,
\item[(c)] 
$\tau({}^w\!P(x\bullet \oo_{\mu,\pm}))=0$
for $x\in{}^w\!I_{\mu,\pm}\setminus{}^w\!I_{\mu,\pm}^\nu$.
\end{enumerate}

The same argument as in \cite[thm.~3.5.3]{BGS},
using \cite[thm.~5.5]{FG2},
implies that the functor $i$ is injective on extensions in ${}^w\bfO^\nu_{\mu,-}$. 
See the proof of Proposition \ref{prop:loc2} for more details.

By Lemma \ref{lem:ringeltroncation}, for each $w\in I^\nu_{\mu,-}$,
the Ringel equivalence yields equivalences of
triangulated categories 
$\bfD^b({}^v\bfO_{\mu,+})\to\bfD^b({}^w\bfO_{\mu,-})$
and
$\bfD^b({}^v\bfO^\nu_{\mu,+})\to\bfD^b({}^w\bfO^\nu_{\mu,-})$,
where $v=w_+:=w_\nu ww_\mu$,
such that the diagram below commutes
$$\xymatrix{
\bfD^b({}^v\bfO_{\mu,+})\ar[r]&\bfD^b({}^w\bfO_{\mu,-})\\ 
\bfD^b({}^v\bfO^\nu_{\mu,+})\ar[r]\ar[u]^-i&
\bfD^b({}^w\bfO^\nu_{\mu,-}).\ar[u]_-i
}$$
Thus, the parabolic inclusion functor $i$ is also injective on extensions in ${}^v\bfO^\nu_{\mu,+}$.

\vspace{3mm}

%%%%%%%%%%%%%%%%%%%%%%%

\vspace{2mm}

%%%%%%%%%%%%%%%%%%%%%%%%%%%%%%%%%%%%%%%%%%%%%

\vspace{2mm}

\subsection{The main result}
Fix integers $e,f>0$ and parabolic types $\mu,\nu\in\calP$.
We choose  
$\oo_{\mu,\pm}$ of level $\pm e-N$ and
$\oo_{\nu,\pm}$ of level $\pm f-N$. 
Fix an element $w\in\widehat W$.
Assume that $w\in I^\nu_{\mu,+},$ and set $v=w_-^{-1}\in I^\mu_{\nu,-}$.

The main result of this paper is the following.

\vspace{2mm}

\begin{thm} \label{thm:main}
We have $\bbC$-algebra isomorphisms
${}^w\!A^\nu_{\mu,+}={}^v\!\bar A^\mu_{\nu,-}$
and
${}^w\!\bar A^\nu_{\mu,+}={}^v\! A^\mu_{\nu,-}$
such that $1_x\mapsto 1_y$ with $y=x_-^{-1}$ for each $x\in{}^w\!I_{\mu,\mp}^\nu$.
The graded $\bbC$-algebras
${}^w\!\bar A^\nu_{\mu,+}$ and
${}^v\!\bar A^\mu_{\nu,-}$
are Koszul and balanced. They are Koszul dual to each other, i.e., we have
a graded $\bbC$-algebra isomorphism
$({}^w\!\bar A^\nu_{\mu,+})^!={}^v\!\bar A^\mu_{\nu,-}$
such that $1_x^!=1_{y}$ for each $x\in{}^w\!I_{\mu,+}^\nu.$
The categories 
${}^w\bfO^\nu_{\mu,+},$
${}^v\bfO^\mu_{\nu,-}$
are Koszul and are Koszul dual to each other.
\qed
\end{thm}

\vspace{2mm}

We'll abbreviate ${}^w\!\bar A^{\nu,!}_{\mu,\pm}=({}^w\!\bar A^\nu_{\mu,\pm})^!$.

\vspace{2mm}

\begin{rk}
\label{rk:2.20}
Since ${}^w\!\bar A^\nu_{\mu,+}$ and
${}^v\!\bar A^\mu_{\nu,-}$
are balanced,
the Koszul duality commutes with the Ringel duality.
In particular, for $w\in I^\nu_{\mu,-}$ and $v=w^{-1}\in I^\mu_{\nu,-}$, composing $(\bullet)^!$ and $(\bullet)^\diamond$ we get
a graded $\bbC$-algebra isomorphism
$({}^w\!\bar A^{\nu,\diamond}_{\mu,-})^!={}^v\!\bar A^\mu_{\nu,-}$
such that $(1_x^\diamond)^!=1_{y}$ with $y=x^{-1}$ for each $x\in{}^w\!I_{\mu,-}^\nu.$

%Composing $ED$ with the BGG duality, we get an equivalence of triangulated categories 
%$\bfD^b({}^w\bfO^\nu_{\mu,-})\to\bfD^b({}^v\bfO^\mu_{\nu,-})$. 
%Since $E$ takes parabolic Verma modules to parabolic Verma modules by 
%Remark \ref{rk:KS}, this equivalence maps
%$V^\nu(x\bullet\oo_{\mu,-})$ to $V^\mu(y\bullet\oo_{\nu,-})$ for each $x,y$ as above.
%Note that we equipped the highest weight categories 
%${}^v\bfO^\mu_{\nu,-}$,
%${}^w\bfO^\nu_{\mu,+}$
%with the Bruhat and opposite Bruhat order respectively.
%We consider the Ringel dual with respect to this order.
%This is not the convention taken usually in the literature.
\end{rk}

\vspace{1cm}

\section{Moment graphs, deformed category $\bfO$ and localization}
\label{sec:momentgraph}

First, we introduce some notation.
Fix integers $e,f>0$ and parabolic types $\mu,\nu\in\calP$.

Let $V$ be a finite dimensional $\bbC$-vector space, let
$S$ be the symmetric $\bbC$-algebra over $V$ and let
$\frakm\subset S$ be the maximal ideal generated by $V$.
We may regard $S$ as a graded $\bbC$-algebra
such that $V$ has the degree 2. 

Let $R$  be a commutative, noetherian, integral domain which is
a (possibly graded) $S$-algebra with 1.

Let $S_0$ be the localization of $S$ at the ideal $\frakm$ and
let $k$ be the residue field of $S_0$. Note that $k=\bbC$ and
that $S_0$ has no natural grading.

\vspace{2mm}

\subsection{Moment graphs}
\label{sec:3.3}
Assume that $R=S$, viewed as a graded $\bbC$-algebra.  

\vspace{2mm}

\begin{df}
A {\it moment graph} over $V$ is
a tuple $\calG=(I,H,\alpha)$ where
$(I,H)$ is a graph with a  set of vertices $I$, 
a set of edges $H$, each edge joins two vertices,
and $\alpha$ is a map $H\to\bbP(V)$, $h\mapsto k\alpha_h$.
\end{df}

We say that the moment graph $\calG$ is a {\it GKM-graph} if we have
$k\alpha_{h_1}\neq k\alpha_{h_2}$ for any edges $h_1\neq h_2$ adjacent 
to the same vertex.

\vspace{1mm}

\begin{df}
An order on $\calG$ is a partial order $\preccurlyeq$
on $I$ such that the two vertices joined by an edge
are comparable.
\end{df}

\vspace{1mm}

Given an order $\preccurlyeq$ on $\calG$ let 
$h',h''$ denote the {\it origin} and the {\it goal} of the edge $h$.
In other words $h'$ and $h"$ are two vertices joined by $h$ and we have $h'\prec h''$.

\begin{rk}
We  use the terminology in \cite{J}.
In \cite{F3}, a moment graph is  always assumed to be ordered. 
We'll also assume that $\calG$ is {\it finite}, i.e., that the sets 
$I$ and $H$ are finite. 
\end{rk}

\vspace{2mm}

\begin{df}
A {\it graded $R$-sheaf over $\calG$} 
is a tuple $\calM=(\calM_x,\calM_h,\rho_{x,h})$
with
\begin{itemize}
\item a graded $R$-module $\calM_x$ for each $x\in I$,
\item a graded $R$-module $\calM_h$ for each $h\in H$ such that 
$\alpha_h\calM_h=0$,
\item a graded $R$-module homomorphism $\rho_{x,h}:\calM_x\to\calM_h$  for $h$ adjacent to $x$.
\end{itemize}
A {\it morphism}  $\calM\to\calN$ is an $I\sqcup H$-tuple
$f$ of graded $R$-module homomorphisms
$f_x:\calM_x\to\calN_x$ and
$f_h:\calM_h\to\calN_h$ which are compatible with $(\rho_{x,h})$.
\end{df}

\vspace{2mm}

The $R$-modules $\calM_x$ are called the \textit{stalks} of $\calM$.
For $J\subset I$ we set
$$\calM(J)=\bigl\{(m_x)_{x\in J}\,;\,m_x\in\calM_x,\,\rho_{h',h}(m_{h'})=
\rho_{h'',h}(m_{h''}),\ \forall h\,\text{with}\, h',h''\in J\bigr\}.$$
Note that $\calM(\{x\})=\calM_x$.
The {\it space of global sections} of the graded
$R$-sheaf $\calM$ is the graded $R$-module
$\Gamma(\calM)=\calM(I)$.

We say that $\calM$ is of {\it finite type} if all $\calM_x$ and all $\calM_h$ are
finitely generated graded $R$-modules.
If  $\calM$ is of finite type, then the graded $R$-module
$\Gamma(\calM)$ is finitely generated, 
because $R$ is noetherian.

The {\it graded structural algebra} of $\calG$
is the graded $R$-algebra $\bar Z_R=\bigl\{(a_x)\in R^{\oplus I}\,;\,a_{h'}-a_{h''}\in\alpha_h R,\,
\forall h\}.$

\vspace{2mm}

\begin{df}\label{def:3.5}
Let $\bar\bfF_{R}$ be the category consisting of the graded $R$-sheaves of finite type
over $\calG$ whose stalks are torsion free $R$-modules.
Let $\bar\bfZ_{R}$ be the category consisting of the graded
$\bar Z_{R}$-modules which are finitely generated and torsion free over $R$.
\end{df}

\vspace{2mm}

The categories $\bar\bfF_{R}$ and $\bar\bfZ_{R}$ are Krull-Schmidt graded 
$R$-categories (because they are hom-finite $R$-categories and each idempotent splits).
Taking the global sections yields a functor $\Gamma:\bar\bfF_R\to\bar\bfZ_R.$

We'll call again graded $R$-sheaves the objects of $\bar\bfZ_{R}$, hoping it will not create any confusion.

The global sections functor $\Gamma$
has a left adjoint $\calL$ called the {\it localization functor}
 \cite[thm.~3.6]{F3}, \cite[prop.~2.16, sec.~2.19]{J}.

We say that $\calM$ is {\it generated by global sections} 
if the counit $\calL\Gamma(\calM)\to\calM$ is an isomorphism, or,
equivalently, if $\calM$ belongs to the essential image of $\calL$. 
This implies that the obvious map $\Gamma(\calM)\to\calM_x$ 
is surjective for each $x$.
The functor $\Gamma$ is fully faithful on the 
full subcategory of $\bar\bfF_{R}$ 
of the graded $R$-sheaves which are generated by global sections,
see \cite[sec.~3.7]{F3}.

For each subsets $E\subset H$ and $J\subset I,$ we set
$
\calM_{E}=\bigoplus_{h\in E}\calM_h
$ and
$$
\rho_{J,E}=
\bigoplus_{x\in J}\bigoplus_{h\in E}\rho_{x,h}:\calM(J)\to \calM_{E}.
$$

Let $d_x$ be the set of edges with goal $x$,
$u_x$ be the set of edges with origin $x$ and
$e_x=d_x\bigsqcup u_x$.
Given an order $\preccurlyeq$ on $\calG,$ we set
$\calM_{\partial x}=\Im(\rho_{\prec x,d_x}),$
$\calM^x=\Ker(\rho_{x,e_x})$ and
$\calM_{[x]}=\Ker(\rho_{x,d_x}),$
where we abbreviate $\rho_{x,E}=\rho_{\{x\},E}$.
Following \cite[sec.~4.3]{FW}, we call $\calM^x$ the {\it costalk} of $\calM$ at $x$.
Note that $\calM^x$, 
$\calM_{[x]}$, 
$\calM_x$ are graded $\bar Z_{R}$-modules such that
$\calM^x\subset\calM_{[x]}\subset\calM_x.$ 

\smallskip

Assume from now on that the moment graph $\calG$ is a GKM graph. Let us quote the following definitions, see \cite[lem.~3.2]{J} and \cite[lem.~4.8]{F3}.

\smallskip

\begin{df}
Assume that $\calM\in\bar\bfF_R$ is generated by global sections. Then,

$(a)$ $\calM$ is {\it flabby} 
if $\Im(\rho_{x,d_x})=\calM_{\partial x}$ for each $x\in I$,

$(b)$ $\calM$ is {\it $\Delta$-filtered} if it is flabby and if
the graded $\bar Z_{R}$-module $\calM_{[x]}$ is a free graded $R$-module 
for each $x\in I$.
\end{df}

\vspace{2mm}

Now, we can define the following categories.

\smallskip

\begin{df}
Let $\bar\bfF^\Delta_{R,\preccurlyeq}$ 
be the full subcategory of $\bar\bfF_{R}$ 
consisting of the $\Delta$-filtered objects.
Let $\bar\bfZ^\Delta_{R,\preccurlyeq}$ be the
essential image of $\bar\bfF^\Delta_{R,\preccurlyeq}$ by $\Gamma$. 
\end{df}

\vspace{2mm}

We may abbreviate $\bar\bfF^\Delta_{R}=\bar\bfF^\Delta_{R,\preccurlyeq}$ 
and $\bar\bfZ^\Delta_{R}=\bar\bfZ^\Delta_{R,\preccurlyeq}$ when the order is clear from the context.
The categories $\bar\bfF^\Delta_{R}$ and $\bar\bfZ^\Delta_{R}$
are Krull-Schmidt graded $R$-categories.
Since $\Delta$-filtered objects are generated by global sections, from the discussion above we deduce that $\calL$, $\Gamma$ are mutually inverse equivalences
of graded $R$-categories
between $\bar\bfF_{R}^\Delta$ and $\bar\bfZ_{R}^\Delta$.

For each $M\in\bar\bfZ_{R}^\Delta$ we write
$M^x=\calL(M)^x,$ $M_{[x]}=\calL(M)_{[x]},$
$M_x=\calL(M)_x,$
$M_{\prec x}=\calL(M)_{\prec x},$
$M_{\partial x}=\calL(M)_{\partial x}$
and $M_h=\calL(M)_h.$

\smallskip

The categories $\bar\bfF^\Delta_{R}$ and $\bar\bfZ^\Delta_{R}$
are exact categories with the exact structures defined as follows.
Following \cite[lem.~4.5]{F3}, we say that a sequence 
$0\to M'\to M\to M''\to 0$ 
in $\bar\bfZ_{R}^\Delta$ 
is exact
if and only if  the  sequence of (free) $R$-modules 
$0\to M'_{[x]}\to M_{[x]}\to M''_{[x]}\to 0$
is exact  for each $x\in I$.
This yields an exact structure on $\bar\bfZ^\Delta_{R}$.  We transport it
to an exact structure on $\bar\bfF^\Delta_{R}$ via the equivalence $(\calL,\Gamma)$.

\vspace{2mm}

\begin{rk}
\label{rk:exact}
The forgetful functor $\bar\bfZ_R^\Delta\to\bar \bfZ_R$ 
is exact by \cite[lem.~2.12]{F4}. Here $\bar \bfZ_R$ is equipped with the exact structure 
naturally associated with its structure of abelian category. To avoid confusion we'll call it the \emph{stupid}
exact structure.
\end{rk}

\vspace{2mm}

\begin{rk}
\label{rk:MI}
Let $M\in\bar\bfZ_R^\Delta.$ Then, we have $M\simeq\Gamma\calL(M)$ and $M_x=\calL(M)_x.$
For each $J\subset I,$ let $M_J$ be the image of the obvious map
$M=\Gamma \calL (M)\to\bigoplus_{x\in J}M_x,$ see also \cite[def.~2.7]{F4}.
It is a graded $R$-module, and
we have $M_J=\calL(M)(J)$ by \cite[lem.~3.4]{F3}.
Since $\bar Z_{R,x}=R$ for all $x$, the set $(\bar Z_R)_J$ is a graded 
$R$-subalgebra of $R^{\oplus J}$.
So $M_J$ is a graded $(\bar Z_R)_J$-module.
\end{rk}

\vspace{2mm}

We'll use two different notions of projective objects in the exact category $\bar\bfZ_{R}^\Delta$,
compare \cite[sec.~3.8]{J}.

\vspace{2mm}

\begin{df}
$(a)$ A module $M\in\bar\bfZ_{R}^\Delta$ 
is {\it projective} if the functor $\Hom_{\bar\bfZ_{R}}(M,\bullet)$ on $\bar\bfZ_{R}^\Delta$
maps short exact sequences to short exact sequences.

$(b)$ A module $M\in\bar\bfZ_{R}$ 
is {\it F-projective} if 
$\calL(M)$ is flabby, 
$M_x$ is a projective graded $R$-module for each $x$,
$M_h=M_{h'}/\alpha_hM_{h'}$ and
$\rho_{h',h}$ is the canonical map for each $h$.
\end{df}

\vspace{2mm}

\begin{rk}\label{rk:Fprojective}
If $M\in\bar\bfZ_{R}^\Delta$ is F-projective, then it is projective \cite[prop.~5.1]{F3}.
Note that loc.~cit.~uses \emph{reflexive} graded $R$-sheaves.
However, $\Delta$-filtered graded $R$-sheaves are automatically reflexive by 
\cite[sec.~4]{F3}.
\end{rk}

\vspace{2mm}

We say that an object $M\in\bar\bfZ_{R}^\Delta$ 
is {\it tilting} if  
the contravariant functor $\Hom_{\bar\bfZ_{R}}(\bullet,M)$ on $\bar\bfZ_{R}^\Delta$
maps short exact sequences to short exact sequences.

For each graded $R$-module $M,$ let $M^*$ be its {\it dual graded $R$-module}, i.e.,
we set
$M^*=\bigoplus_i(M^*)^i$
with $(M^*)^i=\ghom_{R}(M,R\langle i\rangle)$.
Here $\ghom_{R}$ is the Hom's space of graded $R$-modules.
Since $\bar Z_{R}$ is commutative, 
the graded dual $R$-module $M^*$ of a graded $\bar Z_R$-module $M$ is a 
graded $\bar Z_{R}$-module.

As above, let the symbols $\Tilt$ and $\Proj$ denote the sets of indecomposable tilting and projective objects.
 By \cite[sec.~4.6]{F3}, the following holds.
 
 \vspace{2mm}

\begin{prop}\label{prop:F3}
The duality $D:\bar\bfZ_{R}\to\bar\bfZ_{R}$ such that $M\mapsto M^*$
restricts to an exact equivalence 
$\bar\bfZ_{R,\preccurlyeq}^\Delta\to(\bar\bfZ_{R,\succcurlyeq}^\Delta)^\op$.
In particular, it yields bijections
$\Proj(\bar\bfZ_{R,\preccurlyeq}^\Delta)\to\Tilt(\bar\bfZ_{R,\succcurlyeq}^\Delta)$
and
$\Tilt(\bar\bfZ_{R,\preccurlyeq}^\Delta)\to\Proj(\bar\bfZ_{R,\succcurlyeq}^\Delta).$
\qed
\end{prop}

\vspace{.5mm}

We define the {\it support} of a graded 
$R$-sheaf $M$ on $\calG$ to be the set
$\supp(M)=\{x\in I\,;\,M_x\neq 0\}.$
We can now turn to the following important definition.

\vspace{2mm}

\begin{prop-df}[\cite{F3}] \label{df:3.9}
Let $(\calG,\preccurlyeq)$ be an ordered GKM-graph.
There is a unique object 
$\bar B_{R,\preccurlyeq}(x)$ in $\bar\bfZ_{R}$ 
which is indecomposable, F-projective,
supported on the coideal $\{\succcurlyeq\! x\}$ and
with $\bar B_{R,\preccurlyeq}(x)_x=R$. 
We call $\bar B_{R,\preccurlyeq}(x)$ a graded {\it BM-sheaf}.
\end{prop-df}

\vspace{2mm}

We may abbreviate 
$\bar B_{R}(x)=\bar B_{R,\preccurlyeq}(x)$ and
$\bar C_{R}(x)=\bar C_{R,\preccurlyeq}(x)=D(\bar B_{R,\succcurlyeq}(x)).$

\vspace{2mm}

\begin{rk} \label{rk:BM} 
The existence and unicity of graded BM-sheaves is proved in \cite[thm.~5.2]{F3}
using the Braden-MacPherson algorithm \cite[sec.~1.4]{BM}. 
See also \cite[thm.~6.3]{FW}.
The construction of
$\bar B_{R}(x)$ is as follows. 
\begin{itemize}
\item
Set $\bar B_{R}(x)_y=0$ for $y\not\succcurlyeq x$.
\item
Set $\bar B_{R}(x)_x=R$.
\item
Let $y\succ x$ and suppose we have already constructed $\bar B_{R}(x)_z$ and
$\bar B_{R}(x)_h$ for  any $z,h$ such that
$y\succ z,h',h''\succcurlyeq x$. For $h\in d_y$ set 
\begin{itemize}
\item $\bar B_{R}(x)_h=\bar B_{R}(x)_{h'}/\alpha_h\bar B_{R}(x)_{h'}$ 
and $\rho_{h',h}$ is the canonical map,
\item $\bar B_{R}(x)_{\partial y}=
\Im(\rho_{\prec y,d_y})\subset \bar B_{R}(x)_{d_y},$ 
\item $\bar B_{R}(x)_y$ is the projective cover of the graded $R$-module $\bar B_{R}(x)_{\partial y}$,
\item $\rho_{y,h}$ is the composition of the projective cover map
$\bar B_{R}(x)_y\to\bar B_{R}(x)_{\partial y}$ with the obvious projection
$\bar B_{R}(x)_{d_y}\to\bar B_{R}(x)_h$.
\end{itemize}
\end{itemize}
\end{rk}

\vspace{2mm}

\begin{prop}
\label{prop:decomposition}
An F-projective object in $\bar\bfZ_{R}$ is a direct sum of objects of the form 
$\bar B_{R}(x)\langle j\rangle$ with $x\in I$ and $j\in\bbZ$.
\end{prop}

\begin{proof}
See  \cite[prop.~3.9]{J}.\end{proof}

\vspace{2mm}

Next, following \cite[sec.~4.5]{F3}, we introduce the following.

\vspace{2mm}

\begin{prop-df} %[\cite{F3}] 
There is a unique object $\bar V_R(x)\in\bar\bfZ_R$ that is isomorphic to
$R$ as a graded $R$-module and on
which the tuple $(z_y)\in\bar Z_R$ acts by multiplication with $z_x$.
We call it a {\it graded Verma-sheaf}. 
\end{prop-df}

\vspace{2mm}

\begin{prop}\label{prop:vermasheaf}
(a) The graded Verma-sheaves are $\Delta$-filtered and self-dual. 

(b) We have $\bar V_R(x)^y=\bar V_R(x)_y=\bar V_R(x)_{[y]}=R$
if $x=y$ and 0 else.

(c) For $M\in\bar\bfZ_R$ we have
$$\Hom_{Z_R}
\bigl(\bar V_R(x)\langle i\rangle,M\bigr)=M^x\langle-i\rangle,$$
$$\Hom_{Z_R}
\bigl(M,\bar V_R(y)\langle j\rangle\bigr)=(M_y)^*\langle j\rangle,$$
$$
\hom_{\bar Z_R}
\bigl(\bar V_R(x)\langle i\rangle,\bar V_R(y)\langle j\rangle\bigr)=
\delta_{x,y}\,R^{i-j}.$$
\end{prop}

\vspace{.5mm}

\begin{proof}
Obvious, see e.g.,
\cite[sec.~3.11]{J}.
\end{proof}

\vspace{2mm}

\begin{rk}
\label{rk:filtration}
%(d) A $\Delta$-filtered graded $A$-sheaf $M$ is an extension of  {\it shifted Verma-sheaves}. }
%By Proposition \ref{prop:vermasheaf}$(a),(d)$ a graded $A$-sheaf is 
%$\Delta$-filtered if and only if it has a finite filtration by
%graded $\bar Z_A$-submodules whose subquotients
%are shifted Verma-sheaves. 
A $\Delta$-filtered graded $R$-sheaf $M$ has a filtration 
$0=M_0\subseteq M_1\subseteq\dots\subseteq M_n=M$ 
by graded $\bar Z_R$-submodules with
\cite[rk.~4.4, sec.~4.5]{F3}
$$\gathered
\bigoplus_{r=1}^nM_{r}/M_{r-1}=\bigoplus_y M_{[y]},\quad
M_{r}/M_{r-1}
=\bar V_R(y_r)\langle i_r\rangle,\quad
r\leqslant s\Rightarrow y_s\preccurlyeq y_r.
\endgathered$$
In particular, a $\Delta$-filtered graded $\bar Z_R$-module is 
free and finitely generated as a graded $R$-module.
\end{rk}

\vspace{2mm}

\begin{rk}
\label{rk:stalk/costalk}
From Proposition \ref{prop:vermasheaf}$(c)$ we deduce that
$(M_y)^*=(M^*)^y.$
\end{rk}

\vspace{2mm}

\subsection{Ungraded version and base change}
\label{rk:basechange}

Assume that $R$ is the localization of $S$ with respect
to some multiplicative set.
Then, we define an
$R$-sheaf over $\calG$, its space of global sections, the structural algebra
$Z_R$, the categories  $\bfF_{R}$, $\bfZ_{R}$, etc., in the obvious way, 
by forgetting the grading in the definitions above. In particular, we define as above an exact category $\bfZ^\Delta_{R}$.

Now, assume that $R=S$. 
Forgetting the grading yields a faithful and 
faithfully exact functor $\bar\bfZ^\Delta_{R}\to\bfZ^\Delta_{R}$.
See, e.g., \cite{Ishi}, for details on faithfully exact functors. Let
$V_R(x)$, $B_{R}(x)$,  $C_{R}(x)$ be the images of $\bar V_R(x)$, $\bar B_{R}(x)$,  $\bar C_{R}(x)$.
%see \cite[sec.~3.5]{J}. 
We call $V_R(x)$ a Verma-sheaf and $B_{R}(x)$ a BM-sheaf.

\vspace{1mm}

Next, for each morphism of  $S$-algebras 
$R\to R'$ we consider the {\it base change} functor 
$\bullet\otimes_{R}R':R\bfmod\to R'\bfmod$ given by $M\mapsto R'M=M\otimes_RR'$. 

Assume that $R'$ is the localization of $R$ with respect
to some multiplicative subset.
Then, the base change yields functors
$\bullet\otimes_{R}R':\bfF_{R}\to\bfF_{R'}$ and
$\bullet\otimes_{R}R':\bfZ_{R}\to\bfZ_{R'}$.
These functors commute with $\Gamma$, $\calL$ and the canonical map 
$R'\Hom_{\bfZ_R}(M,N)\to \Hom_{\bfZ_{R'}}(R'M,R'N)$
is an isomorphism, see  \cite[sec.~2.7, 2.18, 3.15]{J}.
Further, the base change commutes with the functor
$M\mapsto M_{[x]}$, it preserves the flabby sheaves, the Verma sheaves,
the F-projective ones and it yields an exact and 
faithful functor $\bfZ_{R}^\Delta\to\bfZ_{R'}^\Delta$, see \cite[sec.~3.15]{J}.

\smallskip
Assume now that $R$ is the localization of $S$ with respect
to some multiplicative set. Then, we define $V_{R}(x)=R V_S(x),$
$B_{R}(x)=R B_S(x)$ and $C_{R}(x)=R C_S(x).$
Note that $B_{R}(x),$ $C_{R}(x)$
are F-projective by \cite[sec.~3.8]{J}, and  that
if $R=S_0$ then $B_{R}(x),  C_{R}(x)$ are indecomposable by
\cite[sec.~3.16]{J}. Further, any $F$-projective $R$-sheaf in $\bfF_R$ is a direct sum of objects of the form $B_{R}(x)$ with $x\in I$, see \cite[prop.~3.13]{J}.

\smallskip

Forgetting the grading and taking the base change, we get the functor
$\varepsilon:S\bfgmod\to S_0\bfmod$. It is faithful and faithfully exact. Indeed, if $\varepsilon(M)=0$, then there is an element $f\in S\setminus\frakm$ such that $fM=0$.
Since $M$ is graded, this implies that $M=0$. Hence $\varepsilon$ is faithful,
and by \cite[thm.~1.1]{Ishi} it is also faithfully exact. Finally, note that a graded $S$-module $M$ is free over $S$ 
(as a graded module) if and only if $\varepsilon(M)$
is free over $S_0$. 

\smallskip
Similarly, consider the functor $\varepsilon:\bar\bfZ_{S}\to\bfZ_{S_0}$ given by forgetting the grading and taking the base change.
We have the following lemma.

\begin{lemma}\label{lem:forget}
(a) The functor $\varepsilon$ commutes with the functor $M\mapsto M_{[x]}.$

(b) The functor $\varepsilon:\bar\bfZ_{S}\to\bfZ_{S_0}$ is faithful and faithfully exact.

(c) For each $M\in\bar\bfZ_{S},$ we have
$M\in\bar\bfZ_{S}^\Delta$ if and only if
$\varepsilon(M)\in\bfZ^\Delta_{S_0}$.
\qed
\end{lemma}

\vspace{3mm}

\subsection{The moment graph of $\bfO$}
\label{sec:3.4}
Fix a parabolic type $\mu\in\calP$ and fix $w\in\widehat W$.

\smallskip
Unless specified otherwise, we'll set $V=\bft$.
In this section, we define a moment graph ${}^w\calG_{\mu,\pm}$
over $\bft$ which is naturally associated with the truncated categories ${}^w\bfO_{\mu,\pm}$.  
Then, we consider its category of (graded) $R$-sheaves over ${}^w\calG_{\mu,\pm}$.
In the graded case we'll assume that $R=S$, with the grading above.
In the ungraded one, we'll assume that $R=S$ or $S_0$, see the previous section for details.

\vspace{2mm}

\begin{df}
\label{def:momentgraph}
Let ${}^w\calG_{\mu,\pm}$
be the moment graph over $\bft$ 
whose set of vertices is 
${}^w\!I_{\mu,\pm}$, 
with an edge between $x,y$ labelled by $k\,\check\alpha$  if there is an affine
reflection $s_{\alpha}\in\widehat W$ such that $s_\alpha y\in xW_\mu$.
We equip ${}^w\calG_{\mu,+}$ with the partial order $\preccurlyeq$ given by
the opposite Bruhat order,
and we equip ${}^w\calG_{\mu,-}$ with the partial order $\preccurlyeq$ given by
the Bruhat order.
\end{df}

\smallskip

Let ${}^w\bar Z_{S,\mu,\pm}$
be the graded structural algebra of 
${}^w\calG_{\mu,\pm}$. Let
${}^w\bar\bfZ_{S,\mu,\pm}$
be the category of graded ${}^w\bar Z_{S,\mu,\pm}$-modules which are finitely generated 
and torsion free over $S$. 
Let ${}^w\bar\bfZ_{S,\mu,\pm}^\Delta$ 
be the category of $\Delta$-filtered graded $S$-sheaves 
on ${}^w\calG_{\mu,\pm,\preccurlyeq}$.
Let $\bar V_{S,\mu,\pm}(x)\in{}^w\bar\bfZ_{S,\mu,\pm}$ be the (graded) Verma-sheaf 
whose stalk at $x$ is $S$.
We'll use a similar notation for all other objects attached to ${}^w\calG_{\mu,\pm}$.

\smallskip

We'll use the same notation as in Section \ref{rk:basechange}. 
For example ${}^w\!B_{S,\mu,\pm}(x)$ is the nongraded analogue of ${}^w\!\bar B_{S,\mu,\pm}(x)$, and ${}^w\!B_{R,\mu,\pm}(x)=R{}^w\!B_{S,\mu,\pm}(x)$.

\smallskip

We also set  ${}^w\!\bar Z_{k,\mu,\pm}=k {}^w\!\bar Z_{S,\mu,\pm}$,
${}^w\! Z_{k,\mu,\pm}=k {}^w\! Z_{S_0,\mu,\pm}$,
${}^w\bar\bfZ_{k,\mu,\pm}={}^w\!\bar Z_{k,\mu,\pm}\bfgmod$
and ${}^w\bfZ_{k,\mu,\pm}={}^w\!Z_{k,\mu,\pm}\bfmod$.
In the graded case, we put
${}^w\!\bar B_{k,\mu,\pm}(x)=
k{}^w\!\bar B_{S,\mu,\pm}(x)$.
In the non graded case, we put
${}^w\! B_{k,\mu,\pm}(x)=
k{}^w\! B_{S_0,\mu,\pm}(x)$.
To unburden the notation we may omit the subscript $k$ and write
${}^w\! Z_{\mu,\pm}={}^w\! Z_{k,\mu,\pm}$, ${}^w\bfZ_{\mu,\pm}={}^w\bfZ_{k,\mu,\pm}$, etc.
\vspace{2mm}

%Note that for ${}^w\calG_{\phi,-}$ we have 
%$\sharp d_x=l(x)$ and $\sharp u_x \geqslant l(w)-l(x)$, while for
%$$\gathered
%\sharp d_x=\sharp\{\beta\in\widehat\Pi^+_\re\,;\,s_\beta x<x\}=l(x)\\
%\sharp u_x=\sharp\{\beta\in\widehat\Pi^+_\re\,;\,w\geqslant s_\beta x>x\}
%\geqslant l(w)-l(x).
%\endgathered$$ 
%${}^w\calG_{\phi,+}$ we have 
%$\sharp u_x=l(x)$ and $\sharp d_x \geqslant l(w)-l(x)$.

\vspace{2mm}

\begin{prop}
\label{prop:BMD}
The (graded) BM-sheaves on ${}^w\calG_{\mu,\pm}$
are $\Delta$-filtered.
\end{prop}

\begin{proof}
The (graded) BM-sheaves on ${}^w\calG_{\mu,-}$
are $\Delta$-filtered by \cite[thm.~5.2]{F3}.
We claim that the (graded) BM-sheaves on ${}^w\calG_{\mu,+}$
are also $\Delta$-filtered. 

By Section \ref{rk:basechange}, a graded $S$-sheaf $M$ is 
$\Delta$-filtered if and only if the $S_0$-sheaf $\varepsilon(M)$
is $\Delta$-filtered.
Now, by Proposition \ref{prop:foncteurV}$(c)$ below, the 
BM-sheaves on ${}^w\calG_{\mu,+}$ are $\Delta$-filtered for $R=S_0$.
This proves our claim.
\end{proof}

\vspace{2mm}

Let $l(x)$ be the length of an element $x\in\widehat W$.
From Propositions \ref{df:3.9}, \ref{prop:BMD}, we deduce the following.

\vspace{2mm}

\begin{prop}\label{prop:B} 
For each $x\in{}^w\!I_{\mu,\pm},$ there is 
a unique $\Delta$-filtered graded $S$-sheaf 
${}^w\!\bar B_{S,\mu,\pm}(x)\in{}^w\bar\bfZ_{S,\mu,\pm}^\Delta$
which is indecomposable, F-projective,
supported on the coideal $\{\succcurlyeq\! x\}$ and whose
stalk at $x$ is equal to
$S\langle\pm l(x)\rangle$.
\qed
\end{prop}

\vspace{2mm}

Assume that $w\in I_{\mu,+}$ and let $v=w_-\in I_{\mu,-}$. 
From the anti-isomorphism of posets ${}^w\!I_{\mu,+}\simeq{}^v\!I_{\mu,-}$ in Corollary \ref{lem:C2}(c), 
we  deduce that ${}^w\calG_{\mu,+}$ and ${}^v\calG_{\mu,-}$ are the same moment graph with opposite orders. 
In particular, we have ${}^w\bar Z_{S,\mu,+}={}^v\bar Z_{S,\mu,-}$ as graded algebras. 
Further, by Proposition \ref{prop:F3}, there is an exact equivalence
$$D: {}^w\bar\bfZ_{S,\mu,+}^\Delta\to({}^v\bar\bfZ_{S,\mu,-}^\Delta)^\op.$$

The same is true if $+$ and $-$ are switched everywhere. We define
${}^w\bar C_{S,\mu,\pm}(x)=D({}^v\!\bar B_{S,\mu,\mp}(y))$
with $y=x_{\mp}$, $v=w_\mp$ for each $x\in {}^w\!I_{\mu,\pm}$.
It is a $\Delta$-filtered graded $S$-sheaf on ${}^w\calG_{\mu,\pm}$
which is indecomposable, tilting,
supported on the ideal $\{\preccurlyeq\! x\}$ and
whose costalk at $x$ is equal to $S\langle\pm l(x)\rangle$.
We abbreviate
${}^w\!\bar B_{S,\mu,\pm}=\bigoplus_x{}^w\!\bar B_{S,\mu,\pm}(x)$
and
${}^w\bar C_{S,\mu,\pm}=\bigoplus_x{}^w\bar C_{S,\mu,\pm}(x).$

\vspace{2mm}

By Remark \ref{rk:filtration} we have the following lemma.

\vspace{2mm}

\begin{lemma}
\label{rk:filtration2}
(a) The graded $S$-sheaf ${}^w\!\bar B_{S,\mu,\pm}(x)$ is
filtered by graded Verma-sheaves. The top term in this filtration is $\bar V_{S,\mu,\pm}(x)\langle\pm l(x)\rangle$, which is a quotient of
${}^w\!\bar B_{S,\mu,\pm}(x).$
The other subquotients are of the form
$\bar V_{S,\mu,\pm}(y)\langle j\rangle$ with $y\succ x$ and $j\in\bbZ$. 

(b) The graded $S$-sheaf ${}^w\bar C_{S,\mu,\pm}(x)$ 
is filtered by graded Verma-sheaves. The first term in this filtration is the
sub-object
$\bar V_{S,\mu,\pm}(x)\langle\pm l(x)\rangle.$
The other subquotients are of the form
$\bar V_{S,\mu,\pm}(y)\langle j\rangle$ with $y\prec x$ and $j\in\bbZ$. 
\end{lemma}

\vspace{2mm}

Since ${}^w\bar Z_{S,\mu,+}={}^v\bar Z_{S,\mu,-}$ we may also regard 
${}^v\bar C_{S,\mu,-}(x_-)=D({}^w\!\bar B_{S,\mu,+}(x))$ as a graded ${}^w\bar Z_{S,\mu,+}$-module. 
The following proposition says that ${}^w\!\bar B_{S,\mu,+}(x)$ is self-dual.

\vspace{2mm}

\begin{prop}
\label{prop:duality}
For each $x\in{}^w\!I_{\mu,+},$ we have 
${}^w\!\bar B_{S,\mu,+}(x)={}^v\bar C_{S,\mu,-}(x_-)$ as graded ${}^w\bar Z_{S,\mu,+}$-modules.
\end{prop}

\vspace{.5mm}

\begin{proof}
If $\mu=\phi$ is regular 
then the claim is \cite[thm.~6.1]{F4}.

Assume now that $\mu$ is not regular.
We abbreviate $\bar B_\mu(x)={}^w\!\bar B_{S,\mu,+}(x)$.
We must prove that $D(\bar B_\mu(x))=\bar B_\mu(x).$ 
The regular case implies that $D(\bar B_\phi(x))=\bar B_\phi(x)$.
We can assume that $w\in I_{\mu,-}$ and that $x\in {}^w\!I_{\mu}$.

From Proposition \ref{prop:theta}$(e)$, $(f)$ below, we deduce that
$$\bigoplus_{y\in W_\mu}D(\bar B_\mu(x))\langle 2l(y)-l(w_\mu)\rangle\oplus D(M)=
\bigoplus_{y\in W_\mu}\bar B_\mu(x)\langle l(w_\mu)-2l(y)\rangle\oplus 
M,$$
where $M$ is a direct sum of objects of the form
$\bar B_\mu(z)\langle j\rangle$ with $z< x$.
We have also
$D(\bar B_\mu(e))=\bar B_\mu(e)$.
Thus, by induction, we may assume that
$D(\bar B_\mu(z))=\bar B_\mu(z)$ for all $z<x$, and the equality above implies that
$D(\bar B_\mu(x))=\bar B_\mu(x).$ 
\end{proof}

\vspace{2mm}

\begin{rk} The graded sheaf ${}^w\!\bar B_{S,\mu,-}(x)$
is not self-dual, i.e., ${}^w\!\bar B_{S,\mu,-}(x)$ is not isomorphic to ${}^v\!\bar C_{S,\mu,+}(x_+)$ in general.
\end{rk}

\vspace{3mm}

Let $P_{y,x}^{\mu,q}(t)=\sum_iP_{y,x,i}^{\mu,q}\,t^i$ be
Deodhar's parabolic Kazhdan-Lusztig polynomial ``of type $q$''
associated with the
parabolic subgroup $W_\mu$ of $\widehat W$.

Let $Q_{x,y}^{\mu,-1}(t)=\sum_iQ_{x,y,i}^{\mu,-1}\,t^i$ be
Deodhar's inverse parabolic Kazhdan-Lusztig polynomial 
``of type $-1$''. 

We use the notation in \cite[rk.~2.1]{KT2}, which is not the usual one. 
We abbreviate
$P_{y,x}=P_{y,x}^{\phi,q}$ 
and
$Q_{y,x}=Q_{y,x}^{\phi,-1}$.

\vspace{2mm}

\begin{prop}\label{prop:3.26} 
For each $x,y\in{}^w\!I_{\mu,+},$ we have graded $S$-module isomorphisms

(a) ${}^w\!\bar B_{S,\mu,+}(x)_{y}
=\bigoplus_{i\geqslant 0}(S\langle l(x)-2i\rangle)^{\oplus P^{\mu,q}_{y,x,i}},$

(b) ${}^w\!\bar B_{S,\mu,+}(x)_{[y]}
=\bigoplus_{i\geqslant 0}(S\langle 2l(y)-l(x)+2i\rangle)^{\oplus P^{\mu,q}_{y,x,i}}.$
\end{prop}

\vspace{.5mm}

\begin{proof} Any finitely generated projective $S$-module is free, and that
the $S$-module ${}^w\!\bar B_{S,\mu,+}(x)_{y}$ is projective because ${}^w\!\bar B_{S,\mu,+}(x)$
is F-projective.
Hence, part $(a)$ follows from Proposition \ref{prop:loc1} below and 
\cite[thm.~1.4]{KT2}. Part $(b)$ follows from $(a)$ and Proposition 
\ref{prop:duality} as in \cite[prop.~7.1]{F4} (where it is proved for $\mu$ regular).
More precisely, write
$\bar V(x)=\bar V_{S,\phi}(x),$
$\bar B(x)={}^w\!\bar B_{S,\mu,+}(x)$
and $Z={}^w\! Z_{S,\mu,+}.$
Then, we have
$$\bar B(x)^y=\Ker(\rho_{y,e_y})\subset\bar B(x)_{[y]}=\Ker(\rho_{y,d_y})\subset \bar B(x)_y.$$
In particular, for
$\alpha_{u_y}=\prod_{h\in u_y}\alpha_h$,
we have
$\alpha_{u_y}\,\bar B(x)_{[y]}\subset\bar B(x)^y.$

Next, the graded $S$-module $\bar B(x)_y$ is free and
$\bar B(x)_h=\bar B(x)_y/\alpha_h \bar B(x)_y$ 
for $h\in u_y$ because $\bar B(x)$ is F-projective. 
Thus we have
$\bar B(x)^y\subset\alpha_{u_y}\,\bar B(x)_y,$
because $\alpha_{h_1}$ is prime to $\alpha_{h_2}$ in $S$ if $h_1\neq h_2$.

We claim that 
$\bar B(x)^y=\alpha_{u_y}\,\bar B(x)_{[y]}.$
Let $b\in\bar B(x)^y$.
Write $b=\alpha_{u_y}\, b'$ with $b'\in\bar B(x)_y$.
If $\rho_{y,h}(b')\neq 0$ with $h\in d_y$ then
$\rho_{y,h}(b)=\alpha_{u_y}\,\rho_{y,h}(b')\neq 0,$ 
because the $\alpha_{u_y}$-torsion submodule of
$\bar B(x)_h$ is zero (we have
$\bar B(x)_h=\bar B(x)_{h'}/\alpha_h \bar B(x)_{h'}$,
the $S$-module $\bar B(x)_{h'}$ is free, 
and $\alpha_h$ is prime to $\alpha_{u_y}$).
This implies the claim.

In particular, we have
$\bar B(x)_{[y]}=\bar B(x)^y\langle 2l(y)\rangle$
because $\sharp u_y=l(y)$.
Hence,
by Remark \ref{rk:stalk/costalk}
and Proposition \ref{prop:duality} 
we have a graded $S$-module isomorphism
$$\aligned
(\bar B(x)_y)^*\langle 2l(y)\rangle=\bar B(x)^y\langle 2l(y)\rangle
=\bar B(x)_{[y]}.
\endaligned$$
\end{proof}

\vspace{2mm}

\begin{df}\label{def:3.23}
Consider
the graded $S$-algebra
${}^w\!\bar \scrA_{S,\mu,\pm}=
\End_{{}^w\! Z_{S,\mu,\pm}}\bigl({}^w\!\bar B_{S,\mu,\pm}\bigr)^{\op}$
and the graded $k$-algebra 
${}^w\!\bar \scrA_{\mu,\pm}=
k{}^w\!\bar \scrA_{S,\mu,\pm}.$
Set also ${}^w\!\scrA_{S_0,\mu,\pm}=
\End_{{}^w\! Z_{S_0,\mu}}\bigl({}^w\!B_{S_0,\mu,\pm}\bigr)^{\op}$
and let
${}^w\! \scrA_{\mu,\pm}$ be the 
$k$-algebra underlying ${}^w\!\bar \scrA_{\mu,\pm}.$
\end{df}

\vspace{2mm}

For each $x\in{}^w\!I_{\mu,\pm},$ let $1_x\in {}^w\!\bar \scrA_{\mu,\pm}$ be the idempotent
associated with the direct summand
${}^w\!\bar B_{S,\mu,\pm}(x)$
of ${}^w\!\bar B_{S,\mu,\pm}$.

\vspace{2mm}

\begin{prop}
\label{prop:2.24}
The graded $k$-algebra ${}^w\!\bar \scrA_{\mu,\pm}$ is basic. 
Its Hilbert polynomial is 
$$\gathered
P({}^w\!\bar \scrA_{\mu,+},t)_{x,x'}=
\sum_{y\leqslant x,x'}
P^{\mu,q}_{y,x}(t^{-2})\,P^{\mu,q}_{y,x'}(t^{-2})\,t^{l(x)+l(x')-2l(y)},
\\
P({}^w\!\bar \scrA_{\mu,-},t)_{x,x'}=
\sum_{y\geqslant x,x'}
Q^{\mu,-1}_{x,y}(t^{-2})\,Q^{\mu,-1}_{x',y}(t^{-2})\,t^{2l(y)-l(x)-l(x')}.
\endgathered$$
If $\mu=\phi$ then the matrix equation 
$P({}^w\!\bar \scrA_{\phi,-},t)\,P({}^v\!\bar \scrA_{\phi,+},-t)=1$ holds.
Here $v=w^{-1}$ and
the sets of indices of the matrices in the left and right factors
are identified  through the map $x\mapsto x^{-1}$.
\end{prop}

\vspace{.5mm}

\begin{proof}
First, note that if $M,B\in\bar\bfZ^\Delta_S$ 
and if $B$ is projective,
then the filtration $(M_r)$ of $M$ in Remark \ref{rk:filtration} yields a filtration
of the graded $S$-module $\Hom_{Z_S}(B,M)$ 
whose associated graded is, see Proposition \ref{prop:vermasheaf},
$$\aligned
\bigoplus_{r=1}^n\Hom_{Z_S}(B,M_r/M_{r-1})
=\bigoplus_{r=1}^n(B_{y_r})^*\langle i_r\rangle.
\endaligned$$

Further, for each $x',$ the graded BM-sheaf 
${}^w\!\bar B_{S,\mu,+}(x')$ is projective 
by Remark \ref{rk:Fprojective} and Proposition \ref{prop:B}.

Therefore, for each $x,x',$ 
there is a finite filtration of the graded $S$-module 
$\Hom_{{}^w\!Z_{S,\mu,+}}
\bigl({}^w\!\bar B_{S,\mu,+}(x'),{}^w\!\bar B_{S,\mu,+}(x)\bigr)$
whose associated graded is, see Proposition \ref{prop:3.26}, 
$$\aligned
&\bigoplus_{y}\bigoplus_{i\geqslant 0}
\bigl({}^w\!\bar B_{S,\mu,+}(x')_y\bigr)^*
\langle 2l(y)-l(x)+2i\rangle^{\oplus P^{\mu,q}_{y,x,i}}=
\\
&=\bigoplus_{y}\bigoplus_{i,i'\geqslant 0}
S\langle 2l(y)-l(x)-l(x')+2i+2i'\rangle^{\oplus P^{\mu,q}_{y,x,i}\,
P^{\mu,q}_{y,x',i'}}.
\endaligned$$

Thus we have a graded $k$-vector space isomorphism
$$1_x{}^w\!\bar \scrA_{\mu,+}1_{x'}=
\bigoplus_{y\leqslant x,x'}\bigoplus_{i,i'\geqslant 0}
k\langle 2l(y)-l(x)-l(x')+2i+2i'\rangle^{\oplus P^{\mu,q}_{y,x,i}\,P^{\mu,q}_{y,x',i'}},$$
where $y,x,x'$ run over ${}^w\!I_{\mu,+}$.

We deduce that
$$P({}^w\!\bar \scrA_{\mu,+},t)_{x,x'}=
\sum_{y\leqslant x,x'}P^{\mu,q}_{y,x}(t^{-2})\,P^{\mu,q}_{y,x'}(t^{-2})\,
t^{l(x)+l(x')-2l(y)}.$$
In other words, the  matrix equation
$P({}^w\!\bar \scrA_{\mu,+},t)=P_{\mu,+}(t)^T\,P_{\mu,+}(t)$ holds
with
$P_{\mu,+}(t)_{y,x}=P^{\mu,q}_{y,x}(t^{-2})\,t^{l(x)-l(y)}$
and
$x,y\in{}^w\!I_{\mu,+}.$

Note that
$P^{\mu,q}_{y,x,i}=0$ if $l(y)-l(x)+2i>0$ and that 
if $l(y)-l(x)+2i=0$ then
$P^{\mu,q}_{y,x,i}=0$ unless $y=x$.
Thus we have
$P({}^w\!\bar \scrA_{\mu,+},t)_{x,x'}\in\delta_{x,x'}+t\bbN[[t]].$
Hence the graded $k$-algebra
${}^w\!\bar\scrA_{\mu,+}$
is basic.

Set  
$Q_{\mu,-}(t)_{x,y}=Q^{\mu,-1}_{x,y}(t^{-2})\,t^{l(y)-l(x)}$ 
for each
$x,y\in{}^w\!I_{\mu,+}.$
The matrix equation
$P({}^w\!\bar \scrA_{\mu,-},t)=Q_{\mu,-}(t)\,Q_{\mu,-}(t)^T$ is proved in Proposition \ref{prop:IC} below.
Hence ${}^w\!\bar \scrA_{\mu,-}$ is also basic.

Next, for each $x,x',y\in {}^w\!I_{\mu,+}$ and $a=-1,q$ we have, see e.g., \cite[(2.39)]{KT2}
$$\sum_{x\leqslant y\leqslant x'}
(-1)^{l(y)-l(x)}Q^{\mu,a}_{x,y}(t)\,P^{\mu,a}_{y,x'}(t)=\delta_{x,x'}.$$
Further, if $\mu=\phi$ is regular then we have 
$P_{y,x}(t)=P^{\mu,q}_{y,x}(t)$ and
$Q_{x,y}(t)=Q^{\mu,q}_{x,y}(t)$. 
So, we have the matrix equations
$Q_{\phi,-}(t)\,P_{\phi,+}(-t)=1$,
$P_{\phi,+}(-t)\,Q_{\phi,-}(t)=
1.$
We deduce that
$$\gathered
P({}^w\!\bar \scrA_{\phi,-},t)P({}^w\!\bar \scrA_{\phi,+},-t)=
Q_{\phi,-}(t)\,Q_{\phi,-}(t)^T\,P_{\phi,+}(-t)^T\,P_{\phi,+}(-t)=1.
\endgathered$$
The matrix equation in the proposition follows easily, using the fact that
$P_{y,x}=P_{y^{-1},x^{-1}}$ and
$Q_{x,y}=Q_{x^{-1},y^{-1}}$. 
\end{proof}

\vspace{2mm}

\begin{rk}\label{rk:3.29}
The grading on ${}^w\!\bar \scrA_{\mu,\pm}$ is positive, because
$P^{\mu,q}_{y,x,i}=0$ if $l(y)-l(x)+2i>0$ and 
$P^{\mu,q}_{y,x,i}=0$ if $l(y)-l(x)+2i=0$ and $y\neq x$.

%The Hilbert polynomial $P({}^w\!\bar A_{\mu,+},t)$ can also be computed
%in the same way as the Hilbert polynomial $P({}^w\!\bar A_{\mu,-},t)$
%in Proposition \ref{prop:IC}.
\end{rk}

\vspace{2mm}

\begin{rk} The results in this section have obvious analogues in finite type.
In this case, we may omit the truncation and
Proposition \ref{prop:3.26}, Corollary \ref{cor:B1} yield
$$
\aligned
%\bar B_{A,\mu,+}(x)_{[y]}
%&=\bigoplus_{i\geqslant 0}A\{2l(y)-l(x)+2i\}^{\oplus P^{\mu,q}_{y,x,i}}\\
%\bar B_{A,\mu,-}(x)_{[y]}\{l(w_0)\}
%&=\bigoplus_{i\geqslant 0}A\{2l(w_0y)-l(w_0x)+2i\}^{\oplus Q^{\mu,-1}_{x,y,i}}\\
\bar B_{S,\mu,+}(x)_{y}
&=\bigoplus_{i\geqslant 0}S\langle l(x)-2i\rangle^{\oplus P^{\mu,q}_{y,x,i}}\\
\bar B_{S,\mu,-}(x)_y\{l(w_0)\}
&=\bigoplus_{i\geqslant 0}S\langle l(w_0x)-2i\rangle^{\oplus Q^{\mu,-1}_{x,y,i}},
\endaligned
$$
where $w_0$ is the longuest element in the Weyl group $W$.
Indeed, we have 
\begin{equation*}
\omega^*\bar B_{S,\mu,+}(w_0xw_\mu)=\bar B_{S,\mu,-}(x)\langle l(w_0w_\mu)\rangle,
\end{equation*}
where $\omega:\calG_{\mu,-}\to\calG_{\mu,+}$ 
is the ordered moment graph isomorphism induced by
the bijections
$I_{\mu,+}\to I_{\mu,+},$ $x\mapsto w_0xw_\mu$ and
$\bft\to\bft,$ $h\mapsto w_0(h).$
Note that \cite[prop.~2.4,2.6]{KT2} and
Kazhdan-Lusztig's inversion formula give
$Q^{\mu,-1}_{x,y}=P^{\mu,q}_{w_0yw_\mu,w_0xw_\mu}.$ 
\end{rk}

\vspace{5mm}

\subsection{Deformed category $\bfO$}
\label{sec:3.6}
Assume that  $R=S_0$ or $k$.  
If $R=k$ we may drop it from the notation.

We define
$\bfO_{R,\mu,\pm}$ as the category consisting of the $(\bfg,R)$-bimodules $M$  such
that
\begin{itemize}
\item $M=\bigoplus_{\lambda\in \bft^*} M_\lambda$ with 
$M_\lambda=\{m\in M\,;\,xm=(\lambda(x)+x)\cdot m,\,x\in\bft\}$,
%finitely generated over $A$,
\vskip1mm
\item $U(\bfb)\,(R\cdot m)$ 
is finitely generated over $R$ for each $m\in M$,
\vskip1mm
\item 
the highest weight of any simple subquotient 
is linked to $\oo_{\mu,\pm}$.
\end{itemize}
Here, for each $r\in R$ and each element $m$ of an $R$-module $M$,
the element $r\cdot m$ is the action of $r$ on $m$. Further, in the symbol
$\lambda(x)+x$ we view $x$ as the element of $R$ given by the image of $x$ under the canonical map
$S=S(\bft)\to R$.
The morphisms are the $(\bfg,R)$-bimodule homomorphisms.

Next, consider the following categories

\begin{itemize}

\item ${}^w\bfO_{R,\mu,-}$ is the 
Serre subcategory of $\bfO_{R,\mu,-}$ consisting
of the finitely generated modules
such that the highest weight of any simple subquotient 
is of the form $\lambda=x\bullet\oo_{\mu,-}$
with $x\preccurlyeq w$ and $x\in I_{\mu,-}$.

\item $^w\bfO_{R,\mu,+}$ is the category of the finitely generated 
objects in the Serre quotient of $\bfO_{R,\mu,+}$ 
by the Serre subcategory consisting of the modules
such that the highest weight of any simple subquotient 
is of the form $\lambda=x\bullet\oo_{\mu,+}$
with $x\not\succcurlyeq w$ and $x\in I_{\mu,+}$.
We'll use the same notation for a module in
$\bfO_{R,\mu,+}$ and a module in the quotient category $^w\bfO_{R,\mu,+}$,
hoping it will not create any confusion.
\end{itemize}

For each $\lambda\in\bft^*$ let $R_\lambda$ be the 
$(\bft,R)$-bimodule which is free of rank one over $R$ and such that the element
$x\in\bft$ acts by multiplication by the image of the element 
$\lambda(x)+x$ by the canonical map $S\to R$. 
The {\it deformed Verma module} 
with highest weight $\lambda$ is the $(\bfg,R)$-bimodule given by
$V_{R}(\lambda)=U(\bfg)\otimes_{U(\bfb)}R_\lambda.$

\vspace{2mm}

\begin{prop} 
\label{prop:deformedhw}
(a) The category ${}^w\bfO_{R,\mu,\pm}$ 
is a highest weight $R$-category. The standard
objects are the deformed Verma modules
$V_{R}(x\bullet\oo_{\mu,\pm})$ with $x\in{}^w\!I_{\mu,\pm}$.

(b) Assume that $w\in I^\nu_{\mu,\pm}.$ 
The Ringel equivalence gives an equivalence
$(\bullet)^\diamond:{}^w\bfO_{R,\mu,\pm}^\Delta\to{}^v\bfO_{R,\mu,\mp}^\nabla$
where $v=w_\mp$.
\end{prop}

\vspace{.5mm}

\begin{proof}
First, we consider the category ${}^w\bfO_{R,\mu,-}$.
The deformed Verma modules are split, i.e., their endomorphism ring is $R$.
Further, the $R$-category ${}^w\bfO_{R,\mu,-}$ is Hom finite. 
Thus we must check that ${}^w\bfO_{R,\mu,-}$ 
has a projective generator and that the projective modules are 
$\Delta$-filtered. Both statements follow from \cite[thm.~2.7]{F1}.

Next, we consider the category ${}^w\bfO_{R,\mu,+}$.
Once again it is enough to check that ${}^w\bfO_{R,\mu,+}$ 
has a projective generator and that the projective modules are 
$\Delta$-filtered. 
By \cite[thm.~2.7]{F1}, a simple module
$L(x\bullet\oo_{\mu,+})$ in $\bfO_{R,\mu,+}$
has a projective cover $P_R(x\bullet\oo_{\mu,+})$.
Note that the deformed category O in loc.~cit.~ is indeed a subcategory of
$\bfO_{R,\mu,+}$ containing all finitely generated modules. 
Since $P_R(x\bullet\oo_{\mu,+})$ is finitely generated, the functor
$\Hom_{\bfO_{R,\mu,+}}(P_R(x\bullet\oo_{\mu,+}),\bullet)$ commutes with direct limits.
Thus, since any module in $\bfO_{R,\mu,+}$ is the direct limit of its finitely generated submodules,
the module
$P_R(x\bullet\oo_{\mu,+})$ is again projective in
$\bfO_{R,\mu,+}$.
Now, a standard argument using  \cite[thm.~3.1]{Mi} shows that
the functor 
$$\Hom_{\bfO_{R,\mu,+}}\bigl(\bigoplus_x P_R(x\bullet\oo_{\mu,+}),\bullet\bigr),
\quad x\in{}^w\!I_{\mu,+},$$
factors to an equivalence of abelian $R$-categories 
${}^w\bfO_{R,\mu,+}\to A\bfmod,$ where $A$ is the finite projective $R$-algebra
$$A=\End_{\bfO_{R,\mu,+}}\bigl(\bigoplus_x P_R(x\bullet\oo_{\mu,+})\bigr)^\op.$$
Therefore by Proposition 
\ref{prop:ringel}(a) and \cite[thm.~4.15]{Ro} the category ${}^w\bfO_{R,\mu,+}$ is a highest weight category over $R$. 
This finishes the proof of $(a)$.

By \cite[sec.~2.6]{F2} there is an equivalence of exact categories
$\bfO^{\Delta}_{R,\mu,+}\iso\,(\bfO^{\Delta}_{R,\mu,-})^\op$
analogue to the one in Lemma \ref{lem:ringel-soergel}.
It maps $V_R(x\bullet\oo_{\mu,+})$
to $V_R(x_-\bullet\oo_{\mu,-}),$
and $P_R(x\bullet\oo_{\mu,+})$
to $T_R(x_-\bullet\oo_{\mu,-}).$
We deduce that $({}^v\bfO_{R,\mu,-})^\diamond={}^w\bfO_{R,\mu,+}$ for $w\in I_{\mu,+}$ and $v=w_-\in I_{\mu,-}$.
Now part (b) follows from generalities on Ringel equivalences, see Section \ref{sec:R}.

\end{proof}

\vspace{2mm}

%Following Section \ref{sec:5}, let $^w\bfO_{R,\mu,\pm}^\Delta\subset{}^w\bfO_{R,\mu,\pm}$ 
%be the full subcategory
%consisting of the $\Delta$-filtered modules. 
Let
${}^w\!P_{R}(x\bullet\oo_{\mu,\pm})$ and
${}^wT_{R}(x\bullet\oo_{\mu,\pm})$ be respectively the projective and the tilting objects in the highest weight categories ${}^w\bfO_{R,\mu,\pm}$ as given by Lemma \ref{lem:hwtbasic}.
Set ${}^wP_{R,\mu,\pm}=\bigoplus_x{}^wP_R(x\bullet\oo_{\mu,\pm})$ and
${}^wT_{R,\mu,\pm}=\bigoplus_x{}^wT_R(x\bullet\oo_{\mu,\pm})$, where $x$ runs over the set
${}^wI_{\mu,\pm}$.
Composing the Ringel duality and the BGG duality, we get an equivalence
\begin{equation}\label{eq:ringelbgg}
D=\dd\circ(\bullet)^\diamond:{}^w\bfO_{R,\mu,\pm}^\Delta\to({}^v\bfO_{R,\mu,\mp}^\Delta)^\op.
\end{equation}

\smallskip
Consider the base change functor 
${}^w\bfO_{S_0,\mu,\pm}\to{}^w\bfO_{k,\mu,\pm}$
such that
$M\mapsto kM.$
By Proposition \ref{prop:2.3} we have 
$k{}^w\!P_{S_0}(x\bullet\oo_{\mu,\pm})=
{}^w\!P(x\bullet\oo_{\mu,\pm})$ and
$k{}^wT_{S_0}(x\bullet\oo_{\mu,\pm})=
{}^wT(x\bullet\oo_{\mu,\pm})$
for each $x\in{}^w\!I_{\mu,\pm}$.
%\smallskip
%\begin{prop}
%\label{prop:projdeform}

%(a) For each $x\in{}^w\!I_{\mu,\pm}$ we have 
%$k{}^w\!P_{S_0}(x\bullet\oo_{\mu,\pm})=
%{}^w\!P(x\bullet\oo_{\mu,\pm})$ and
%$k{}^wT_{S_0}(x\bullet\oo_{\mu,\pm})=
%{}^wT(x\bullet\oo_{\mu,\pm}).$

%(b) Projective objects in
%${}^w\bfO_{R,\mu,\pm}^\Delta$ are finite direct sums
%of ${}^w\!P_{R}(\lambda)$'s.
%Tilting objects are finite direct sums
%of ${}^wT_{R}(\lambda)$'s. 

%(c) The base change takes projectives to projectives and tiltings to tiltings.
%The obvious map 
%$k\Hom(M,N)\to\Hom(kM,kN)$
%is invertible if
%$M\in {}^w\bfO_{S_0,\mu,\pm}^\Delta$ and $N\in {}^w\bfO_{S_0,\mu,\pm}^\nabla$, or if
%$M$ is projective.
%\end{prop}

%\begin{proof}
%The proposition follows from 
%Proposition \ref{prop:deformedhw} and from the general theory of highest weight categories over rings.
%For instance, the last claim in $(c)$ follows from Proposition \ref{prop:2.3}.
%\end{proof}

\vspace{3mm}

\subsection{The functor $\bbV$}

Assume that $R=S_0$ or $k$. 
There is a well defined functor
$$\bbV_{\! R}=\Hom_{{}^w\bfO_{R,\mu,-}}({}^w\!P(\oo_{\mu,-}),\bullet): {}^w\bfO_{R,\mu,-}^\Delta\to {}^w\bfZ_{R,\mu,-},$$
called the \emph{structure functor},
see \cite[sec.~3.4]{F2}.
%It is exact for the stupid exact structure on ${}^w\bfZ_{R,\mu}$, see
%Remark \ref{rk:exact} for the terminology.

At the positive level, following \cite[sec. before rem.~6]{F2}, we define the structure functor $\bbV_{\! R}$ as the composition
\[\xymatrix{{}^w\bfO_{R,\mu,+}^\Delta\ar[r]^{D}
&({}^v\bfO_{R,\mu,-}^\Delta)^\op\ar[r]^{\bbV_{\! R}}
&({}^v\bfZ_{R,\mu,-})^\op\ar[r]^{D}&{}^w\bfZ_{R,\mu,+}.
}\]
%It is an exact functor 
%${}^w\bfO_{R,\mu,+}^\Delta\to {}^w\bfZ_{R,\mu}$.

\vspace{2mm}

\begin{prop}
\label{prop:foncteurV}
Let $R=S_0$ or $k$. 
The following hold 

(a) $\bbV_{\! R}$ is exact on ${}^w\bfO_{R,\mu,-}$ and on ${}^w\bfO_{R,\mu,+}^\Delta$,

(b) $\bbV_{\! R}$ commutes with base change,

(c)
the functor $\bbV_{S_0}$ on ${}^w\bfO_{S_0,\mu,\pm}^\Delta$
is such that
\begin{itemize}
\item[(i)] 
$\bbV_{\! S_0}\circ D=D\circ \bbV_{\! S_0},$
\item[(ii)] $\bbV_{\! S_0}$ is an equivalence of exact categories 
${}^w\bfO_{S_0,\mu,\pm}^\Delta\to{}^w\bfZ_{S_0,\mu,\pm}^\Delta,$
\item[(iii)] for each $x\in {}^w\!I_{\mu,\pm}$ we have
\begin{itemize}
\item[-] 
$\bbV_{\! S_0}(V_{S_0}(x\bullet\oo_{\mu,\pm}))=V_{S_0,\mu,\pm}(x),$
\item[-] 
$\bbV_{\! S_0}({}^w\!P_{S_0}(x\bullet\oo_{\mu,\pm}))={}^w\!B_{S_0,\mu,\pm}(x),$
\item[-] 
$\bbV_{\! S_0}({}^wT_{S_0}(x\bullet\oo_{\mu,\pm}))={}^w\!C_{S_0,\mu,\pm}(x).$
\end{itemize}
\end{itemize}

%{\color{red}(d) $\bbV_k$ is fully-faithful on projectives in ${}^w\bfO_{\mu,+}^\Delta$.}
\end{prop}

\vspace{.5mm}

\begin{proof} 
Parts $(a),(b)$ and the first claim of  $(c)$ follow from the construction of $\bbV_R$.

The second claim of $(c)$ follows from \cite[thm.~7.1]{F3}. 
Note that we use here the exact structure
on ${}^w\bfZ_{S_0,\mu,\pm}^\Delta$ defined before Remark \ref{rk:exact}, see \cite[sec.~7.1]{F3}.

The first claim of $(c)$(iii) is given by \cite[prop.~2(2)]{F2}. 
The second one is proved in \cite[rk.~7.6]{F3}. The third one follows from the second and $(c)$(i).
\end{proof}

\vspace{2mm}

Now, we compare the $k$-algebra ${}^w\!A_{\mu,\pm}$ in 
Definition \ref{df:3.6} with the $k$-algebra ${}^w\!\scrA_{\mu,\pm}$ in Definition \ref{def:3.23}.

\vspace{2mm}

\begin{cor}\label{cor:2.31}
We have a $k$-algebra isomorphism
${}^w\!A_{\mu,\pm}\iso\, {}^w\!\scrA_{\mu,\pm}$ such that 
$1_x\mapsto 1_x$ for each $x\in{}^w\!I_{\mu,\pm}$.
\end{cor}

\vspace{.5mm}

\begin{proof}
By applying Proposition \ref{prop:2.3} to base change functor ${}^w\bfO_{S_0,\mu,\pm}\to{}^w\bfO_{k,\mu,\pm}$
, we get
$
{}^w\!A_{\mu,\pm}
=k\End_{{}^w\bfO_{S_0,\mu,\pm}}({}^w\!P_{S_0,\mu,\pm})^{\op}.
$
Thus, Proposition \ref{prop:foncteurV}$(c)$ yields
$$
{}^w\!A_{\mu,\pm}
=k\End_{{}^w\!Z_{S_0,\mu,\pm}}({}^w\!B_{S_0,\mu,\pm})^{\op}
=k\End_{{}^w\!Z_{S,\mu,\pm}}({}^w\!B_{S,\mu,\pm})^{\op}
={}^w\!\scrA_{\mu,\pm}.$$
\end{proof}

\vspace{2mm}

%\begin{rk}
%The highest weight $A$-category 
%${}^w\bfO_{A,\mu,\pm}$ does not depend on the choice of 
%$\oo_{\mu,\pm}$ and $e$
%but only on $\mu$, $\nu$, see \cite[thm.~11]{F2}. 
%\end{rk}

\vspace{2mm}

\subsection{Translation functors in $\bfO$}
\label{sec:translationO}
Assume that $R=S_0$ or $k$.

Fix positive integers $d,$ $e$.
Assume that
$\oo_{\mu,-}$, $\oo_{\phi,-}$ have the level $-e-N$, $-d-N$ respectively.
Thus, the integral affine weights $\oo_{\mu,+}$, $\oo_{\phi,+}$ have the level $e-N$, $d-N$ respectively.
In particular, the integral affine weight $\oo_{\mu,\pm}-\oo_{\phi,\pm}$ has the level
$\pm (e-d)$. Assume that $e-d\neq\mp N$, hence we have $\pm (e-d)\neq -N$.
The  affine weight $\oo_{\mu,\pm}-\oo_{\phi,\pm}$ is positive if $\pm( e-d)>-N$ and it is negative else, see
Section \ref{sec:affine} for the terminology.

First, we consider the translation functors on the (deformed) $\Delta$-filtered modules.
We have the following.

\vspace{2mm}

\begin{prop} 
\label{prop:translation1}
Let $R=S_0$ or $k$.
Let $w\in\widehat W$ and $z\in I_{\mu,-}$. Assume that $wW_\mu=zW_\mu$ and that $\pm (e-d)>-N$.
We have $k$-linear functors 
$T_{\phi,\mu}:{}^{z}\bfO_{R,\phi,\pm}^\Delta
\to{}^w\bfO_{R,\mu,\pm}^\Delta$ and
$T_{\mu,\phi}:{}^w\bfO_{R,\mu,\pm}^\Delta
\to{}^{z}\bfO_{R,\phi,\pm}^\Delta$
such that

(a) $T_{\phi,\mu}$, $T_{\mu,\phi}$ 
are exact,

(b) $T_{\phi,\mu}$, $T_{\mu,\phi}$ 
are bi-adjoint,

(c) $T_{\phi,\mu}$, $T_{\mu,\phi}$ 
commute with base change.
\end{prop}

\vspace{.5mm}

\begin{proof} The functors $T_{\phi,\mu}$, $T_{\mu,\phi}$ 
on $\bfO_{R,\phi,\pm}^\Delta$, $\bfO_{R,\mu,\pm}^\Delta$ are constructed in
\cite{F1}, see also \cite[thm.~4(2)]{F2}. %, see e.g.,  \cite[thm.~4]{F2}. 
Since $z\in I_{\mu,-}$,
the functors $T_{\phi,\mu}$, $T_{\mu,\phi}$ 
preserve the subcategories
${}^z\bfO_{R,\phi,-}^\Delta$, ${}^w\bfO_{R,\mu,-}^\Delta$
by \cite[thm.~4(2)]{F2}.
For the same reason the functors $T_{\phi,\mu}$, $T_{\mu,\phi}$ 
factor to the categories
${}^z\bfO_{R,\phi,+}^\Delta$, ${}^w\bfO_{R,\mu,+}^\Delta$.
The claims $(a), (b), (c)$ are proved in \cite[thm.~4(1),(6)(4)]{F2}.
\end{proof}

\vspace{2mm}

Next, in the non-deformed setting, we consider the translation functors on the whole category O. 
We have the following.

\vspace{2mm}

\begin{prop}
\label{prop:translation2}
Let $R=k$. Let $w\in\widehat W$ and $z\in I_{\mu,-}$. Assume that $wW_\mu=zW_\mu$ and that
$\pm (e-d)>-N$.
We have a $k$-linear functor 
$T_{\phi,\mu}:{}^z\bfO_{\phi,\pm}
\to{}^w\bfO_{\mu,\pm}$
which coincides with the functors in Proposition \ref{prop:translation1}
on ${}^z\bfO_{\phi,\pm}^\Delta$ and such that
the following hold

(a) $T_{\phi,\mu}$ has a left adjoint functor $T_{\mu,\phi}$, both functors are exact,

(b) $T_{\phi,\mu}$, $T_{\mu,\phi}$
take projectives to projectives,

(c) $T_{\phi,\mu}$, $T_{\mu,\phi}$ preserve the parabolic category O and
commute with $i,\tau$,

(d)
$T_{\phi,\mu}(L(xw_\mu\bullet \oo_{\phi,\pm}))=L(x\bullet \oo_{\mu,\pm})$
for each $x\in{}^z\!I_{\mu,\pm}$,

(e)
$T_{\phi,\mu}(L(xw_\mu\bullet \oo_{\phi,\pm}))=0$
for each $x\in{}^z\!I_{\phi,\pm}\setminus{}^z\!I_{\mu,\pm}$,

(f)
$T_{\phi,\mu}({}^z\!L_{\phi,\pm}^\nu)={}^w\!L_{\mu,\pm}^\nu$,

(g)
$T_{\mu,\phi}({}^w\!P^\nu(x\bullet \oo_{\mu,\pm}))=
{}^z\!P^\nu(xw_\mu\bullet \oo_{\phi,\pm})$
for each $x\in{}^w\!I_{\mu,\pm}^\nu$.
\end{prop}

\vspace{2mm}

\begin{proof}
The definition of the translation functor
$T_{\phi,\mu}:\bfO_{\phi,\pm}
\to\bfO_{\mu,\pm}$ is well-known,
see e.g., \cite[sec.~3]{KT}.
Its restriction to $\Delta$-filtered objects
coincides with the functor in Proposition \ref{prop:translation1} if $R=k$,
see \cite{F1}, \cite{KT} for details.
Formulas $(d)$, $(e)$ are proved in \cite[prop.~3.8]{KT}. 
Since $z\in I_{\mu,-}$, they imply that
$T_{\phi,\mu}$ factors to a functor ${}^z\bfO_{\phi,\pm}\to{}^w\bfO_{\mu,\pm}$
which satisfies again $(d)$, $(e)$. 

The existence of the left adjoint 
$T_{\mu,\phi}$ follows from the following general fact.

\vspace{2mm}

\begin{claim}
Let $A$, $B$ be finite dimensional $k$-algebras and 
$T : A\bfmod \to B\bfmod$ be an exact $k$-linear functor. Then $T$ has a left and a right adjoint.
\end{claim}

\vspace{2mm}

The exactness of $T_{\phi,\mu}$ is obvious by construction,
see e.g., \cite[sec.~3]{KT}. It is easy to check that $T_{\phi,\mu}$ commutes with the BGG duality.
Thus, its right adjoint is the conjugate of $T_{\mu,\phi}$ by the duality. So, to check that
$T_{\mu,\phi}$ is exact it is enough to prove that $T_{\phi,\mu}$
takes projectives to projectives. This follows from Proposition \ref{prop:translation1}.

Part $(b)$ follows from the proof of $(a)$.

Part $(c)$ for $T_{\phi,\mu}$ is obvious by construction. For $T_{\mu,\phi}$, it follows by adjunction.

To prove $(f)$, note that $(d)$ implies that
$T_{\phi,\mu}({}^z\!L_{\phi,\pm})={}^w\!L_{\mu,\pm}$.
Thus $(c)$ gives
$T_{\phi,\mu}({}^z\!L_{\phi,\pm}^\nu)\subset{}^w\!L_{\mu,\pm}^\nu$.
Further, for each $x\in{}^z\!I_{\mu,\pm}^\nu$ part $(e)$ yields 
\begin{equation}
\label{toto}
L(xw_\mu\bullet \oo_{\phi,\pm})\subset 
T_{\mu,\phi}T_{\phi,\mu}L(xw_\mu\bullet \oo_{\phi,\pm})=
T_{\mu,\phi}L(x\bullet \oo_{\mu,\pm}).
\end{equation}
Thus, by adjunction, 
we have a surjective map
$T_{\phi,\mu}L(xw_\mu\bullet \oo_{\phi,\pm})
\to L(x\bullet \oo_{\mu,\pm}).$
Now, since the right hand side of \eqref{toto} is in
${}^w\bfO^\nu_{\phi,\pm}$ by $(c)$, the left hand side is also
in ${}^z\bfO^\nu_{\phi,\pm}$.
Thus, we have a surjective map
$T_{\phi,\mu}({}^z\!L_{\phi,\pm}^\nu)
\to {}^wL^\nu_{\mu,\pm}.$

Part $(g)$ is a consequence of $(a)$, $(d)$, $(e)$.
\end{proof}

\vspace{2mm}

\begin{rk}
\label{rk:fidele}
Let $w\in\widehat W$ and $w\in I_{\mu,+}^\nu$.
Set $z=ww_\mu$. We have $z\in I^\nu_{\phi,+}\cap I_{\mu,-}$.
Assume that $e-d>-N$. Then, the  functor
$T_{\mu,\phi}:{}^w\bfO_{\mu,+}^\nu\to{}^z\bfO_{\phi,+}^\nu $
is faithful.
To prove this it is enough to check that $T_{\mu,\phi}(L)\neq 0$ for any simple module $L$.
This follows from Proposition \ref{prop:translation2}$(a),(g)$.
%(to prove this, by Remark \ref{rk:theta/k} and Corollary \ref{cor:speBpm} below, it is enough
%to check that the functor $\theta_{\mu,\phi}$  defined there is faithful
%on ${}^w\bfZ_{\mu,+}$,
%and this is obvious because the unit
%$\bfone\to\theta_{\phi,\mu}\circ\theta_{\mu,\phi}$
%is a direct summand by definition of
%$\theta_{\mu,\phi}$, $\theta_{\phi,\mu}$).
\end{rk}

\vspace{2mm}

\subsection{Translation functors in $\bfZ$}
\label{sec:translationZ}
Fix elements $w\in\widehat W$ and $z\in I_{\mu,-}$.

The algebra ${}^z\!\bar Z_{S,\phi,+}={}^z\!\bar Z_{S,\phi,-}$ consists of the
tuples $(a_x)$ of elements of $S$
labelled by elements $x\in \widehat W$ such that $x\leqslant z$ and
$a_{x}-a_{y}\in\check\alpha S$ for all $x,y$ such that $x=s_\alpha y$.
Since $z$ is maximal in the left coset $z W_\mu$, switching the coordinates yields a left $W_\mu$-action on 
such tuples such that $y\cdot(a_x)=(a_{xy})$ for each $y\in W_\mu$.
This yields a left $W_\mu$-action on the algebra ${}^z\!\bar Z_{S,\phi,\pm}$.

Assume that $w$ belongs to the left coset $zW_\mu$. Then, we have a map 
${}^w\!\bar Z_{S,\mu,\pm}\to{}^{z}\!\bar Z_{S,\phi,\pm}$ such that
$(a_{x})\mapsto (b_{xy}),$ where we define
$b_{xy}=a_x$ for each $x\in{}^{w}\!I_{\mu,\pm}$ and $y\in W_\mu.$
This map identifies the algebra ${}^w\!\bar Z_{S,\mu,\pm}$ with the
set of $W_\mu$-invariant elements in the algebra ${}^{z}\!\bar Z_{S,\phi,\pm}$.

\begin{lemma}\label{lem:isom12}
For $z\in I_{\mu,-}$ and $w\in I_{\mu,\pm}$ such that $w\in zW_\mu$ there are graded
${}^w\!\bar Z_{S,\mu,\pm}$-module isomorphisms
\smallskip

(a) $
{}^{z}\!\bar Z_{S,\phi,\pm}
\ \simeq\ 
\bigoplus_{y\in W_\mu}{}^w\!\bar Z_{S,\mu,\pm}\langle -2l(y)\rangle,$

\smallskip

(b) ${}^{z}\!\bar Z_{S,\phi,\pm}\langle 2l(w_\mu)\rangle
\ \simeq\ 
\Hom_{{}^w\!\bar Z_{S,\mu,\pm}}\bigl(
{}^{z}\!\bar Z_{S,\phi,\pm},
{}^w\!\bar Z_{S,\mu,\pm}
\bigr).$
\end{lemma}
\begin{proof}
It is enough to prove the Lemma for ${}^w\!\bar Z_{S,\mu,+}$ because we have an isomorphism of graded algebra 
${}^w\!\bar Z_{S,\mu,-}={}^v\!\bar Z_{S,\mu,+}$ for $w\in I_{\mu,-}$ and $v=w_+\in I_{\mu,+}$.

\smallskip

By Proposition  \ref{prop:loc1}$(a)$ below, the graded $S$-algebra
${}^w\!\bar Z_{S,\mu,+}$ is isomorphic to the equivariant cohomology (graded) algebra
$H_T(\bar X_{\mu,w})$ of the affine Schubert variety $\bar X_{\mu,w}=\bar X_{\mu,z}$. 
Therefore, the Bruhat decomposition yields
$$\aligned
{}^w\!\bar Z_{S,\mu,+}
=H_T(\bar X_{\mu,z})
=\bigoplus_{x\in I_{\mu,-},x\leqslant z}H_T(X_{\mu,x})\langle 2l(z)-2l(x)\rangle,
\endaligned$$
$$\aligned
{}^z\!\bar Z_{S,\phi,+}
=H_T(\bar X_{\phi,z})
=\bigoplus_{y\in W_\mu}\bigoplus_{x\in I_{\mu,-},x\leqslant z}H_T(X_{\phi,xy})\langle 2l(z)-2l(y)-2l(x)\rangle.
\endaligned$$
Since the obvious projection
$X_{\phi,xW_\mu}\to X_{\mu,x}$ is a $P_\mu/B$-fibration, we have
$
H_T(X_{\mu,x})=\bigoplus_{y\in W_\mu}H_T(X_{\phi,xy}).
$
This implies part $(a)$. 

\smallskip

Finally, part $(b)$ follows from $(a)$, because
$$\aligned
\Hom_{{}^w\!\bar Z_{S,\mu,+}}\bigl({}^{z}\!\bar Z_{S,\phi,+},{}^w\!\bar Z_{S,\mu,+}\bigr)
&=
\bigoplus_{y\in W_\mu}\Hom_{{}^w\!\bar Z_{S,\mu,+}}\bigl({}^w\!\bar Z_{S,\mu,+}\langle -2l(y)\rangle,
{}^w\!\bar Z_{S,\mu,+}\bigr)\\
&=
\bigoplus_{y\in W_\mu}{}^w\!\bar Z_{S,\mu,+}\langle 2l(y)\rangle\\
&={}^{z}\!\bar Z_{S,\phi,+}\langle 2l(w_\mu)\rangle.
\endaligned$$
\end{proof}

\begin{rk}
An algebraic proof of Lemma \ref{lem:isom12} is given in \cite[lem.~5.1]{F4}
in the particular case where $\sharp W_\mu=2$. 
\end{rk}

\vspace{3mm}

Let 
$\bar\theta_{\phi,\mu}:{}^{z}\bar\bfZ_{S,\phi,\pm}\to
{}^w\bar\bfZ_{S,\mu,\pm}$ and
$\bar\theta_{\mu,\phi}:{}^w\bar\bfZ_{S,\mu,\pm}\to
{}^{z}\bar\bfZ_{S,\phi,\pm}$
be the restriction and induction functors with respect to the inclusion
${}^w\!\bar Z_{S,\mu,\pm}\subset{}^{z}\!\bar Z_{S,\phi,\pm}.$

Next, let $R=S$ or $S_0$.
Forgetting the gradings, we define in the same way as above
the functors $\theta_{\phi,\mu}:{}^{z}\bfZ_{R,\phi,\pm}\to
{}^w\bfZ_{R,\mu,\pm}$ and
$\theta_{\mu,\phi}:{}^w\bfZ_{R,\mu,\pm}\to
{}^{z}\bfZ_{R,\phi,\pm}.$

\vspace{2mm}

\begin{prop}
\label{prop:theta}
Assume that $w\in zW_\mu$ and $z\in I_{\mu,-}$. Then, we have

(a)
$\bbV_{S_0}\circ T_{\phi,\mu}=\theta_{\phi,\mu}\circ \bbV_{S_0}$
on ${}^z\bfO^\Delta_{S_0,\phi,\pm},$
and $\bbV_{S_0}\circ T_{\mu,\phi}=\theta_{\mu,\phi}\circ \bbV_{S_0}$
on ${}^w\bfO^\Delta_{S_0,\mu,\pm},$

(b) 
$\bar\theta_{\phi,\mu}$ and $\bar\theta_{\mu,\phi}$
are exact functors 
${}^{z}\bar\bfZ_{S,\phi,\pm}^\Delta\to
{}^w\bar\bfZ_{S,\mu,\pm}^\Delta$
and
${}^w\bar\bfZ_{S,\mu,\pm}^\Delta\to
{}^{z}\bar\bfZ_{S,\phi,\pm}^\Delta,$

(c) $(\bar\theta_{\mu,\phi},\bar\theta_{\phi,\mu},
\bar\theta_{\mu,\phi}\langle 2l(w_\mu)\rangle)$ 
is a triple of adjoint functors,

(d) $\bar\theta_{\phi,\mu}\circ D=D\circ\bar\theta_{\phi,\mu}$ and
$\bar\theta_{\mu,\phi}\circ D=D\circ\bar\theta_{\mu,\phi}\circ\langle 2l(w_\mu)\rangle$,

(e) for each $x\in {}^w\!I_{\mu,+},$ there is a sum $M$ of 
${}^w\!\bar B_{S,\mu,+}(t)\langle j\rangle$'s with $t<x$ such that
$$\bar\theta_{\phi,\mu}({}^z\!\bar B_{S,\phi,+}(xw_\mu))=
\bigoplus_{y\in W_\mu}{}^w\!\bar B_{S,\mu,+}(x)\langle l(w_\mu)-2l(y)\rangle\bigoplus M,$$

(f) We have
\begin{itemize}
\item[(i)] $\bar\theta_{\mu,\phi}({}^w\!\bar B_{S,\mu,+}(x))=
{}^z\!\bar B_{S,\phi,+}(xw_\mu)\langle -l(w_\mu)\rangle$ for $x\in {}^w\!I_{\mu,+}$,
\item[(ii)] 
$\bar\theta_{\mu,\phi}({}^w\!\bar B_{S,\mu,-}(x))=
{}^z\!\bar B_{S,\phi,-}(xw_\mu)$ for $x\in {}^w\!I_{\mu,-}$.
\end{itemize}
\end{prop}

\vspace{.5mm}

\begin{proof}
%Part $(a)$ is obvious. 
Part $(a)$ is proved in \cite[thm.~9]{F2}.

\smallskip

For $(b)$, note that the base change with respect to $S\to S_0$ commutes with 
$\theta_{\phi,\mu}$, $\theta_{\mu,\phi}$ and with the functor $M\mapsto M_{[x]}$,
see Section \ref{rk:basechange}.
Now, by $(a)$ and by Propositions \ref{prop:foncteurV}$(c)$, \ref{prop:translation1}(a),
the functors
$\theta_{\phi,\mu}$, $\theta_{\mu,\phi}$ preserve
${}^{z}\bfZ_{S_0,\phi,\pm}^\Delta$,
${}^w\bfZ_{S_0,\mu,\pm}^\Delta.$
Hence, by Lemma \ref{lem:forget}(c), the functor 
$\bar\theta_{\phi,\mu}$, $\bar\theta_{\mu,\phi}$ preserve
${}^{z}\bar\bfZ_{S,\phi,\pm}^\Delta$,
${}^w\bar\bfZ_{S,\mu,\pm}^\Delta.$

Next,  to
prove the exactness of $\bar\theta_{\phi,\mu}$, $\bar\theta_{\mu,\phi}$ it is enough 
to check the exactness of $\theta_{\phi,\mu}$, $\theta_{\mu,\phi}$, because
$\varepsilon$ is faithfully exact by Lemma \ref{lem:forget}(b).
This follows from $(a)$ and Propositions 
\ref{prop:foncteurV}$(c)$, \ref{prop:translation1}(a).

\smallskip

The proof of part $(c)$ is similar to \cite[prop.~5.2]{F4}.
More precisely, by Lemma \ref{lem:isom12}(b)
there is an isomorphism of functors
$$\aligned
{}^{z}\!\bar Z_{S,\phi,\pm}\langle 2l(w_\mu)\rangle
\otimes_{{}^w\!\bar Z_{S,\mu,\pm}}\!\bullet
&\ \simeq\ \Hom_{{}^w\!\bar Z_{S,\mu,\pm}}\bigl(
{}^{z}\!\bar Z_{S,\phi,\pm},\,
{}^w\!\bar Z_{S,\mu,\pm}
\bigr)\otimes_{{}^w\!\bar Z_{S,\mu,\pm}}\bullet
\\
&\ \simeq\ \Hom_{{}^w\!\bar Z_{S,\mu,\pm}}\bigl(
{}^{z}\!\bar Z_{S,\phi,\pm},\,\bullet\bigr).
\endaligned$$
Therefore
$(\bar\theta_{\mu,\phi},\bar\theta_{\phi,\mu},\bar\theta_{\mu,\phi}\langle 2l(w_\mu)\rangle)$ 
is a triple of adjoint of functors.

\smallskip

Part $(d)$ follows from $(c)$.
Indeed, since $\bar\theta_{\phi,\mu}(M)=M$ as a
graded $S$-module, we have
$\bar\theta_{\phi,\mu}\circ D=D\circ\bar\theta_{\phi,\mu}$.
Then, part $(c)$ implies
$\bar\theta_{\mu,\phi}\circ D=D\circ\bar\theta_{\mu,\phi}\langle 2l(w_\mu)\rangle$ by unicity of adjoint functors.

\smallskip

Now, we prove $(e)$. 
To unburden the notation let us temporarily abbreviate $\bar B_\mu(x)={}^w\!\bar B_{S,\mu,+}(x)$,
$\bar B_\phi(x)={}^z\!\bar B_{S,\phi,+}(x)$,
$\bar Z_\mu={}^w\!\bar Z_{S,\mu,+}$, $\bar Z_\phi={}^z\!\bar Z_{S,\phi,+}$ and
$\bar\bfZ_{\phi}={}^z\bar\bfZ_{S,\phi,+}$.
First, note that $\bar\theta_{\phi,\mu}(\bar B_{\phi}(xw_\mu))$ is 
is a direct sum of objects of the form
$\bar B_\mu(z)\langle j\rangle$ by 
Corollary \ref{cor:4.44} below.
Next, for each subset $J\subset {}^z\!I_{\phi,+}$ and each object $M\in\bar\bfZ_{\phi}^\Delta,$
we have defined a graded $S$-module $M_J=\calL(M)(J)$ in Remark \ref{rk:MI}.
Setting $J=xW_\mu$, we claim that
\begin{equation}\label{3.4}\bar B_{\phi}(xw_\mu)_{xW_\mu}=\bar Z_{\phi,xW_\mu}\langle l(xw_\mu)\rangle.
\end{equation}
To prove this, we use the geometric approach to sheaves over moment graphs recalled in Section \ref{sec:localization} below.
We also use the notation given there (for the dual root system). 
%Recall that $x\in {}^z\!I_{\mu,+}$.
The $\bar Z_{\phi}$-module $\bar B_{\phi}(xw_\mu)\in\bar\bfZ_{\phi}^\Delta$
is a graded BM-sheaf on ${}^z\calG_{\phi,+}$.  
%Note that ${}^z\!I_{\phi,+}=\{u\in\widehat W\,;\,u\leqslant\! z\}$, where $\leqslant$ is the Bruhat order.
Since $x\in I_{\mu,+}$, 
the set $xW_\mu$ is an ideal of the poset ${}^{xw_\mu}\!I_{\mu,+}$.
Thus, by Lemma \ref{lem:foncteurW}, we have
$\bar B_{\phi}(xw_\mu)_{xW_\mu}=IH_T\big(X_{\phi,xW_\mu}\big).$
Since $X_{\phi,xW_\mu}$ is a smooth open subset of $\bar X_{\phi,xw_\mu}$ (see comments before Lemma \ref{lem:foncteurW}),
we deduce that
$$\bar B_{\phi}(xw_\mu)_{xW_\mu}=IH_T\big(X_{\phi,xW_\mu}\big)=
H_T\big(X_{\phi,xW_\mu}\big)\langle l(xw_\mu)\rangle
=\bar Z_{\phi,xW_\mu}\langle l(xw_\mu)\rangle.$$
%This implies the claim \eqref{3.4}.
Now, by \cite[prop.~5.3]{F4}, for each $x\in {}^z\!I_{\mu,+}$ we have
$\bar\theta_{\phi,\mu}(\bar B_{\phi}(xw_\mu))_{ x}
=\bar B_{\phi}(xw_\mu)_{xW_\mu}$.
Thus, by \eqref{3.4} and Lemma \ref{lem:isom12}(a) we have an isomorphism of graded $S$-modules
\begin{eqnarray*}
\bar\theta_{\phi,\mu}(\bar B_{\phi}(xw_\mu))_{x}
&=&(\bar Z_{\phi,xW_\mu})_{xW_\mu}\langle l(xw_\mu)\rangle\\
&=&\bigoplus_{y\in W_\mu}(\bar Z_{\mu,x})_{x}\langle l(xw_\mu)-2l(y)\rangle\\
&=&\bigoplus_{y\in W_\mu}S\langle l(xw_\mu)-2l(y)\rangle.
\end{eqnarray*}
Finally, since for each $t\in I_{\mu,+}$ we have
$$t\not\leqslant x\Rightarrow
\bar\theta_{\phi,\mu}(\bar B_{\phi}(xw_\mu))_{ t}
=\bar B_{\phi}(xw_\mu)_{t W_\mu}=0,$$
the support condition in Proposition \ref{prop:B} implies that the decomposition of 
$\bar\theta_{\phi,\mu}(\bar B_{\phi}(xw_\mu))$ into direct sums of $\bar B_\mu(z)\langle j\rangle$'s has the form as predicted in part $(e)$.

\smallskip

Finally, we prove $(f)$.
By Propositions 
\ref{prop:foncteurV}$(c)$, 
\ref{prop:translation2}$(g)$ 
and part $(a)$ we have
$\theta_{\mu,\phi}({}^w\! B_{S,\mu,\pm}(x))={}^z\! B_{S,\phi,\pm}(xw_\mu)$ for each $x\in {}^w\!I_{\mu,\pm}$.
To identify the gradings, note that, by \cite[prop.~5.3]{F4},
we have
$$\aligned
\bar\theta_{\mu,\phi}({}^w\!\bar B_{S,\mu,\pm}(x))_{xW_\mu}
={}^z\!\bar Z_{\phi,xW_\mu}\otimes_{{}^w\!\bar Z_{\mu,x}}
{}^w\!\bar B_{S,\mu,\pm}(x)_{x}
={}^z\!\bar Z_{\phi,xW_\mu}\langle\pm l(x)\rangle,
\endaligned$$
because 
${}^w\!\bar Z_{\mu, x}=S$
and
${}^w\!\bar B_{S,\mu,\pm}(x)_{ x}=S\langle\pm l(x)\rangle.$
Since
$({}^z\!\bar Z_{\phi, x})_y=S$
for all $y\in xW_\mu,$
we deduce that 
$\bar\theta_{\mu,\phi}({}^w\!\bar B_{S,\mu,+}(x))_{xw_\mu}=S\langle l(x)\rangle$
and
$\bar\theta_{\mu,\phi}({}^w\!\bar B_{S,\mu,-}(x))_{x}=S\langle-l(x)\rangle.$
\end{proof}

\vspace{2mm}

\begin{rk}
\label{rk:theta/k}
Now, we consider the case $R=k$.
Recall the algebras
${}^w Z_\mu$, ${}^w\bar Z_\mu$ and the categories
${}^w \bfZ_\mu$, 
${}^w\bar \bfZ_\mu$ from the beginning of Section \ref{sec:3.4},
and the translation functors $T_{\mu,\phi}$, $T_{\phi,\mu}$
in Section \ref{sec:translationO}.

We define 
$\bar\theta_{\phi,\mu}:{}^{z}\bar\bfZ_{\phi}\to
{}^w\bar\bfZ_{\mu}$ and
$\bar\theta_{\mu,\phi}:{}^w\bar\bfZ_{\mu}\to
{}^{z}\bar\bfZ_{\phi}$
to be the restriction and induction functors with respect to the inclusion
${}^w\!\bar Z_{\mu}\subset{}^{z}\!\bar Z_{\phi}.$
They are exact (for the stupid exact structure).

In the non-graded case, we define $\theta_{\mu,\phi}$, $\theta_{\phi,\mu}$ in a similar way.
Then, we have
$\bbV_k\circ T_{\phi,\mu}=\theta_{\phi,\mu}\circ \bbV_k$
and $\bbV_k\circ T_{\mu,\phi}=\theta_{\mu,\phi}\circ \bbV_k$ on $\Delta$-filtered modules
by Proposition \ref{prop:theta}$(a)$, because $\theta_{\phi,\mu}$, $\theta_{\mu,\phi}$,
$T_{\phi,\mu}$, $T_{\mu,\phi}$ and $\bbV$ commute with base change by 
Propositions  \ref{prop:foncteurV}, \ref{prop:translation1}.
\end{rk}

\vspace{2mm}

\subsection{Moment graphs and the flag ind-scheme}
\label{sec:localization}
Fix a parabolic type $\mu\in\calP$. 
Let $P_\mu\subset G(k((t)))$ be the
parabolic subgroup with Lie algebra $\bfp_\mu$.
Write $B=P_\phi$ and let $T$ be the torus with Lie algebra $\bft$.
Note that $T$ is the product of $k^\times\times k^\times$ by the maximal torus of $G$.

Let $X'$ be the partial (affine) flag ind-scheme $G(k((t)))/P_\mu$.
The affine Bruhat cells are indexed by the cosets $\widehat W/W_\mu$, which we identify with $I_{\mu,+}(=I_\mu^\min)$. 
For each $w\in I_{\mu,+}$
let $X_w\subset\bar X_w\subset X'$ be the corresponding finite dimensional
affine Bruhat cell and Schubert variety. 
To avoid confusions we may write $X'_\mu=X'$, $X_{\mu,w}=X_w$ and $\bar X_{\mu,w}=\bar X_w$.
For any subset $J\subset I_{\mu,+}$ we abbreviate $X_J=\bigcup_{w\in J}X_w$.

The group $T$ acts on $\bar X_w$, with  the first copy of $k^\times$ 
acting by rotating the loop and the last one acting trivially.
The varieties $\bar X_w$ form an inductive system of
complex projective varieties with closed embeddings and 
$X'$ is represented by the ind-scheme
$\text{ind}_w\bar X_w$. Recall that $\bar X_w$ has dimension $l(w)$
and
$\bar X_w=\bigcup_{x\in {}^w\!I_{\mu,+}}X_x$.

Let $\bfD^b(\bar X_w)$ be the bounded 
derived category of constructible
sheaves of $k$-vector spaces on 
$\bar X_w,$ which are locally constant along the $B$-orbits, and
let $\bfP(\bar X_w)$ be the full 
lubcategory  of  perverse sheaves.
Let $\bfD^b_T(\bar X_w)$ and $\bfP_T(\bar X_w)$ be be their $T$-equivariant analogue.

For each $x\in {}^w\!I_{\mu,+},$ let ${}^w\!IC(\bar X_x)$ 
be the  intersection cohomology complex in $\bfP(\bar X_w)$
associated with $\bar X_x$ and
let ${}^w\!IC_T(\bar X_x)$ be its $T$-equivariant analogue.
Let $IH(\bar X_x)$
be the  intersection cohomology of $\bar X_x$
and let  $IH_T(\bar X_x)$ be its $T$-equivariant analogue.
Finally, let $H(\bar X_x)$, $H_T(\bar X_x)$
denote the ($T$-equivariant) cohomology of $\bar X_x$. 
See Section \ref{sec:HT} for details. 

\smallskip
In this section we set $V=\bft^*$. Since
$S$ is the symmetric algebra over $V$, we have $S=H_T(\bullet)$.
Let $S^\vee$ denote the symmetric algebra over $V^*=\bft$.

Let ${}^w\calG_{\mu,\pm}^\vee$ 
be the moment graph over $V$ 
whose set of vertices is ${}^w\!I_{\mu,\pm}$,
with an edge labelled by $k\,\alpha$ between $x,y$ if there is an affine
reflection $s_{\alpha}$ such that $s_\alpha y\in xW_\mu$. 

We define
${}^w\bar Z_{S,\mu,\pm}^\vee$,
${}^w\bar\bfZ_{S,\mu,\pm}^\vee$,
${}^w\bar B_{S,\mu,\pm}(x)^\vee$,
${}^w\bar C_{S,\mu,\pm}(x)^\vee$, etc.,
in the obvious way.
We'll also write ${}^w\bar\bfF_{S,\mu,\pm}^\vee$ for the category of graded $S$-sheaves
of finite type over ${}^w\calG_{\mu,\pm}^\vee$ whose stalks are torsion free as $S$-modules.
We also write 
${}^w\bar B_{\mu,\pm}^\vee(x)=
k\,{}^w\bar B_{S,\mu,\pm}^\vee(x),$
${}^w\bar C_{\mu,\pm}^\vee(x)=
k\,{}^w\bar C_{S,\mu,\pm}^\vee(x),$
${}^w\!\bar Z_{\mu,\pm}^\vee=k{}^w\!\bar Z_{S,\mu,\pm}^\vee$,
${}^w\bar\bfZ_{\mu,\pm}^\vee={}^w\!\bar Z_{\mu,\pm}^\vee\bfmod$, etc.
In the non graded case we use a similar notation.

\vspace{2mm}

\begin{prop} \label{prop:loc1} For $w\in I_{\mu,+}$ and $x\in{}^w\!I_{\mu,+}$ we have

(a) $H_T(\bar X_w)=
{}^w\!\bar Z_{S,\mu,+}^\vee$ and $H(\bar X_w)={}^w\!\bar Z_{\mu,+}^\vee$
as graded $k$-algebras,

(b) 
$IH_T(\bar X_x)={}^w\!\bar B_{S,\mu,+}^\vee(x)$
as graded ${}^w\!\bar Z_{S,\mu,+}^\vee$-modules,

(c) $IH(\bar X_x)={}^w\!\bar B_{\mu,+}^\vee(x)$
as  graded ${}^w\!\bar Z_{\mu,+}^\vee$-modules.
\end{prop}

\vspace{.5mm}

\begin{proof}
Part $(a)$ follows from \cite{GKM},
parts $(b),$ $(c)$ from \cite[thm.~1.5, 1.6, 1.8]{BM}.
\end{proof}

\vspace{2mm}

\begin{cor}\label{cor:4.44} Assume that  $w\in I_{\mu,+}$ and $z=w_-\in I_{\mu,-}$.
For each $x\in{}^z\!I_{\phi,+}$, the graded ${}^w\!\bar Z^\vee_{S,\mu,+}$-module
$\bar\theta_{\phi\mu}({}^z\!\bar B_{S,\phi,+}^\vee(x))$ is a direct sum of graded modules of the form
${}^w\!\bar B_{S,\mu,+}^\vee(y)\langle j\rangle$ with $y\in {}^w\!I_{\mu,+}$ and $j\in\bbZ$.
\end{cor}

\vspace{.5mm}

\begin{proof}
The obvious projection
$\pi:\bar X_{\phi,z}\to\bar X_{\mu,w}$ is proper.
By Proposition \ref{prop:loc1},  for each $\calE\in\bfP_T(\bar X_{\phi,z}),$ we may regard the cohomology spaces
$H(\calE)$, $H(\pi_*\calE)$ as graded modules over
${}^z\!\bar Z_{S,\phi,+}^\vee$ and ${}^w\!\bar Z_{S,\mu,+}^\vee$ respectively.
Since
$\bar\theta_{\phi,\mu}$
is the restriction of graded modules with respect to the inclusion
${}^w\!\bar Z^\vee_{S,\mu,+}\subset{}^{z}\!\bar Z^\vee_{S,\phi,+},$
we have an obvious isomorphism of graded
${}^w\! \bar Z^\vee_{S,\mu,+}$-modules
$H(\pi_*\calE)\simeq\bar\theta_{\phi\mu}H(\calE)$.
Setting $\calE={}^w\!IC_T(\bar X_{\phi,x})$, we obtain a graded module isomorphism
$H(\pi_*({}^z\!IC_T(\bar X_{\phi,x})))\simeq\bar\theta_{\phi\mu}({}^z\!\bar B_{S,\phi,+}^\vee(x))$.
Now, by the decomposition theorem, the complex $\pi_*({}^z\!IC_T(\bar X_{\phi,x}))$ is a direct sum of complexes of the form
${}^w\!IC_T(\bar X_{\mu,y})[j],$ with $y\in{}^w\!I_{\mu,+}$ and $j\in\bbZ$.
We deduce that $\bar\theta_{\phi\mu}({}^z\!\bar B_{S,\phi,+}^\vee(x))$ is a direct sum of graded ${}^w\!\bar Z^\vee_{S,\mu,+}$-modules of the form
${}^w\!\bar B_{S,\mu,+}^\vee(y)\langle j\rangle.$
\end{proof}

\vspace{2mm}

\begin{rk}
The counterpart of Proposition \ref{prop:loc1} for ${}^w\!\bar Z_{S,\mu,-}^\vee$ and ${}^w\!\bar B_{S,\mu,-}^\vee(x)$ is 
proved in Proposition \ref{prop:IH} below. 
\end{rk}

\vspace{2mm}
Recall that the poset ${}^x\!I_{\mu,+}$ is equipped with the opposite Bruhat order. An ideal $J\subset {}^x\!I_{\mu,+}$ is a 
coideal in the set $\{y\in I_{\mu}^\min\,;\, y\leqslant x\}$ 
equipped with the Bruhat order. Hence the variety $X_J$ is a $T$-stable open subvariety of $\bar X_x$. 
We have the following proposition.

\vspace{2mm}

\begin{lemma}\label{lem:foncteurW}
Let $w\in I_{\mu,+}$. For any $x\in{}^w\!I_{\mu,+}$ and any ideal $J\subset {}^x\!I_{\mu,+}$ there is an isomorphism of graded 
$({}^w\!\bar Z^\vee_{S,\mu,+})_J$-modules
${}^w\!\bar B^\vee_{S,\mu,+}(x)_J=IH_T(X_J).$
\end{lemma}

\vspace{.5mm}

\begin{proof}
The $T$-variety $X_J$ satisfies the assumption in \cite[sec.~1.1]{BM}. Hence, by loc.~cit., one can associate a moment graph $\calG^\vee_J$ to it. 
By construction, the graph $\calG^\vee_J$ is the same as the subgraph of ${}^w\calG^\vee_{\mu,+}$ consisting of the vertices 
which belong to $J$ and the edges $h$ with $h',h''\in J$. Let $\bar Z^\vee_{S,J}$ be the graded structural algebra of $\calG^\vee_J$.
We have $\bar Z^\vee_{S,J}=({}^w\!\bar Z^\vee_{S,\mu,+})_J$ as graded rings, see Remark \ref{rk:MI}.

For any $y\in J$, let $\bar B^\vee_{S,J}(y)$ be the set of sections of the graded BM-sheaf on $\calG^\vee_J$ associated with $y$,
as defined in Proposition-Definition \ref{df:3.9}. 
By \cite[thm.~1.5,1.7]{BM}, we have a graded ring isomorphism $\bar Z^\vee_{S,J}=H_T(X_J)$ and a graded $\bar Z^\vee_{S,J}$-module 
isomorphism $\bar B^\vee_{S,J}(y)=IH_T(\bar X_y\cap X_J)$. In particular, we have $\bar B^\vee_{S,J}(x)=IH_T(X_J)$.

Next, since $J\subset {}^x\!I_{\mu,+}$ is an ideal, it follows from the construction of BM-sheaves recalled in Remark \ref{rk:BM} that we have canonical identifications $\bar B^\vee_{S,J}(x)_y={}^w\!\bar B^\vee_{S,\mu,+}(x)_y$, $\bar B^\vee_{S,J}(x)_h={}^w\!\bar B^\vee_{S,\mu,+}(x)_h$ for any $y$, $h\in\calG^\vee_J$ which are compatible with the maps $\rho_{y,h}$. This implies that $\bar B^\vee_{S,J}(x)={}^w\!\bar B^\vee_{S,\mu,+}(x)_J$ as graded $\bar Z^\vee_{S,J}$-modules. We deduce that ${}^w\!\bar B^\vee_{S,\mu,+}(x)_J=IH_T(X_J)$ as $({}^w\!\bar Z^\vee_{S,\mu,+})_J$-modules.

\iffalse%%%%%%%%%%%%%%%%%%%%%%%%%%%%%%%%
By \cite[sec.~4.5]{FW}, there is a functor $\bbW_T:\bfD^b_T(\bar X_w)\to{}^w\bar\bfF^\vee_{S,\mu}$
such that $H_T(\calE_x)=\bbW_T(\calE)_x$ for each 
$\calE\in\bfD^b_T(\bar X_w)$ and each $x\in {}^w\!I_{\mu,+}$. 
Here, the element $x$ is viewed as
a $T$-fixed point of $\bar X_w$ and
$\calE_x$ is the stalk of $\calE$ at this point.
We have $H_T(\calE)=\Gamma(\bbW_T(\calE))$ for $\calE\in\bfD^b_T(\bar X_w)$ by \cite[thm.~4.4]{FW} and
$\bbW_T({}^wIC_T(\bar X_x))=\calL({}^w\!\bar B^\vee_{S,\mu,+}(x))$
by \cite[thm.~6.10 prop.~7.1]{FW}.

\smallskip

Now, for each ideal $J$ in the set ${}^x\!I_{\mu,+}$ we have
\begin{eqnarray*}
{}^w\!\bar B^\vee_{S,\mu,+}(x)_J&=&\Im\big(\Gamma\calL({}^w\!\bar B^\vee_{S,\mu,+}(x))\to\bigoplus_{y\in J}\calL({}^w\!\bar B^\vee_{S,\mu,+}(x))_y\big)\\
&=&\Im\big(\Gamma\bbW_T({}^wIC_T(\bar X_x))\to\bigoplus_{y\in J}\bbW_T({}^wIC_T(\bar X_x))_y\big)\\
&=&\bbW_T({}^wIC_T(\bar X_x))(J)\\
&=&\Gamma\bbW_T({}^wIC_T(X_J))\\
&=&IH_T(X_J).
\end{eqnarray*}
See Remark \ref{rk:MI} for the third equality.
\fi%%%%%%%%%%%%%%%%%%%%%%%%%%%%%%%%%%%%%%
\end{proof}

\vspace{2mm}

\begin{prop} 
\label{prop:loc2}
Assume that  $w\in I_{\mu,+}$. Let $v=w_-^{-1}\in I^\mu_{\phi,-}$.
There is an equivalence of abelian categories 
$\Phi=\Phi^\mu:{}^v\bfO^\mu_{\phi,-}\to\bfP(\bar X_{\mu,w})$
such that $\Phi^\mu(L(y\bullet\oo_{\phi,-}))={}^w\!IC(\bar X_{\mu,x})$
with $x=y_+^{-1}$ for each $y\in{}^v\!I_{\phi,-}^\mu$.
\end{prop}

\vspace{.5mm}

\begin{proof}
Set $z=v^{-1}\in I_{\mu,-}$.
We have  $w=zw_\mu$ and  $x=y^{-1}w_\mu.$
By Corollary \ref{lem:C2}$(b)(c)$, the assignment $y\mapsto x$ yields a bijection 
%$\!I_{\phi,-}^\mu\to \!I_{\mu,+}$, which restricts to a bijection
${}^v\!I_{\phi,-}^\mu\to {}^w\!I_{\mu,+}$. 

By \cite[thm.~5.5]{FG2},
see also \cite[thm.~7.15, 7.16]{BD} and \cite{KT4},
we have
an equivalence of abelian categories 
$\Phi^\phi:{}^v\bfO_{\phi,-}\to\bfP(\bar X_{\phi,z})$ such that
$L(u\bullet\oo_{\phi,-})\mapsto{}^{z}\!IC(\bar X_{u^{-1}})$
for each $u\in{}^v\!I_{\phi,-}.$
Composing $\Phi^\phi$ and the parabolic inclusion functor $i_{\mu,\phi}:{}^v\bfO_{\phi,-}^\mu\to{}^v\bfO_{\phi,-}$ yields an embedding of abelian categories
$\Phi^\phi\circ i_{\mu,\phi}:{}^v\bfO_{\phi,-}^\mu\to\bfP(\bar X_{\phi,z})$.
It takes
$L(y\bullet\oo_{\phi,-})$ to ${}^{z}\!IC(\bar X_{y^{-1}})$ for each
$y\in{}^v\!I_{\phi,-}^\mu.$

Finally, since $z\in I_{\mu,-}$, the obvious projection
$\pi:\bar X_{\phi,z}\to\bar X_{\mu,w}$ is smooth.
Hence, the functor $\pi^{!*}=\pi^*[\dim\pi]$ yields an embedding of abelian categories
$\bfP(\bar X_{\mu,w})\to\bfP(\bar X_{\phi,z})$
such that ${}^w\!IC(\bar X_{\mu,x})\mapsto{}^{z}\!IC(\bar X_{\phi, xw_\mu})$ for each
$x\in{}^w\!I_{\mu,+}.$
The essential images of the functors $\Phi^\phi\circ i_{\mu,\phi}$ and $\pi^{!*}$ are
Serre subcategories of $\bfP(\bar X_{\phi,v^{-1}})$ generated by the same 
set of simple objects. Thus, they are equivalent. Hence, the categories
${}^v\bfO_{\phi,-}^\mu$ and $\bfP(\bar X_{\mu,w})$ are also equivalent.
\end{proof}

\vspace{2mm}

By Proposition \ref{prop:loc1}, 
composing $\Phi^\mu$ and the cohomology we get a functor
$\bbH=\bbH^\mu:{}^v\bfO^\mu_{\phi,-}\to{}^w\bfZ_{\mu,+}^\vee$ for each $w\in I_{\mu,+}$,
$v\in I^\mu_{\phi,-}$ such that $v=w_-^{-1}$.

\vspace{2mm}

\begin{prop}
\label{prop:loc3}
Let $v,w$ be as above and $z=v^{-1}$.
Let $y,t\in{}^v\!I_{\phi,-}^\mu$. Set
$x=y_+^{-1}$ and $s=t_+^{-1}$. We have

(a)
$\bbH^\mu(L(y\bullet \oo_{\phi,-}))={}^w\!B_{\mu,+}^\vee(x)$,

(b) we have graded $k$-vector space isomorphisms
$$\aligned
\Ext_{{}^v\bfO^\mu_{\phi,-}}\bigl(L(y\bullet \oo_{\phi,-}),
L(t\bullet \oo_{\phi,-})\bigr)
&=k\Hom_{{}^w\!Z_{S,\mu,+}^\vee}\bigl({}^w\!\bar B_{S,\mu,+}^\vee(x),
{}^w\!\bar B_{S,\mu,+}^\vee(s)\bigr)\\
&=\Hom_{{}^w\!Z_{\mu,+}^\vee}\bigl({}^w\!\bar B_{\mu,+}^\vee(x),
{}^w\!\bar B_{\mu,+}^\vee(s)\bigr),\endaligned$$

(c) there is a morphism of functors $\theta_{\mu,\phi}\circ\bbH^\mu\to \bbH^\phi\circ i_{\mu,\phi}$
which yields a ${}^z\!Z_{\phi,+}^\vee$-module isomorphism
$\theta_{\mu,\phi}\bbH^\mu(L)\simeq \bbH^\phi i_{\mu,\phi}(L)$
for each $L\in\Irr({}^v\bfO^\mu_{\phi,-})$.
\end{prop}

\vspace{.5mm}

\begin{proof} 
Part $(a)$ follows from Propositions  \ref{prop:loc1}, \ref{prop:loc2}.
To prove part $(b)$, note that, by Proposition \ref{prop:loc2}, we have a graded $k$-vector space isomorphism
$$\Ext_{{}^v\bfO^\mu_{\phi,-}}\bigl(L(y\bullet \oo_{\phi,-}), L(t\bullet \oo_{\phi,-})\bigr)=
\Ext_{\bfD^b(\bar X_w)}\bigl({}^w\!IC(\bar X_x),{}^w\!IC(\bar X_s)\bigr).$$
Further, by Propositions \ref{prop:IC/ICT/IH},  \ref{prop:loc1}, we have
graded $k$-vector space isomorphisms
$$\Ext_{\bfD^b_T(\bar X_w)}\bigl({}^w\!IC_T(\bar X_x),
{}^w\!IC_T(\bar X_s)\bigr)=
\Hom_{{}^w\!Z_{S,\mu,+}^\vee}\bigl({}^w\!\bar B_{S,\mu,+}^\vee(x),
{}^w\!\bar B_{S,\mu,+}^\vee(s)\bigr),$$
$$\Ext_{\bfD^b(\bar X_w)}\bigl({}^w\!IC(\bar X_x),
{}^w\!IC(\bar X_s)\bigr)=
k\Ext_{\bfD^b_T(\bar X_w)}\bigl({}^w\!IC_T(\bar X_x),
{}^w\!IC_T(\bar X_s)\bigr).$$
This proves the first isomorphism in $(b)$. 
To prove the second one, note that we have ${}^w\!\bar B^\vee_{\mu,+}(x)=k{}^w\!\bar B^\vee_{S,\mu,+}(x)$.
Thus, by Propositions \ref{prop:loc1}, \ref{prop:IC/ICT/IH}, we have
graded $k$-vector space isomorphisms
$$\aligned
k\Hom_{{}^w\!Z_{S,\mu,+}^\vee}\bigl({}^w\!\bar B_{S,\mu,+}^\vee(x),
{}^w\!\bar B_{S,\mu,+}^\vee(s)\bigr)&=
k\End_{H_T(\bar X_w)}\big(IH_T(\bar X_x),IH_T(\bar X_s)\big),\\
&=
\End_{H(\bar X_w)}\big(IH(\bar X_x),IH(\bar X_s)\big),\\
&=\Hom_{{}^w\!Z_{\mu,+}^\vee}\bigl({}^w\!\bar B_{\mu,+}^\vee(x),
{}^w\!\bar B_{\mu,+}^\vee(s)\bigr).
\endaligned$$

Now, we prove $(c)$.
%Note that $\theta_{\mu,\phi}$ is well-defined, because $z\in I^\max_\mu$ and $w=zw_\mu$.
By Proposition \ref{prop:loc1}, taking the cohomology gives a functor
$\bfP(\bar X_{\mu,w})\to{}^w\! Z^\vee_{\mu,+}\bfmod$ such that
$\calE\mapsto H(\calE).$
Since $z\in I_{\mu,-}$, the obvious projection
$\pi:\bar X_{\phi,z}\to\bar X_{\mu,w}$ is smooth.
By Proposition \ref{prop:loc1},  for each $\calE\in\bfP(\bar X_{\mu,w})$ we may regard the cohomology spaces
$H(\calE)$, $H(\pi^*\calE)$ as modules over
${}^w\!Z_{\mu,+}^\vee$ and ${}^z\!Z_{\phi,+}^\vee$ respectively.
The unit $1\to \pi_*\pi^*$ yields a map $H(\calE)\to H(\pi_*\pi^*\calE)$.
By Proposition \ref{prop:loc1}, it may be viewed as natural morphism of ${}^w\! Z^\vee_{\mu,+}$-modules
$H(\calE)\to\theta_{\phi,\mu}H(\pi^*\calE)$.
Hence, by adjunction, we get a morphism of functors
$\theta_{\mu,\phi}\circ H\to H\circ \pi^*$ from $\bfP(\bar X_{\mu,w})$ to ${}^z\! Z^\vee_{\phi,+}\bfmod$.
Let $\Phi^\mu:{}^v\bfO^\mu_{\phi,-}\to\bfP(\bar X_{\mu,w})$ and
$\Phi^\phi:{}^v\bfO_{\phi,-}\to\bfP(\bar X_{\phi,z})$ be as above.
The proof of Proposition \ref{prop:loc2} implies
that $\bbH^\mu=H\circ\Phi^\mu$, $\bbH^\phi=H\circ\Phi^\phi$ and $\Phi^\phi\circ i_{\mu,\phi}=\pi^{!*}\circ\Phi^\mu$.
Hence,  we have a morphism of functors
$\theta_{\mu,\phi}\circ\bbH^\mu\to\bbH^\phi\circ i_{\mu,\phi}$.

We claim that, for each $L\in\Irr({}^v\bfO^\mu_{\phi,-}),$
the corresponding ${}^z\! Z_{\phi,+}^\vee$-module homomorphism
$f:\theta_{\mu,\phi}\bbH^\mu(L)\to\bbH^\phi i_{\mu,\phi}(L)$ is an isomorphism.
Indeed, set $L=L(y\bullet\oo_{\phi,-})$ with $y\in {}^v\!I_{\phi,-}^\mu$.
By part $(a)$, we have $\bbH^\mu(L)={}^z\!B_{\mu,+}^\vee(x)$, $\bbH^\phi i_{\mu,\phi}(L)={}^z\!B_{\phi,+}^\vee(xw_\mu)$ with $x=y_+^{-1}$.
Hence, by Proposition \ref{prop:theta}$(f)$ and base change, we have
$\theta_{\mu,\phi}\bbH^\mu(L)=\bbH^\phi i_{\mu,\phi}(L)={}^z\!B_{\phi,+}^\vee(xw_\mu),$
see Remark \ref{rk:theta/k}. Hence, the map $f$ can be regarded as an element in
$\End_{{}^z\!Z_{\phi,+}^\vee}\bigl({}^z\!B_{\phi,+}^\vee(xw_\mu)\bigr).$ 
Now, by part $(b)$ and Propositions \ref{prop:2.3}$(b)$, \ref{prop:foncteurV}$(b), (c)$, we have
$$
\aligned
\End_{{}^z\!Z_{\phi,+}^\vee}\bigl({}^z\!B_{\phi,+}^\vee(xw_\mu)\bigr)
&=k\End_{{}^z\!Z_{S,\phi,+}^\vee}\bigl({}^z\!B_{S,\phi,+}^\vee(xw_\mu)\bigr),\\
&=k\End_{{}^z\bfO^\vee_{S,\phi,+}}\bigl({}^z\!P^\vee_{S}(xw_\mu\bullet \oo_{\phi,+})\bigr),\\
&=\End_{{}^z\bfO^\vee_{\phi,+}}\bigl({}^z\!P^\vee(xw_\mu\bullet \oo_{\phi,+})\bigr).
\endaligned$$
Here, the symbols $\bfO^\vee$ and $P^\vee$ denote the category O and the projective modules associated with the dual root system.
Since the module ${}^z\!P^\vee(xw_\mu\bullet \oo_{\phi,+})$ is projective and indecomposable,
we deduce that the $k$-algebra $\End_{{}^z\!Z_{\phi,+}^\vee}\bigl({}^z\!B_{\phi,+}^\vee(xw_\mu)\bigr)$ is local.
Thus, to prove the claim it is enough to observe that the map $f$ is not nilpotent.
\end{proof}

\vspace{2mm}

We can now prove the following graded analogue of (part of) Corollary \ref{cor:2.31} which
compares the graded $k$-algebra  ${}^{v}\!\bar A^\mu_{\phi,-}$ in 
Definition \ref{df:3.6} with the graded $k$-algebra  ${}^{w}\!\bar \scrA_{\mu,+}$ in Definition \ref{def:3.23}.
The comparison between ${}^{u}\!\bar A^\mu_{\phi,+}$ and ${}^{z}\!\bar \scrA_{\mu,-}$ will be done in Corollary \ref{cor:2.42} below.

\vspace{2mm}

\begin{cor}
\label{cor:loc4}
Let $w\in I_{\mu,+}$ and $v=w^{-1}_-\in I^\mu_{\phi,-}$.
We have
a graded $k$-algebra isomorphism
${}^{v}\!\bar A^\mu_{\phi,-}\to{}^{w}\!\bar \scrA_{\mu,+}$
such that $1_{y}\mapsto 1_x$ with $x=y_+^{-1}$ for each $y\in{}^v\!I_{\phi,-}^\mu$.
\end{cor}

\vspace{.5mm}

\begin{proof}
We have
${}^w\!\bar \scrA_{\mu,+}=k\End_{{}^w\!Z_{S^\vee,\mu,+}}\bigl({}^w\!\bar B_{S^\vee,\mu,+}\bigr)^{\op},$
where 
${}^w\!\bar B_{S^\vee,\mu,+}\in{}^w\bar\bfZ_{S^\vee,\mu,+}$ is the space of sections of a direct sum of graded BM-sheaves
on ${}^w\calG_{\mu,+}$. 
Next, since
${}^v\!\bar A^\mu_{\phi,-}=
\Ext_{{}^{v}\bfO^\mu_{\phi,-}}({}^{v}\!L^\mu_{\phi,-})^\op$ by Definition \ref{df:3.6},
we have a graded $k$-algebra isomorphism
${}^v\!\bar A^\mu_{\phi,-}=
k\End_{{}^w\!Z_{S,\mu}^\vee}({}^w\!\bar B_{S,\mu,+}^\vee)^{\op}$
by Proposition \ref{prop:loc3}$(b)$.
Thus, we must check that
$k\End_{{}^w\!Z_{S,\mu}^\vee}({}^w\!\bar B_{S,\mu,+}^\vee)\simeq
k\End_{{}^w\!Z_{S^\vee,\mu}}\bigl({}^w\!\bar B_{S^\vee,\mu,+}\bigr).$

The choice of a $\widehat W$-invariant pairing on $\bft$ yields a $\widehat W$-equivariant 
isomorphism 
$\bft^*\simeq\bft$. 
It induces a $\widehat W$-equivariant graded algebra isomorphism $S\simeq S^\vee$. 
The moment graphs
${}^w\calG_{\mu,+}^\vee$ and ${}^w\calG_{\mu,+}$ are indeed isomorphic.
Using Remark \ref{rk:BM}, it is easy to check that this isomorphism identifies the 
graded BM-sheaves  
$\calL({}^w\!\bar B_{S,\mu,+}^\vee)$ and
$\calL({}^w\!\bar B_{S^\vee,\mu,+})$.
This proves the corollary.
\end{proof}

\vspace{2mm}

\begin{cor}\label{cor:speBpm}
For any $w\in I_{\mu,\pm}$, the following hold

(a)
we have graded $k$-algebra isomorphisms
${}^w\!\bar \scrA_{\mu,\pm}\to
\End_{{}^w\!Z_{\mu,\pm}}({}^w\! \bar B_{\mu,\pm})^\op$,

(b)
the functor $\bbV_{\!k}$ is fully-faithful on projective objects in ${}^w\bfO_{\mu,\pm}^\Delta$.
\end{cor}

\vspace{.5mm}

\begin{proof}
The obvious map ${}^w\!\bar \scrA_{\mu,+}\to
\End_{{}^w\!Z_{\mu,+}}({}^w\! \bar B_{\mu,+})^\op$
is  an isomorphism by Proposition \ref{prop:loc3}$(b)$ (applied to the dual root system). 
This proves $(a)$ in the positive-level case.
The negative one is proved in a similar way, see Corollary \ref{cor:B2}.

Now, let us concentrate on part $(b)$.
By Proposition \ref{prop:foncteurV}$(b), (c),$ we have
${}^w\! B_{\mu,+}=\bbV_{\!k}({}^w\!P_{\mu,+})$ 
and
${}^w\!\scrA_{\mu,+}=k\End_{{}^w\bfO_{S_0,\mu,+}}({}^w\!P_{S_0,\mu,+})^\op.$
The latter is isomorphic to $\End_{{}^w\bfO_{\mu,+}}({}^w\!P_{\mu,+})^\op$ by Proposition \ref{prop:2.3}$(b)$. We deduce
that $\bbV_{\!k}$ yields an isomorphism
$\End_{{}^w\!Z_{\mu,+}}(\bbV_{\!k}({}^w\!P_{\mu,+}))
=\End_{{}^w\bfO_{\mu,+}}({}^w\!P_{\mu,+})$,
which proves that $\bbV_{\!k}$ is fully-faithful on projectives in ${}^w\bfO_{\mu,+}^\Delta$.

The negative-level case is similar.
Indeed, since ${}^w\!P_{\mu,-}=k {}^w\!P_{S_0,\mu,-}$ by Proposition \ref{prop:2.3}$(e)$, it follows from
Proposition \ref{prop:foncteurV}$(b), (c)$ that
${}^w\! B_{\mu,-}=\bbV_{\!k}({}^w\!P_{\mu,-})$ and 
${}^w\!\scrA_{\mu,-}=k\End_{{}^w\bfO_{S_0,\mu,-}}({}^w\!P_{S_0,\mu,-})^\op.$
Thus
$\End_{{}^w\!Z_{\mu,-}}(\bbV_{\!k}({}^w\!P_{\mu,-}))
=\End_{{}^w\bfO_{\mu,-}}({}^w\!P_{\mu,-})$
by Proposition \ref{prop:2.3}$(b)$.
We deduce that the functor $\bbV_{\!k}$ is fully faithful on projectives in ${}^w\bfO_{\mu,-}^\Delta$.
\end{proof}

\vspace{2mm}

\begin{rk}
Assume that $w\in I_{\mu,\pm}$. We have
${}^w\bar C_{S,\mu,\pm}(x)=D({}^v\!\bar B_{S,\mu,\mp}(y))$
with $y=x_{\mp}$, $v=w_\mp$ for each $x\in {}^w\!I_{\mu,\pm}$.
Therefore, there are graded $k$-algebra isomorphisms
${}^v\!\bar \scrA_{\mu,\mp}=
k\End_{{}^w\!Z_{S,\mu,\pm}}\bigl({}^w\bar C_{S,\mu,\pm}\bigr).$
By Corollary \ref{cor:speBpm}, we also have
graded $k$-algebra isomorphisms
${}^v\!\bar \scrA_{\mu,\mp}=
\End_{{}^w\!Z_{\mu,\pm}}({}^w \bar C_{\mu,\pm})$.
\end{rk}

\vspace{1cm}

\section{Proof of the main theorem}
\label{sec:5}

\subsection{The regular case}
\label{sec:3.9}
Fix integers $d,e>0$ and fix $\mu\in\calP$.
Assume that the weights $\oo_{\mu,-}$,  $\oo_{\phi,-}$ have the level
$-e-N$ and $-d-N$ respectively.

Let $w\in I_{\mu,+}$ and put
$u=w^{-1}$, $v=w^{-1}_-$ and $z=w_-$.
Note that $u\in I^\mu_{\phi,+},$
$v\in I^\mu_{\phi,-}$ and $z\in I_{\mu,-}.$

The first step is to compare the algebras
${}^w\!A_{\mu,+}=\End_{{}^w\bfO_{\mu,+}}({}^w\!P_{\mu,+})^{\op}$ and
${}^v\!\bar A^\mu_{\phi,-}=\Ext_{{}^{v}\bfO^\mu_{\phi,-}}({}^{v}\!L^\mu_{\phi,-})^\op,$
and, then, the algebras 
${}^w\!\bar A_{\mu,+}=\Ext_{{}^{w}\bfO_{\mu,+}}({}^{w}\!L_{\mu,+})^\op$ and
${}^v\!A^\mu_{\phi,-}=\End_{{}^v\bfO^\mu_{\phi,-}}({}^v\!P^\mu_{\phi,-})^{\op}.$
More precisely, we prove the following.

\vspace{2mm}

\begin{prop} 
\label{prop:reg1}
If $x\in{}^w\!I_{\mu,+}$ and
$y=x_-^{-1},$ then $y\in{}^v\!I_{\phi,-}^\mu$.
We have a $k$-algebra isomorphism
${}^w\!A_{\mu,+}={}^v\!\bar A^\mu_{\phi,-}$ such that
$1_x\mapsto 1_y$ for each
$x\in{}^w\!I_{\mu,+}$.
We have a $k$-algebra isomorphism
${}^w\!\bar A_{\mu,+}= {}^v\!A^\mu_{\phi,-}$
such that $1_x\mapsto 1_y$ for each $x\in{}^w\!I_{\mu,+}$.
The graded $k$-algebras 
${}^w\!\bar A_{\mu,+}$ and
${}^v\!\bar A^\mu_{\phi,-}$
are Koszul and are Koszul dual to each other.
Further, we have $1_x=1_y^!$ for each $x\in{}^w\!I_{\mu,+}$.
\end{prop}

\vspace{.5mm}

\begin{proof}
By Corollaries \ref{cor:2.31}, \ref{cor:loc4}, composing $\bbH$, $\bbV_k$
yields $k$-algebra isomorphisms
\begin{equation}\label{4.1a}{}^w\!A_{\mu,+}={}^w\! \scrA_{\mu,+}={}^v\!\bar A^\mu_{\phi,-},
\end{equation}
which identify the idempotents
$1_{y}\in {}^{v}\!\bar A^\mu_{\phi,-}$
and $1_x\in{}^w\!A_{\mu,+},{}^w\!\scrA_{\mu,+}$
for each $x\in{}^w\!I_{\mu,+}$ and
$y=x_-^{-1}.$

Now, we claim that 
${}^v\!A^\mu_{\phi,-}$ 
has a Koszul grading.
By Lemma \ref{lem:1.2} and Section \ref{sec:tau}, 
it is enough to check that 
${}^v\! A_{\phi,-}$ has a Koszul grading.
This follows from the matrix equation in Proposition \ref{prop:2.24}
and from \cite[thm.~2.11.1]{BGS}, because
${}^v\!A_{\phi,-}={}^v\! \scrA_{\phi,-}$ as $k$-algebras by 
Corollary \ref{cor:2.31}.

Equip ${}^v\! A^\mu_{\phi,-}$ with the Koszul grading above. 
Then, Lemma \ref{lem:1.1} implies that
\begin{equation}\label{4.1b}
{}^v\!A^{\mu,!}_{\phi,-}=\Ext_{{}^v\bfO^\mu_{\phi,-}}({}^v\!L^\mu_{\phi,-})^\op={}^v\!\bar A^\mu_{\phi,-}.
\end{equation}
Thus,  the graded $k$-algebra ${}^v\!\bar A^\mu_{\phi,-}$ is also Koszul.
Since
${}^w\!A_{\mu,+}={}^v\!\bar A^\mu_{\phi,-}$ as a $k$-algebra,
this implies that the $k$-algebra ${}^w\!A_{\mu,+}$ has a Koszul grading. 

Applying Lemma \ref{lem:1.1} once again, we get
\begin{equation}\label{4.1c}
{}^w\!A_{\mu,+}^!=\Ext_{{}^w\bfO_{\mu,+}}({}^w\!L_{\mu,+})^\op={}^w\!\bar A_{\mu,+}.
\end{equation}
In particular, the graded $k$-algebra ${}^w\!\bar A_{\mu,+}$ is also Koszul.

Finally, using \eqref{4.1b}, \eqref{4.1a} and \eqref{4.1c} we get $k$-algebra isomorphisms
$${}^v\!A^\mu_{\phi,-}=
{}^v\!\bar A^{\mu,!}_{\phi,-}=
{}^w\!A_{\mu,+}^!={}^w\!\bar A_{\mu,+}.$$ 
They identify the idempotent
$1_y\in {}^{v}\! A^\mu_{\phi,-}$
with the idempotent
$1_x\in{}^w\!\bar A_{\mu,+}.$
\end{proof}

\vspace{2mm}

\begin{rk}
The Koszul grading on ${}^v\!A^\mu_{\phi,-}$ 
can also be obtained 
using mixed perverse sheaves on the ind-scheme 
$X'$ as in \cite[thm.~4.5.4]{BGS}, \cite{AK}. 
%Our argument via moment graphs is elementary. 
Note that there is no analogue of \cite[lem.~3.9.2]{BGS} in our situation,
because ${}^v\!R_{\phi,-}$  is not Koszul self-dual.
%Note also that there is no analogue of the localization
%functor $\Phi$ in Proposition \ref{prop:loc2} for positive levels.
\end{rk}

\vspace{2mm}

The second step consists of comparing the $k$-algebras
${}^z\!\bar A_{\mu,-}=\Ext_{{}^{z}\bfO_{\mu,-}}({}^{z}\!L_{\mu,-})^\op,$
${}^u\!A^\mu_{\phi,+}=\End_{{}^u\bfO^\mu_{\phi,+}}({}^u\!P^\mu_{\phi,+})^{\op}$
and the $k$-algebras
${}^z\!A_{\mu,-}=\End_{{}^z\bfO_{\mu,+}}({}^w\!P_{\mu,-})^{\op},$
${}^u\!\bar A^\mu_{\phi,+}=\Ext_{{}^{u}\bfO^\mu_{\phi,+}}({}^{u}\!L^\mu_{\phi,+})^\op.$

We can not argue as in the previous proposition, because we have no analogue of the localization functor
$\Phi$ in Proposition \ref{prop:loc2} for positive levels. 
Hence, we have no analogue of Proposition \ref{prop:loc3} and 
Corollary \ref{cor:loc4}.

To remedy this, we'll use the technic of standard Koszul duality.
To do so, we need the following crucial result.

\vspace{2mm}

\begin{lemma}
\label{lem:balanced}
The quasi-hereditary $k$-algebra
${}^w\!A_{\mu,+}$ has a balanced grading.
\end{lemma}

\vspace{.5mm}

Now we can prove the second main result of this section.

\vspace{2mm}

\begin{prop} 
\label{prop:reg2}
If $x\in{}^z\!I_{\mu,-}$ and $y=x_+^{-1}$ then
$y\in{}^u\!I^\mu_{\phi,+}$.
We have a $k$-algebra isomorphism
${}^z\!\bar A_{\mu,-}={}^{u}\! A^\mu_{\phi,+}$
such that
$1_x\mapsto 1_{y}$ for each $x\in{}^z\!I_{\mu,-}$,
and a $k$-algebra isomorphism
${}^z\!A_{\mu,-}={}^{u}\!\bar A^\mu_{\phi,+}$
such that
$1_x\mapsto 1_y$ for each $x\in{}^z\!I_{\mu,-}$.
The graded $k$-algebras ${}^{u}\!\bar A^\mu_{\phi,+}$ and 
${}^z\!\bar A_{\mu,-}$ are Koszul and are Koszul dual to each other.
Further, we have 
 $1_x=1_y^!$ for each $x\in{}^z\!I_{\mu,-}$.
\end{prop}

\vspace{.5mm}

\begin{proof}
By Proposition \ref{prop:ringel}, the Ringel dual of ${}^w\!A_{\mu,+}$ is
${}^w\!A_{\mu,+}^\diamond={}^z\!A_{\mu,-}$.
Thus, since the $k$-algebra ${}^w\!A_{\mu,+}$ has a balanced grading by Lemma \ref{lem:balanced},
we deduce that the $k$-algebra ${}^z\!A_{\mu,-}$ has a  Koszul grading.

We equip ${}^z\!A_{\mu,-}$ with this grading.
Then, Lemma \ref{lem:1.1} implies that
\begin{equation}\label{4.2b}
{}^z\!A^!_{\mu,-}=\Ext_{{}^z\bfO_{\mu,-}}({}^z\!L_{\mu,-})^\op={}^z\!\bar A_{\mu,-}.
\end{equation}
Hence, the graded $k$-algebra ${}^z\!\bar A_{\mu,-}$ is Koszul and balanced.

Therefore, \cite[thm.~1]{Ma}, Proposition \ref{prop:ringel} and \eqref{4.2b} yield a $k$-algebra isomorphism
$${}^z\!\bar A_{\mu,-}=
{}^z\!A_{\mu,-}^!=
(({}^z\!A^\diamond_{\mu,-})^!)^\diamond=
({}^w\! A_{\mu,+}^!)^\diamond.$$
Hence, using \eqref{4.1c} and Propositions \ref{prop:reg1}, \ref{prop:ringel}, we get
a $k$-algebra isomorphism
$${}^z\!\bar A_{\mu,-}=
{}^w\! \bar A_{\mu,+}^\diamond=
{}^v\!  A^{\mu,\diamond}_{\phi,-}=
{}^{u}\!  A^\mu_{\phi,+}$$
such that $1_x\mapsto 1_y$.
So, the $k$-algebra
${}^{u}\! A^{\mu}_{\phi,+}$ has a Koszul grading which is balanced.

We equip ${}^{u}\! A^{\mu}_{\phi,+}$ with this grading.
Lemma \ref{lem:1.1} implies that
\begin{equation}\label{4.2c}
{}^{u}\! A^{\mu,!}_{\phi,+}=\Ext_{{}^u\bfO^\mu_{\phi,+}}({}^u\!L^\mu_{\phi,+})^\op={}^{u}\! \bar A^{\mu}_{\phi,+}.
\end{equation}
Hence, the graded $k$-algebra ${}^{u}\! \bar A^{\mu}_{\phi,+}$ is Koszul and balanced.

Therefore, \cite[thm.~1]{Ma}, Proposition \ref{prop:ringel}
yield
$${}^{u}\!\bar A^\mu_{\phi,+}=
{}^{u}\!A^{\mu,!}_{\phi,+}=
(({}^{u}\!A^{\mu,\diamond}_{\phi,+})^!)^\diamond=
({}^v\!  A^{\mu,!}_{\phi,-})^\diamond.$$
So, using \eqref{4.1b} and Propositions \ref{prop:reg1}, \ref{prop:ringel}, we get
a $k$-algebra isomorphism
$${}^{u}\!\bar A^\mu_{\phi,+}=({}^v\! \bar A^\mu_{\phi,-})^\diamond=
{}^w\!  A^\diamond_{\mu,+}=
{}^z\!  A_{\mu,-}
.$$
\end{proof}

\vspace{2mm}

Now, let us prove Lemma \ref{lem:balanced}.

\vspace{2mm}

\begin{proof}[Proof of Lemma \ref{lem:balanced}]
We equip
${}^w\!A_{\mu,+}$ 
with the grading given by ${}^v\!\bar A^\mu_{\phi,-}$, see 
Proposition \ref{prop:reg1}. 
Since ${}^v\!\bar A^\mu_{\phi,-}$ is Koszul, we must prove that
${}^v\!\bar A^{\mu,!}_{\phi,-}$ is quasi-hereditary and that
the grading on ${}^v\!\bar A^{\mu,\diamond}_{\phi,-}$ is positive, see Section 
\ref{sec:standardkoszul}.

We have ${}^v\!\bar A^{\mu,!}_{\phi,-}={}^w\!\bar A_{\mu,+}$ 
by Proposition \ref{prop:reg1}.
Hence, it is quasi-hereditary. 
Next, the grading of ${}^z\! \bar \scrA_{\mu,-}$ is positive by Remark \ref{rk:3.29}.
Thus, to prove the lemma, it is enough to check
that we have a graded $k$-algebra isomorphism
${}^v\!\bar A^{\mu,\diamond}_{\phi,-}={}^z\!\bar \scrA_{\mu,-}.$

First, we check that
${}^w\! A^\diamond_{\mu,+}={}^z\! \scrA_{\mu,-}$ as (ungraded) $k$-algebras. 

To do this, note that
Proposition \ref{prop:ringel}$(c)$ yields a $k$-algebra isomorphism 
${}^w\! A_{\mu,+}^\diamond
={}^z\! A_{\mu,-}.$
Further, we have ${}^z\! A_{\mu,-}
=\End_{{}^z\bfO_{\mu,-}}({}^z\! P_{\mu,-})^\op$ and, by Proposition \ref{prop:ringel}$(b),$
 the Ringel equivalence takes
${}^z\bfO_{\mu,-}^\Delta$ to ${}^z\bfO_{\mu,+}^\Delta$. 
Thus, Propositions  \ref{prop:2.3}(b),
\ref{prop:foncteurV}$(c)$ yield
$$\aligned
{}^w\! A_{\mu,+}^\diamond
&=\End_{{}^z\bfO_{\mu,+}}({}^z T_{\mu,+})^\op\\
&=k\End_{{}^z\bfO_{S_0,\mu,+}}({}^z T_{S_0,\mu,+})^\op\\
&=k\End_{{}^z\!Z_{S_0,\mu,+}}({}^z\! C_{S_0,\mu,+})^\op\\
&={}^z\! \scrA_{\mu,-},
\endaligned$$
where the last equality is Definition \ref{def:3.23}.

Now, we must identify the gradings of ${}^v\!\bar A^{\mu,\diamond}_{\phi,-}$ and
${}^z\!\bar \scrA_{\mu,-}$ under the isomorphism
${}^w\! A_{\mu,+}^\diamond={}^z\! \scrA_{\mu,-}$
above.

First, we consider the grading on ${}^v\!\bar A^{\mu,\diamond}_{\phi,-}$.
Let ${}^w\bar T_{\mu,+}$ be the graded ${}^v\!\bar A^\mu_{\phi,-}$-module
equal to ${}^wT_{\mu,+}$ as an ${}^w\! A_{\mu,+}$-module,
with the natural grading (this is well-defined, because ${}^v\!\bar A^\mu_{\phi,-}$ is Koszul).
Then, by Section \ref{sec:standardkoszul}, we have
${}^v\!\bar A^{\mu,\diamond}_{\phi,-}=
\End_{{}^w\! A_{\mu,+}}({}^w\bar T_{\mu,+})^\op$.

Next, we consider the grading on ${}^z\!\bar \scrA_{\mu,-}$.
By Corollary \ref{cor:B2}, we have
$
{}^z\!\bar \scrA_{\mu,-}=
k\End_{{}^w\! Z_{S,\mu,+}}\bigl({}^w\!\bar C_{S,\mu,+}\bigr)
=\End_{{}^w\! Z_{\mu,+}}\bigl({}^w\!\bar C_{\mu,+}\bigr).
$

Let us identify ${}^w\!A_{\mu,+}={}^w\! \scrA_{\mu,+}$ via \eqref{4.1a}.
We must prove the following.

\vspace{2mm}

\begin{claim}\label{claim:4.7} There is a graded $k$-algebra isomorphism
$$\End_{{}^w\! \scrA_{\mu,+}}({}^w\bar T_{\mu,+})=
\End_{{}^w\! Z_{\mu,+}}\bigl({}^w\!\bar C_{\mu,+}\bigr).$$
\end{claim}

\vspace{.5mm}

To do that, we'll need some new material.
%We must introduce a deformed analogue of $\End_{{}^w\! A_{\mu,+}}({}^w\bar T_{\mu,+})$.

Consider the graded categories given by
${}^w\bar\bfO_{\mu,+}={}^w\! \bar \scrA_{\mu,+}\bfgmod$ and
${}^w\bar\bfO_{S,\mu,+}={}^w\!\bar \scrA_{S,\mu,+}\bfgmod.$
Since we have ${}^w\bar T_{\mu,+}\in{}^v\!\bar A^\mu_{\phi,-}\bfgmod$ and
${}^w\! \bar \scrA_{\mu,+}={}^v\!\bar A^\mu_{\phi,-}$ by Corollary \ref{cor:loc4},
we may view 
${}^w\bar T_{\mu,+}$ as an object of ${}^w\bar\bfO_{\mu,+}$.

We have an isomorphism of graded $S$-algebras
${}^w\!\bar \scrA_{S,\mu,+}=
\End_{{}^w\! Z_{S,\mu,+}}\bigl({}^w\!\bar B_{S,\mu,+}\bigr)^{\op}$.
Consider the pair of adjoint functors $(\bar\bbV_S,\bar\psi_S)$ between 
${}^w\bar\bfO_{S,\mu,+}$ and ${}^w\bar Z_{S,\mu,+}\bfgmod$
given by
$$\gathered
\bar\bbV_{S}={}^w\!\bar B_{S,\mu,+}\otimes_{{}^w\!\bar \scrA_{S,\mu,+}}\bullet: 
{}^w\bar\bfO_{S,\mu,+}\to{}^w\bar Z_{S,\mu,+}\bfgmod,\\
\bar\psi_S=\Hom_{{}^w\! Z_{S,\mu,+}}\bigl({}^w\!\bar B_{S,\mu,+},\bullet\bigr):
{}^w\bar Z_{S,\mu,+}\bfgmod\to{}^w\bar\bfO_{S,\mu,+}.
\endgathered$$ 
We consider the object 
${}^w\bar T_{S,\mu,+}$ of ${}^w\bar\bfO_{S,\mu,+}$ given by
${}^w\bar T_{S,\mu,+}=\bar\psi_S({}^w\bar C_{S,\mu,+})$.

Recall the functor $\bbV_{S_0}:{}^w\bfO_{S_0,\mu,+}^\Delta\to{}^w\bfZ_{S_0,\mu,+}^\Delta$
studied in Proposition \ref{prop:foncteurV}.
It yields an $S_0$-algebra isomorphism
$\End_{{}^w\bfO_{S_0,\mu,+}}({}^wP_{S_0,\mu,+})^\op\to{}^w\! \scrA_{S_0,\mu,+}.$
Thus, 
since ${}^wP_{S_0,\mu,+}$ is a pro-generator of 
${}^w\bfO_{S_0,\mu,+}$,
the functor $\phi_{S_0}=\Hom_{{}^w\bfO_{S_0,\mu,+}}({}^w\!P_{S_0,\mu,+},\bullet)$ identifies
${}^w\bfO_{S_0,\mu,+}$ with the category
${}^w\!\scrA_{S_0,\mu,+}\bfmod.$

Consider the functor
$\psi_{S_0}:
{}^wZ_{S_0,\mu,+}\bfmod\to{}^w\!\scrA_{S_0,\mu,+}\bfmod$
given by
$\psi_{S_0}=\Hom_{{}^w\! Z_{S_0,\mu,+}}\bigl({}^w\!B_{S_0,\mu,+},\bullet\bigr)$. 
We have the commutative triangle
\begin{equation}\label{triangle}\begin{split}\xymatrix{
{}^w\bfO_{S_0,\mu,+}^\Delta\ar[r]^-{\bbV_{S_0}}_-\sim\ar@{_{(}->}[d]_-{\phi_{S_0}}&
{}^w\bfZ_{S_0,\mu,+}^\Delta\ar@{^{(}->}[ld]^-{\psi_{S_0}}\cr
{}^w\!\scrA_{S_0,\mu,+}\bfmod,& }\end{split}\end{equation}
such that $\bbV_{S_0}({}^w\!T_{S_0,\mu,+})={}^w\!C_{S_0,\mu,+}$.
Now, we define the module ${}^wT'_{S_0,\mu,+}=\psi_{S_0}({}^w\!C_{S_0,\mu,+})$
in ${}^w\!\scrA_{S_0,\mu,+}\bfmod$.
It is identified with ${}^wT_{S_0,\mu,+}$ by $\phi_{S_0}$, because \eqref{triangle} is commutative.
By Proposition \ref{prop:foncteurV}$(c)$, the functor
$\psi_{S_0}$ yields a $k$-algebra isomorphism
$\End_{{}^w\! Z_{S_0,\mu,+}}\bigl({}^w\!C_{S_0,\mu,+}\bigr)\to
\End_{{}^w\! \scrA_{S_0,\mu,+}}\bigl({}^w\!T'_{S_0,\mu,+}\bigr).$
Note that $\phi_{S_0}$ and $\psi_{S_0}$ are both exact.

Since taking hom's commutes with localization, we have
$\varepsilon({}^w\!\bar \scrA_{S,\mu,+})={}^w\!\scrA_{S_0,\mu,+}$.
Thus,  forgetting the grading and taking base change yield a functor 
$\varepsilon:{}^w\bar\bfO_{S,\mu,+}\to{}^w\bfO_{S_0,\mu,+}$
which is faithful and faithfully exact, see Section \ref{rk:basechange}.
Let ${}^w\bar\bfO_{S,\mu,+}^\Delta\subset{}^w\bar\bfO_{S,\mu,+}$
be the full subcategory of the modules taken to
${}^w\bfO_{S_0,\mu,+}^\Delta$ by $\varepsilon$.

We identify ${}^w\bfO_{S_0,\mu,+}$ with 
${}^w\!\scrA_{S_0,\mu,+}\bfmod$ by $\phi_{S_0}$.
We have the following.

\vspace{2mm}

\begin{claim}\label{claim:4.6} 
We have the following commutative diagram
\begin{equation*}
\label{square}
\begin{split}
\xymatrix{
{}^w\bar\bfO_{S,\mu,+}^\Delta\ar[r]^-{\bar\bbV_{S}}
\ar[d]_-\varepsilon&{}^w\bar\bfZ_{S,\mu,+}^\Delta\ar[d]^-\varepsilon
\ar[r]^-{\bar\psi_{S}}&{}^w\bar\bfO_{S,\mu,+}^\Delta\ar[d]_-\varepsilon
\cr
{}^w\bfO_{S_0,\mu,+}^\Delta\ar[r]^-{\bbV_{S_0}}&{}^w\bfZ_{S_0,\mu,+}^\Delta
\ar[r]^-{\psi_{S_0}}&{}^w\bfO_{S_0,\mu,+}^\Delta.
}
\end{split}
\end{equation*}
\end{claim}

\vspace{2mm}

Since  ${}^w\!\scrA_{S_0,\mu,+}$ is Noetherian, any finitely generated module has a finite presentation.
By Proposition \ref{prop:foncteurV}$(c),$ the functor $\bbV_{S_0}$ is exact and we have
$\bbV_{S_0}({}^w\!P_{S_0,\mu,+})={}^w\!B_{S_0,\mu,+}$.
We identify ${}^w\bfO_{S_0,\mu,+}$ with
${}^w\!\scrA_{S_0,\mu,+}\bfmod$ as above.
By the five lemma, the canonical morphism 
${}^w\!B_{S_0,\mu,+}\otimes_{{}^w\!\scrA_{S_0,\mu,+}}\bullet\to\bbV_{S_0}$
is an isomorphism.
We deduce that the left square in Claim \ref{claim:4.6} is commutative.

Next, by definition we have $\varepsilon\circ\bar\psi_S=\psi_{S_0}\circ\varepsilon$.
From the diagram \ref{triangle}, we deduce that $\psi_{S_0}$ maps into
${}^w\bfO_{S_0,\mu,+}^\Delta$ and is a quasi-inverse to $\bbV_{S_0}$. 
Claim \ref{claim:4.6} follows.

%The functor $\bbV_{S_0}$ is exact and 
%$\varepsilon:{}^w\bar\bfZ_{S,\mu,+}^\Delta\to{}^w\bfZ^\Delta_{S_0,\mu,+}$
%is faithfully exact by Remark \ref{rk:basechange}. 
%Using claim \ref{claim:4.6}, we deduce that 
%$\bar\bbV_S$ is exact on ${}^w\bar\bfO_{S,\mu,+}^\Delta.$
%we deduce that the functor $\bar\bbV_S$ is fully faithful and exact on ${}^w\bar\bfO_{S,\mu,+}^\Delta.$

%Next, we consider the right adjoint $\bar\psi_S$ to $\bar\bbV_S$, which is the composition of
%$\Hom_{{}^w\! Z_{S,\mu}}\bigl({}^w\!\bar B_{S,\mu,+},\bullet\bigr):
%{}^w\bar\bfZ_{S,\mu,+}^\Delta\to{}^w\bar\bfO_{S,\mu,+}^\Delta$ and the right adjoint to the inclusion
%${}^w\bar\bfO_{S,\mu,+}^\Delta\subset{}^w\bar\bfO_{S,\mu,+}$.
%Similarly,  the right adjoint $\psi_{S_0}$ to $\bbV_{S_0}$ is the composition of
%$\Hom_{{}^w\! Z_{S_0,\mu}}\bigl({}^w\! B_{S_0,\mu,+},\bullet\bigr)$
%and the right adjoint to the inclusion
%${}^w\bfO_{S_0,\mu,+}^\Delta\subset{}^w\bfO_{S_0,\mu,+}$.  

From claim \ref{claim:4.6} we deduce that
$\varepsilon({}^w\bar T_{S,\mu,+})={}^wT_{S_0,\mu,+}$
and ${}^w\bar T_{S,\mu,+}\in{}^w\bar\bfO_{S,\mu,+}^\Delta$.

Next, since 
$\varepsilon:S\bfgmod\to S_0\bfmod$ is faithfully exact by Section \ref{rk:basechange}
and $\psi_{S_0}$ gives an isomorphism $\End_{{}^w\! Z_{S_0,\mu,+}}\bigl({}^w\!C_{S_0,\mu,+}\bigr)\to
\End_{{}^w\! \scrA_{S_0,\mu,+}}\bigl({}^w\!T_{S_0,\mu,+}\bigr)$, we deduce from claim \ref{claim:4.6} that
$\bar\psi_S$ gives a graded $k$-algebra isomorphism 
$\End_{{}^w\! Z_{S,\mu,+}}\bigl({}^w\!\bar C_{S,\mu,+}\bigr)\to
\End_{{}^w\! \scrA_{S,\mu,+}}\bigl({}^w\!\bar T_{S,\mu,+}\bigr).$

Finally, since the functor $\psi_{S_0}$ is exact and 
$\varepsilon$
%:{}^w\bar\bfZ_{S,\mu,+}^\Delta\to{}^w\bfZ^\Delta_{S_0,\mu,+}$
is faithfully exact by Section \ref{rk:basechange}, we deduce from
claim \ref{claim:4.6} that 
$\bar\psi_S$ is exact on ${}^w\bar\bfZ_{S,\mu,+}^\Delta.$

%The functor $\bbV_{S_0}$ is fully faithful by Proposition \ref{prop:foncteurV}$(c)$, i.e., 
%the unit $1\to\psi_{S_0}\circ\bbV_{S_0}$ is an isomorphism on ${}^w\bfO_{S_0,\mu,+}^\Delta$.
%Since $\varepsilon$ is faithfully exact, we deduce that the unit 
%$1\to\bar\psi_S\circ\bar\bbV_S$
%is an isomorphism
%on ${}^w\bar\bfO_{S,\mu,+}^\Delta$, i.e.,
%the functor $\bar\bbV_S$ is fully faithful.

Now, by Definition \ref{def:3.23}, we have ${}^w\!\bar \scrA_{\mu,+}=k{}^w\!\bar \scrA_{S,\mu,+}$.
Thus, the specialization at $k$ gives the module 
$k{}^w\bar T_{S,\mu,+}$ in ${}^w\bar\bfO_{\mu,+}={}^w\! \bar \scrA_{\mu,+}\bfgmod.$

\vspace{2mm}

\begin{claim}\label{claim:4.8} We have $k{}^w\bar T_{S,\mu,+}={}^w\bar T_{\mu,+}$.
\end{claim}

\vspace{.5mm}

The specialization gives a graded ring homomorphism
$k\End_{{}^w\! \scrA_{S,\mu,+}}({}^w\bar T_{S,\mu,+})\to
\End_{{}^w\! \scrA_{\mu,+}}(k{}^w\bar T_{S,\mu,+}).$
Since $\varepsilon({}^w\bar T_{S,\mu,+})={}^wT_{S_0,\mu,+}$, forgetting the grading it is taken to the obvious
map $k\End_{{}^w\! \bfO_{S_0,\mu,+}}({}^wT_{S_0,\mu,+})\to
\End_{{}^w\! \bfO_{\mu,+}}(k{}^w T_{S_0,\mu,+}).$
The latter is an isomorphism by
Proposition \ref{prop:2.3}$(b)$.
Since $k{}^w\bar T_{S,\mu,+}={}^w\bar T_{\mu,+}$, we deduce that
$
k\End_{{}^w\! \scrA_{S,\mu,+}}({}^w\bar T_{S,\mu,+})=
\End_{{}^w\! \scrA_{\mu,+}}({}^w\bar T_{\mu,+}).
$

Since we have $k\End_{{}^w\! Z_{S,\mu,+}}\bigl({}^w\!\bar C_{S,\mu,+}\bigr)
=\End_{{}^w\! Z_{\mu}}\bigl({}^w\!\bar C_{\mu,+}\bigr)$ and
%$\bar\bbV_S({}^w \bar T_{S,\mu,+})={}^w\!\bar C_{S,\mu,+}$
$\End_{{}^w\! Z_{S,\mu,+}}\bigl({}^w\!\bar C_{S,\mu,+}\bigr)=
\End_{{}^w\! \scrA_{S,\mu,+}}\bigl({}^w\!\bar T_{S,\mu,+}\bigr),$
%$\bar\bbV_S$ is fully faithful on ${}^w\bar\bfO_{S,\mu,+}^\Delta,$
this implies Claim \ref{claim:4.7}.

%Since ${}^w\!\bar B_{S,\mu,+}$ is projective in ${}^w\bar\bfZ_{S,\mu,+}^\Delta$
%by Proposition \ref{prop:B} and Remark \ref{rk:Fprojective}, the  functor $\bar\psi_S$ is exact on
%${}^w\bar\bfZ_{S,\mu,+}^\Delta$. 
%We have ${}^wT_{S_0,\mu,+}=\psi_{S_0}({}^w\! C_{S_0,\mu,+})$
%by Proposition \ref{prop:foncteurV}$(c)$.
%We deduce that ${}^w\bar T_{S,\mu,+}\in{}^w\bar\bfO_{S,\mu,+}^\Delta$, because
%${}^w T_{S_0,\mu,+}\in{}^w\bfO_{S_0,\mu,+}^\Delta$ and
%$${}^wT_{S_0,\mu,+}=\psi_{S_0}({}^w\! C_{S_0,\mu,+})=
%\psi_{S_0}\varepsilon({}^w\! C_{S,\mu,+})=
%\varepsilon({}^w\bar T_{S,\mu,+}).$$

%Next, since $\bar\bbV_{S_0}\circ\bar\psi_{S_0}\simeq 1$
%on ${}^w\bfZ_{S_0,\mu,+}^\Delta$ by Proposition \ref{prop:foncteurV}$(c)$, 
%the same argument as above using $\varepsilon$ implies that we have
%$\bar\bbV_S\circ\bar\psi_S\simeq 1$
%on ${}^w\bar\bfZ_{S,\mu,+}^\Delta$. We deduce that 
%$\bar\bbV_S({}^w \bar T_{S,\mu,+})={}^w\!\bar C_{S,\mu,+}$.

Now, we prove Claim \ref{claim:4.8}.
The grading of ${}^w\bar T_{\mu,+}$ is characterized in the following way.
First, we have a decomposition ${}^w\bar T_{\mu,+}=\bigoplus_x{}^w\bar T(x\bullet\oo_{\mu,+})$
where ${}^w\bar T(x\bullet\oo_{\mu,+})$ is the natural graded lift of ${}^wT(x\bullet\oo_{\mu,+})$.
Since ${}^w T(x\bullet\oo_{\mu,+})$ is indecomposable in ${}^w\bfO_{\mu,+}$, 
it admits at most one graded lift in
${}^w\bar\bfO_{\mu,+}$, up to grading shift.
Let $\bar V(x\bullet\oo_{\mu,+})$ be the graded ${}^v\!\bar A^\mu_{\phi,-}$-module
equal to $V(x\bullet\oo_{\mu,+})$ as an ${}^w\! A_{\mu,+}$-module,
with the natural grading (this is well-defined, because ${}^v\!\bar A^\mu_{\phi,-}$ is Koszul).
The natural grading is characterized
by the fact that there is an inclusion
$\bar V(x\bullet\oo_{\mu,+})\subset 
{}^w\bar T(x\bullet\oo_{\mu,+})$
which is homogeneous of degree 0.

The graded object ${}^w\bar T_{S,\mu,+}$ satisfies a similar property.
Indeed, by Lemma \ref{rk:filtration2},
the graded $S$-sheaf ${}^w\!\bar C_{S,\mu,+}(x)$ 
is filtered by shifted Verma-sheaves, and the lower term of this filtration yields
an inclusion
$\bar V_{S,\mu}(x)\langle l(x)\rangle\subset {}^w\bar C_{S,\mu,+}(x).$
Further, consider the 
object ${}^w\bar T_S(x\bullet\oo_{\mu,+})$
in ${}^w\bar\bfO_{S,\mu,+}$ given by
${}^w\bar T_S(x\bullet\oo_{\mu,+})=\bar\psi_S({}^w\bar C_{S,\mu,+}(x))$.
Then, we have the decomposition
${}^w\bar T_{S,\mu,+}=\bigoplus_x{}^w\bar T_S(x\bullet\oo_{\mu,+})$.

For each $x\in{}^wI_{\mu,+},$ we consider also
the 
object $\bar V_S(x\bullet\oo_{\mu,+})$
in ${}^w\bar\bfO_{S,\mu,+}$ given by
$\bar V_{S}(x\bullet \oo_{\mu,+})=\bar\psi_S(\bar V_{S,\mu}(x))\langle l(x)\rangle$.
Since $\bar\psi_S$ is exact, we deduce that
${}^w\bar T_S(x\bullet\oo_{\mu,+})$ 
is filtered by $\bar V_{S}(y\bullet \oo_{\mu,+})$'s, and the lower term of this filtration yields
an inclusion 
$\bar V_{S}(x\bullet \oo_{\mu,+})\subset {}^w\bar T_S(x\bullet\oo_{\mu,+})$
which is homogeneous of degree 0.

Now, by Proposition \ref{prop:foncteurV}$(c),$ we have
$\bbV_{S_0}(V_{S_0}(x\bullet \oo_{\mu,+}))=V_{S_0,\mu}(x)$.
From the proof of claim \ref{claim:4.6}, we deduce that
$\psi_{S_0}(V_{S_0,\mu}(x))=V_{S_0}(x\bullet \oo_{\mu,+})$.
Thus
$$\varepsilon(\bar V_{S}(x\bullet \oo_{\mu,+}))=
\varepsilon\bar\psi_S(\bar V_{S,\mu}(x))=
\psi_{S_0}\varepsilon(\bar V_{S,\mu}(x))=V_{S_0}(x\bullet \oo_{\mu,+}).$$
Since $V_{S_0}(x\bullet \oo_{\mu,+})$ is free over $S_0,$
we deduce that $\bar V_{S}(x\bullet \oo_{\mu,+})$ is free over $S$ by  Section \ref{rk:basechange}.
Therefore, the inclusion $\bar V_{S}(x\bullet \oo_{\mu,+})\subset {}^w\bar T_S(x\bullet\oo_{\mu,+})$ 
gives an inclusion
$k\bar V_{S}(x\bullet \oo_{\mu,+})\subset k{}^w\bar T_S(x\bullet\oo_{\mu,+})$
which is homogeneous of degree 0.

Hence, to prove Claim \ref{claim:4.8} we are reduced to check that  we have
$\bar V(x\bullet\oo_{\mu,+})=k\bar V_{S}(x\bullet \oo_{\mu,+})$
in ${}^w\bar\bfO_{\mu,+}$.
Note that $\bar V(x\bullet\oo_{\mu,+})$ and $k\bar V_{S}(x\bullet \oo_{\mu,+})$ are
both graded lifts of
the Verma module $V(x\bullet\oo_{\mu,+}),$ which is indecomposable.
Thus they coincide up to a grading shift.

To identify this shift, recall that by Lemma \ref{rk:filtration2} we have
a surjection
${}^w\!\bar B_{S,\mu,+}(x)\to\bar V_{S}(x)\langle l(x)\rangle.$
We define ${}^w\bar P_{S}(x\bullet\oo_{\mu,+})=\bar\psi_S({}^w\!\bar B_{S,\mu,+}(x)).$
Since $\bar\psi_S$ is exact, we have a surjection
${}^w\!\bar P_{S}(x\bullet\oo_{\mu,+})\to\bar V_{S}(x\bullet\oo_{\mu,+}).$

Now, since ${}^v\!\bar A^\mu_{\phi,-}$ is Koszul, we can
consider the natural graded lift ${}^w\!\bar P(x\bullet\oo_{\mu,+})$ of
${}^w\!P(x\bullet\oo_{\mu,+})$ in ${}^w\bar\bfO_{\mu,+}$. By definition of the natural grading, 
we have a surjection
${}^w\!\bar P(x\bullet\oo_{\mu,+})\to\bar V(x\bullet\oo_{\mu,+}).$

So, we must prove that
${}^w\!\bar P(x\bullet\oo_{\mu,+})=k{}^w\!\bar P_{S}(x\bullet \oo_{\mu,+})$ in ${}^w\bar\bfO_{\mu,+}.$
This is obvious, because 
${}^w\!\bar P(x\bullet\oo_{\mu,+})={}^w\! \bar \scrA_{\mu,+}1_x$ 
by definition of the natural grading of ${}^w\! P(x\bullet\oo_{\mu,+}),$
and because Definition \ref{def:3.23} yields the chain of isomorphisms
$${}^w\bar P_{S}(x\bullet\oo_{\mu,+})=\bar\psi_S({}^w\!\bar B_{S,\mu,+}(x))
=\End_{{}^w\! Z_{S,\mu,+}}\bigl({}^w\!\bar B_{S,\mu,+}\bigr)^\op1_x
={}^w\! \bar \scrA_{\mu,+}1_x.$$
\end{proof}

\vspace{2mm}

We can now prove the following graded analogue of Corollary \ref{cor:2.31}.
See also Corollary \ref{cor:loc4}.

\vspace{2mm}

\begin{cor}
\label{cor:2.42}
Assume that $u\in I^\mu_{\phi,+}$. Let $z=u_-^{-1}\in I_{\mu,-}.$
We have an isomorphism of graded $k$-algebras
${}^{u}\!\bar A^\mu_{\phi,+}\to
{}^{z}\!\bar \scrA_{\mu,-}$
such that $1_{y}\mapsto 1_x$ with $x=y_-^{-1}$ for each $y\in{}^v\!I_{\phi,+}^\mu$.
\end{cor}

\vspace{.5mm}

\begin{proof}
By \cite[thm.~1]{Ma} and Lemma
\ref{lem:balanced}, the graded $k$-algebra
${}^v\!\bar A^{\mu,\diamond}_{\phi,-}={}^z\! \bar \scrA_{\mu,-}$ is Koszul.
By Proposition \ref{prop:reg2}, the graded $k$-algebra
%${}^{v}\!\bar R^\mu_{\phi,-}={}^w\! R_{\mu,+}$ as a $k$-algebra and 
${}^u\!\bar A^\mu_{\phi,+}$ is  Koszul and is isomorphic to
${}^w\! A_{\mu,-}$ as a $k$-algebra. 
By Proposition \ref{prop:reg2}, we have
${}^{v}\!\bar A^\mu_{\phi,-}={}^w\! A_{\mu,+}$ as a $k$-algebra.
Finally, by Proposition \ref{prop:ringel}, the Ringel dual of ${}^w\!A_{\mu,+}$ is
${}^w\!A_{\mu,+}^\diamond={}^z\!A_{\mu,-}$.
Thus we have a $k$-algebra isomorphism
${}^z\! A_{\mu,-}={}^z\! \scrA_{\mu,-},$
which lifts to a graded $k$-algebra isomorphism
${}^u\!\bar A^\mu_{\phi,+}={}^{z}\!\bar \scrA_{\mu,-}$
by unicity of the Koszul grading \cite[cor.~2.5.2]{BGS}.
\end{proof}

\vspace{2mm}

\subsection{The general case} 
\label{sec:3.10}
We can now complete the proof of Theorem \ref{thm:main}.
We first prove a series of preliminary lemmas.

Fix parabolic types $\mu,\nu\in\calP$ and integers $d,e,f>0$.
Choose integral weights $\oo_{\mu,-}$,  $\oo_{\nu,-}$,  $\oo_{\phi,-}$ of level
$- e-N$, $-f-N$ and $-d-N$ respectively.

Fix  an element $w\in\widehat W$. Let $\tau_{\phi,\nu}:{}^w\bfO_{\mu,+}\to{}^w\bfO^\nu_{\mu,+}$ be
the parabolic truncation functor, see Section \ref{sec:tau}.
Applying the functor $\tau_{\phi,\nu}$ to the module ${}^w\!P_{\mu,+}$, we get
a $k$-algebra homomorphism
$\tau_{\phi,\nu}:{}^w\!A_{\mu,+}\to{}^w\!A_{\mu,+}^\nu$.

\vspace{2mm}

\begin{lemma}
\label{lem:gen1}
Assume that $w\in I_{\mu,+}^\nu.$ 
Then, the functor the $k$-algebra homomorphism 
$\tau_{\phi,\nu}:{}^w\!A_{\mu,+}\to {}^w\!A^\nu_{\mu,+}$ is surjective.
Its kernel is the two-sided ideal generated by the idempotents
$1_x$ with $x\in{}^w\!I_{\mu,+}\setminus {}^w\!I_{\mu,+}^\nu$.
Further, we have $\tau_{\phi,\nu}(1_x)=1_x$ 
for each $x\in{}^w\!I_{\mu,+}^\nu$.
\end{lemma}

\vspace{.5mm}

\begin{proof}
By Section \ref{sec:tau}, the functor
$\tau_{\phi,\nu}$ takes the minimal projective generator of ${}^w\bfO_{\mu,+}$
to the minimal projective generator of ${}^w\bfO_{\mu,+}^\nu$.
Let $i_{\nu,\phi}$ be the right adjoint of $\tau_{\phi,\nu}$.
For any $M$ the unit $M\to i_{\nu,\phi}\tau_{\phi,\nu}(M)$ is surjective. 
Hence, for any projective module $P$ we have a surjective map
$$\Hom_{{}^w\bfO_{\mu,+}}(P,M)\to
\Hom_{{}^w\bfO_{\mu,+}}(P,i_{\nu,\phi}\tau_{\phi,\nu}(M))=
\Hom_{{}^w\bfO_{\mu,+}^\nu}(\tau_{\phi,\nu}(P),\tau_{\phi,\nu}(M)).$$
Thus, the $k$-algebra homomorphism 
$\tau_{\phi,\nu}$ is surjective.

Let $I\subset {}^w\!A_{\mu,+}$ be the 2-sided ideal generated by the idempotents
$1_x$ such that $x\in{}^w\!I_{\mu,+}$ and $\tau_{\phi,\nu}(1_x)=0$.
By  Section \ref{sec:tau}(c), the latter are precisely the idempotents
$1_x$ with $x\in{}^w\!I_{\mu,+}\setminus {}^w\!I_{\mu,+}^\nu$.

We have a $k$-algebra isomorphism ${}^w\!A_{\mu,+}/I\simeq {}^w\!A^\nu_{\mu,+}$,
because ${}^w\bfO_{\mu,+}^\nu$ is the Serre subcategory of ${}^w\bfO_{\mu,+}$ generated by the
simple modules killed by $I$. Under this isomorphism, the map $\tau_{\phi,\nu}$ is the
canonical projection ${}^w\!A_{\mu,+}\to{}^w\!A_{\mu,+}/I$.

The last claim in the lemma follows from  Section \ref{sec:tau}(b).
\end{proof}

\vspace{2mm}

Now, let $v\in I_{\nu,-}$ and $w=v^{-1}_+$. We have $w\in I^\nu_{\phi,+}.$

We equip the $k$-algebras 
${}^v\!A_{\nu,-}$,
${}^v\!A_{\phi,-}$
with the gradings
${}^w\!\bar A^\nu_{\phi,+}$,
${}^w\!\bar A_{\phi,+}$,
see Proposition \ref{prop:reg2}.
Consider the categories 
${}^v\bar\bfO_{\nu,-}={}^w\!\bar A_{\phi,+}^\nu\bfgmod$
and
${}^v\bar\bfO_{\phi,-}={}^w\!\bar A_{\phi,+}\bfgmod,$
which are graded analogues of the categories
${}^v\bfO_{\nu,-}$ and
${}^v\bfO_{\phi,-}.$
%${}^v\bfO_{\nu,-}=\bfmod({}^v\! R_{\nu,-})$ and
%${}^v\bfO_{\phi,-}=\bfmod({}^v\! R_{\phi,-}).$

To unburden the notation, we'll abbreviate $L_\nu={}^vL_{\nu,-}$ and $L_\phi={}^vL_{\phi,-}$.
Let $\bar L_{\nu}$,
$\bar L_{\phi}$ be the natural graded lifts  in
${}^v\bar\bfO_{\nu,-}$,
${}^v\bar\bfO_{\phi,-}$ of the semi-simple modules
$L_{\nu}$, $L_{\phi}$
respectively.

For $v\in I_{\nu,-}$ and $d+N>f,$ we'll use the following graded analogue of the translation functor 
$T_{\phi,\nu}:{}^v\bfO_{\phi,-}\to {}^v\bfO_{\nu,-}$
in Proposition \ref{prop:translation2}.

\vspace{2mm}

\begin{lemma}
\label{lem:gradedtranslationfunctor}
Assume that $v\in I_{\nu,-}$ and $d+N>f$.
Then, there is an exact functor of graded categories
$\bar T_{\phi,\nu}:{}^v\bar\bfO_{\phi,-}\to {}^v\bar\bfO_{\nu,-}$
such that $\bar T_{\phi,\nu}(\bar L_{\phi})=\bar L_{\nu}$
and $\bar T_{\phi,\nu}$
coincides with $T_{\phi,\nu}$ when forgetting the grading.
\end{lemma}

\vspace{0.5mm}

\begin{proof}
First, note that the translation functor 
$T_{\phi,\nu}:{}^v\bfO_{\phi,-}\to{}^v\bfO_{\nu,-}$ is well-defined,
because $d+N>f$ and $v\in I_{\nu,-}$.
See Section \ref{sec:translationO} for more details.

Now, by Corollary \ref{cor:2.31}, we have
${}^v\bfO_{\nu,-}={}^v\! \scrA_{\nu,-}\bfmod$ and
${}^v\bfO_{\phi,-}={}^v\! \scrA_{\phi,-}\bfmod.$
By Corollary \ref{cor:speBpm}, we can view
${}^v\! B_{\nu,-}$,
${}^v\! B_{\phi,-}$ as right modules over
${}^v\! \scrA_{\nu,-}$,
${}^v\! \scrA_{\phi,-}$ respectively.
Consider the functors
$\bbV'_k:{}^v\bfO_{\nu,-}\to{}^v\bfZ_{\nu,-}$
and
$\bbV'_k:{}^v\bfO_{\phi,-}\to{}^v\bfZ_{\phi,-}$
given by
$\bbV'_k={}^v\! B_{\nu,-}\otimes_{{}^v\! \scrA_{\nu,-}}\bullet$
and
$\bbV'_k={}^v\! B_{\phi,-}\otimes_{{}^v\! \scrA_{\phi,-}}\bullet.$

By Proposition \ref{prop:foncteurV},
the functors $\bbV_k$ on ${}^v\bfO_{\nu,-}$, ${}^v\bfO_{\phi,-}$ are exact
and we have
$\bbV_k({}^v\! \scrA_{\nu,-})={}^v\! B_{\nu,-}$,
$\bbV_k({}^v\! \scrA_{\phi,-})={}^v\! B_{\phi,-}.$
Further, since  ${}^v\!\scrA_{\nu,+}$, ${}^v\!\scrA_{\phi,+}$ are Noetherian, 
any finitely generated module has a finite presentation.
Thus, by the five lemma, the obvious morphism of functors $\bbV'_k\to\bbV_k$ is invertible.

Next, by Corollary \ref{cor:2.42}, we have 
$
{}^w\!\bar A^\nu_{\phi,+}=
{}^v\!\bar \scrA_{\nu,-}$
and
$
{}^w\!\bar A_{\phi,+}=
{}^v\!\bar \scrA_{\phi,-}
$
as graded $k$-algebras.
Thus, we have
${}^v\bar\bfO_{\nu,-}={}^v\! \bar \scrA_{\nu,-}\bfmod$ and
${}^v\bar\bfO_{\phi,-}={}^v\! \bar \scrA_{\phi,-}\bfmod.$
Consider the functors on ${}^v\bar\bfO_{\nu,-}$,
${}^v\bar\bfO_{\phi,-}$ given by
$\bar\bbV_k={}^v\!\bar B_{\nu,-}\otimes_{{}^v\!\bar \scrA_{\nu,-}}\bullet,$
$\bar\bbV_k={}^v\!\bar B_{\phi,-}\otimes_{{}^v\!\bar \scrA_{\phi,-}}\bullet$ respectively.
We have the commutative squares
$$
\xymatrix{
{}^v\bar\bfO_{\nu,-}\ar[r]_-{\bar\bbV_k}
\ar[d]&{}^v\bar\bfZ_{\nu,-}\ar[d]
\cr
{}^v\bfO_{\nu,-}\ar[r]^-{\bbV_k}&{}^v\bfZ_{\nu,-}
}
\qquad
\xymatrix{
{}^v\bar\bfO_{\phi,-}\ar[r]_-{\bar\bbV_k}
\ar[d]&{}^v\bar\bfZ_{\phi,-}\ar[d]
\cr
{}^v\bfO_{\phi,-}\ar[r]^-{\bbV_k}&{}^v\bfZ_{\phi,-}.
}
$$

Now, let
${}^v\bar\bfO_{\phi,-}^\text{proj}\subset
{}^v\bar\bfO_{\phi,-}$
be the full subcategory of the projective objects.
We define ${}^v\bar\bfO_{\nu,-}^\text{proj}$,
${}^v\bfO_{\nu,-}^\text{proj}$ and 
${}^v\bfO_{\phi,-}^\text{proj}$
in a similar way.

The functor $\bbV_k$ is fully faithful
on ${}^v\bfO_{\nu,-}^\text{proj}$ 
and
${}^v\bfO_{\phi,-}^\text{proj}$ by Corollary \ref{cor:speBpm}.
Hence, the functor $\bar\bbV_k$ is fully faithful on 
${}^v\bar\bfO_{\nu,-}^\text{proj}$ 
and
${}^v\bar\bfO_{\phi,-}^\text{proj}$.
Therefore, we may identify
${}^v\bar\bfO_{\nu,-}^\text{proj}$,
${}^v\bar\bfO_{\phi,-}^\text{proj}$
with some full subcategories
${}^v\bar\bfZ_{\nu,-}^\text{proj}$, 
${}^v\bar\bfZ_{\phi,-}^\text{proj}$
of
${}^v\bar\bfZ_{\nu,-}$, 
${}^v\bar\bfZ_{\phi,-}$
via $\bar\bbV_k$

The functor $\bar\theta_{\phi,\nu}$ in Remark \ref{rk:theta/k} gives an exact functor
${}^v\bar\bfZ_{\phi,-}^\text{proj}\to{}^v\bar\bfZ_{\nu,-}^\text{proj}$.
We define the functor
$\bar T_{\phi,\nu}:{}^v\bar\bfO_{\phi,-}^\text{proj}\to
{}^v\bar\bfO_{\nu,-}^\text{proj}$ so that it coincides with
$\bar\theta_{\phi,\nu}$ under $\bar\bbV_k$.
It gives a functor of triangulated categories
$\bfK^b({}^v\bar\bfO_{\phi,-}^\text{proj})\to\bfK^b({}^v\bar\bfO_{\nu,-}^\text{proj}),$
where $\bfK^b$ denotes the bounded homotopy category.

Since ${}^v\bar\bfO_{\nu,-}^\text{proj}$ 
and
${}^v\bar\bfO_{\phi,-}^\text{proj}$ have finite global dimensions,  there are canonical equivalences
$\bfK^b({}^v\bar\bfO_{\phi,-}^\text{proj})\simeq\bfD^b({}^v\bar\bfO_{\phi,-})$
and
$\bfK^b({}^v\bar\bfO_{\nu,-}^\text{proj})\simeq \bfD^b({}^v\bar\bfO_{\nu,-})$.
Thus, we can view $\bar T_{\phi,\nu}$ as a functor of triangulated categories
$\bfD^b({}^v\bar\bfO_{\phi,-})\to\bfD^b({}^v\bar\bfO_{\nu,-}).$

By Proposition \ref{prop:translation2}$(b),$ the functor
$T_{\phi,\nu}$ gives an exact functor
${}^v\bfO_{\phi,-}^\text{proj}\to{}^v\bfO_{\nu,-}^\text{proj}$.
By Remark \ref{rk:theta/k}, the functor $\bar T_{\phi,\nu}$ coincides with $T_{\phi,\nu}$
when forgetting the grading.
The functor $T_{\phi,\nu}$ yields a functor of triangulated categories
$\bfD^b({}^v\bfO_{\phi,-})\to\bfK^b({}^v\bfO_{\nu,-})$ which is $t$-exact for the standard $t$-structures
and which coincides with $T_{\phi,\nu}$ when forgetting the grading.
Hence, the functor $\bar T_{\phi,\nu}$ is also $t$-exact.
Thus, it gives an exact functor 
$\bar T_{\phi,\nu}:{}^{v}\bar\bfO_{\phi,-}\to {}^v\bar\bfO_{\nu,-}$
which coincides with $T_{\phi,\nu}$ when forgetting the grading.

Now, we concentrate on the equality 
$\bar T_{\phi,\nu}(\bar L_{\phi})=\bar L_{\nu}$.

Since $\bar T_{\phi,\nu}$
coincides with $T_{\phi,\nu}$ when forgetting the grading, by
Proposition \ref{prop:translation2}$(d)$,
for each $x\in{}^w\!I_{\nu,-}$ there is an integer $j$ such that
$\bar T_{\phi,\nu}(\bar L(xw_\nu\bullet\oo_{\phi,-}))=
\bar L(x\bullet\oo_{\nu,-})\langle j\rangle.$ 
We must check that $j=0$. 

%Recall that ${}^v\!\bar P(x\bullet\oo_{\nu,-})$ 
%is equal to ${}^v\! P(x\bullet\oo_{\nu,-})$,  with its natural grading.
Let $\bar T_{\nu,\phi}$ be the left adjoint to $\bar T_{\phi,\nu}$,
see Remark \ref{rk:gradedtranslationfunctor} below.
By Corollary \ref{cor:2.42} 
we have
${}^v\bar\bfO_{\nu,-}={}^v\! \bar \scrA_{\nu,-}\bfmod$ and
${}^v\bar\bfO_{\phi,-}={}^v\! \bar \scrA_{\phi,-}\bfmod.$
The graded $k$-algebras ${}^w\!\bar A^\nu_{\phi,+}$,
${}^w\!\bar A_{\phi,+}$ (=${}^v\! \bar \scrA_{\nu,-}$, 
${}^v\! \bar \scrA_{\phi,-}$) are Koszul.
The natural graded indecomposable projective modules are of the form
${}^v\! \bar \scrA_{\nu,-}1_x$, 
${}^v\! \bar \scrA_{\phi,-}1_x$ with $x\in {}^v\! I_{\nu,-}, 
{}^v\! I_{\phi,-}$ respectively.
We must check that
$\bar T_{\nu,\phi}({}^v\!\bar \scrA_{\nu,-}1_{x})=
{}^v\!\bar \scrA_{\phi,-}1_{xw_\nu}$
for each $x\in{}^w\!I_{\nu,-}$.

By definition of $\bar T_{\phi,\nu}$ we have an isomorphism 
$\bar\theta_{\phi,\nu}\circ\bar\bbV_k\simeq\bar\bbV_k\circ\bar T_{\phi,\nu}$ 
of functors on ${}^v\bar\bfO_{\phi,-}^\text{proj}$.
Further, by definition of $\bar\bbV_k,$ we have
$\bar\bbV_k({}^v\! \bar \scrA_{\nu,-}1_x)={}^v\!\bar B_{\nu,-}(x)$ and
$\bar\bbV_k({}^v\!\bar \scrA_{\phi,-}1_{xw_\nu})={}^v\!\bar B_{\phi,-}(xw_\nu).$
Therefore, it is enough to check that
$\bar\theta_{\nu,\phi}({}^v\!\bar B_{\nu,-}(x))=
{}^v\!\bar B_{\phi,-}(xw_\nu).$

By Remark \ref{rk:theta/k}, this follows from  Proposition \ref{prop:theta}$(f)$ by base change.
\end{proof}

\vspace{2mm}

\begin{rk}
\label{rk:gradedtranslationfunctor}
General facts imply that
the functor $\bar T_{\phi,\nu}$ has a left adjoint 
$\bar T_{\nu,\phi}:{}^v\bar\bfO_{\nu,-}\to{}^v\bar\bfO_{\phi,-}$,
see the proof of Proposition \ref{prop:translation2}.
By definition of
$T_{\nu,\phi}$ and the unicity of the left adjoint, the functor $\bar T_{\nu,\phi}$
coincides with $T_{\nu,\phi}$ when forgetting the grading.
\end{rk}

\vspace{2mm}

Now, we prove the following lemma which is dual to Lemma \ref{lem:gen1}.

\vspace{2mm}

\begin{lemma}
\label{lem:gen2}
Assume that $v\in I_{\nu,-}$ and $d+N>f.$
Then, the functor  $T_{\phi,\nu}:{}^v\bfO_{\phi,-}\to{}^v\bfO_{\nu,-}$ 
induces a surjective graded $k$-algebra homomorphism
$T_{\phi,\nu}:{}^v\!\bar A^\mu_{\phi,-}\to {}^v\!\bar A^\mu_{\nu,-}.$
Its kernel contains the two-sided ideal generated by 
$\{1_x\,;\,x\in{}^v\!I_{\phi,-}^\mu,\,x\notin I_{\nu,+}\}$.
Further, for $x\in {}^v\!I^\mu_{\phi,-}\cap I_{\nu,+}$
we have $xw_\nu\in {}^v\!I^\mu_{\nu,-}$
and $T_{\phi,\nu}(1_{x})=1_{xw_\nu}$.
\end{lemma}

\vspace{0.5mm}

\begin{proof}
First, note that, since $v\in I_{\nu,-}$ and $d+N>f$,
the functor  $T_{\phi,\nu}:{}^v\bfO_{\phi,-}\to{}^v\bfO_{\nu,-}$  is well-defined and it takes
${}^v\bfO_{\phi,-}^\mu$ into ${}^v\bfO_{\nu,-}^\mu$ by Proposition 
\ref{prop:translation2}$(c)$.

We'll abbreviate
$L_\phi={}^v\!L_{\phi,-}^\mu$ and $L_\nu={}^v\!L_{\nu,-}^\mu.$
By Proposition \ref{prop:translation2}$(f),$ we have 
$T_{\phi,\nu}(L_{\phi})=L_{\nu}$.
Thus, since $T_{\phi,\nu}$ is exact, it
induces a graded $k$-algebra homomorphism
$T_{\phi,\nu}:{}^v\!\bar A^\mu_{\phi,-}=\Ext_{{}^v\bfO^\mu_{\phi,-}}(L_{\phi})^\op\to
{}^v\!\bar A^\mu_{\nu,-}=\Ext_{{}^v\bfO^\mu_{\nu,-}}(L_{\nu})^\op.$

Composing $T_{\phi,\nu}$ with its left adjoint $T_{\nu,\phi},$ we get
the functor $\Theta=T_{\nu,\phi}\circ T_{\phi,\nu}$.
To prove that $T_{\phi,\nu}$ is surjective,
we must prove that the counit 
$\Theta\to\bf 1$ yields a surjective map 
$\Ext_{{}^v\bfO_{\phi,-}^\mu}(L_{\phi})\to
\Ext_{{}^v\bfO_{\phi,-}^\mu}(\Theta(L_{\phi}),L_{\phi}).$

The parabolic inclusion 
${}^v\bfO_{\phi,-}^\mu\subset{}^v\bfO_{\phi,-}$
is injective on extensions by Section \ref{sec:tau}. So we must prove
that the counit yields a surjective map 
$\Ext_{{}^v\bfO_{\phi,-}}(L_{\phi})\to
\Ext_{{}^v\bfO_{\phi,-}}(\Theta(L_{\phi}),L_{\phi}).$

Let us consider the graded analogue of this statement.
Set $\bar\Theta=\bar T_{\nu,\phi}\circ \bar T_{\phi,\nu}$,
where $\bar T_{\nu,\phi}$, $\bar T_{\phi,\nu}$ are as in Lemma \ref{lem:gradedtranslationfunctor}
and Remark \ref{rk:gradedtranslationfunctor}.
We have
$$\gathered
\Ext_{{}^v\bfO_{\phi,-}}(L_{\phi})=
\bigoplus_j\Ext_{{}^v\bar\bfO_{\phi,-}}(
\bar L_{\phi},\bar L_{\phi}\langle j\rangle),\cr
\Ext_{{}^v\bfO_{\phi,-}}(\Theta(L_{\phi}),L_{\phi})=
\bigoplus_j\Ext_{{}^v\bar\bfO_{\phi,-}}(
\bar\Theta(\bar L_{\phi}),\bar L_{\phi}\langle j\rangle).
\endgathered$$
Thus we must prove that for each $i,j$ the counit
$\eta:\bar\Theta\to\bf 1$ yields a surjective map
\begin{equation}
\label{surj}
\Ext^i_{{}^v\bar\bfO_{\phi,-}}(\bar L_{\phi},
\bar L_{\phi}\langle j\rangle)\to
\Ext^i_{{}^v\bar\bfO_{\phi,-}}
(\bar\Theta(\bar L_{\phi}),\bar L_{\phi}\langle j\rangle).
\end{equation}

By Lemma \ref{lem:gradedtranslationfunctor},
we have
$\bar T_{\phi,\nu}(\bar L_{\phi})=\bar L_{\nu}$.
Further, since ${}^v\bar\bfO_{\nu,-}={}^w\!\bar A_{\phi,+}^\nu\bfgmod$
and
${}^v\bar\bfO_{\phi,-}={}^w\!\bar A_{\phi,+}\bfgmod$, the gradings on 
${}^v\bar\bfO_{\phi,-}$ and ${}^v\bar\bfO_{\nu,-}$
are Koszul by Proposition \ref{prop:reg2}. 
Hence, since the right hand side of \eqref{surj} is 
$\Ext^i_{{}^v\bar\bfO_{\nu,-}}(\bar L_{\nu},\bar L_{\nu}\langle j\rangle),$ 
it is zero unless $i=j$.

Now, we define the integer
$\ell=\min\bigl\{d\,;\,\bar\Theta(\bar L_{\phi})^{d}\neq 0\bigr\}.$
Recall that the grading on ${}^w\!\bar A_{\phi,+}$ is positive.
Further, since 
${}^v\bar\bfO_{\phi,-}={}^w\!\bar A_{\phi,+}\bfgmod$, we can view
$\bar\Theta(\bar L_{\phi})$ as a graded ${}^w\!\bar A_{\phi,+}$-module.
Thus $\bar\Theta(\bar L_{\phi})^\ell$ 
is a quotient of $\bar\Theta(\bar L_{\phi})$ which is killed by the radical of 
${}^w\!\bar A_{\phi,+}$. We deduce that
$\bar\Theta(\bar L_{\phi})^\ell\subset\top(\bar\Theta(\bar L_{\phi})).$

Next, we claim that for any simple graded ${}^w\!\bar A_{\phi,+}$-module $\bar L$ such that
$\bar T_{\phi,\nu}(\bar L)\neq 0$, the map $\eta(\bar L):\bar\Theta(\bar L)\to\bar L$ 
yields an isomorphism  $\top(\bar\Theta(\bar L))\to\bar L$. 
Indeed, $\eta(\bar L)$ is surjective because it is non zero, and
for any simple quotient
$\bar\Theta(\bar L)\to\bar L'$ we have
$$0\neq\Hom_{{}^v\bar\bfO_{\phi,-}}(\bar\Theta(\bar L),\bar L')=
\Hom_{{}^v\bar\bfO_\nu}(\bar T_{\phi,\nu}(\bar L),\bar T_{\phi,\nu}(\bar L')).$$
By Proposition \ref{prop:translation2}$(d),(e)$ this implies that
$\bar T_{\phi,\nu}(\bar L)=\bar T_{\phi,\nu}(\bar L'),$ and that it is non zero.
Therefore, we have $\bar L=\bar L'$, proving the claim.

Recall that $\bar L_\phi$ is a semi-simple module, see Section \ref{sec:truncation}.
Applying the claim to the simple summands $\bar L\subset \bar L_\phi$
such that $\bar T_{\phi,\nu}(\bar L)\neq 0$, we get that
$\top(\bar\Theta(\bar L_\phi))=\Im\eta(\bar L_\phi)$.
In particular 
$\top(\bar\Theta(\bar L_{\phi}))$ is pure of degree zero.
This implies that we have $\ell=0$.
Therefore, the kernel of $\eta(\bar L_{\phi})$ 
lives in degrees $>0$.
Hence, by Koszulity of the grading of ${}^v\bar\bfO_{\phi,-}$, we get
$$\Ext^i_{{}^v\bar\bfO_{\phi,-}}\bigl(\Ker\eta(\bar L_\phi),\bar L_{\phi}\langle i\rangle\bigr)=0.$$
Hence the surjectivity of \eqref{surj} for $i=j$ follows from the long exact sequence of Ext's groups
associated with the exact sequence
$$0\to\Ker(\eta(\bar L_{\phi}))\to\bar\Theta(\bar L_\phi)\to\bar L_\phi.$$
This proves that the map $T_{\phi,\nu}$ is surjective, proving the first part of the lemma.

Next, we have
$I^\mu_{\nu,-}=\{xw_\nu\;;\; x\in \!I^\mu_{\phi,-}\cap I_{\nu,+}\}$
by Corollary \ref{lem:C2}$(a)$.
Further, we have
$T_{\phi,\nu}(1_{xw_\nu})=0$ for $x\notin I_{\nu,-}$
by Proposition \ref{prop:translation2}$(e)$, and 
$xw_\nu\in I_{\nu,+}$ if and only if $x\in I_{\nu,-}$. 
This proves the second claim of the lemma.
Finally, the last claim of the lemma follows from
Proposition \ref{prop:translation2}$(d)$.
\end{proof}

\vspace{2mm}

Next, we prove the following.

\vspace{2mm}

\begin{lemma}
\label{lem:gen3}
Assume that $w\in I_{\mu,+}^\nu$. Let $v=w_-^{-1}$. We have $v\in I^\mu_{\nu,-}$. 
Assume also that $d+N>f$ and $e+N>d$. There is a $k$-algebra isomorphism 
$p_{\mu,\nu}:{}^w\! A^\nu_{\mu,+}\to
{}^v\!\bar A^\mu_{\nu,-}$ such that $p_{\mu,\nu}(1_x)=1_{y}$ 
for each $x\in{}^w\!I_{\mu,+}^\nu$, where $y=x_-^{-1}$, and such that the following square
is commutative
\begin{equation}
\label{square4}
\begin{split}
\xymatrix{
{}^{w_\nu w}\!A_{\mu,+}\ar@{=}[r]_-{\eqref{prop:reg1}}\ar[d]^-{\tau_{\phi,\nu}}
%\ar[rd]_-{\pi_{\mu,\nu}}
&{}^v\!\bar A^\mu_{\phi,-}\ar[d]_{T_{\phi,\nu}}\cr
{}^w\! A^\nu_{\mu,+}\ar@{=}[r]^-{p_{\mu,\nu}}
&{}^v\!\bar A^\mu_{\nu,-}.}
\end{split}
\end{equation}
\end{lemma}

\vspace{.5mm}

\begin{proof}
We have $v=w_\mu w^{-1} w_\nu$. Note that
$w_\nu w\in I^\nu_{\phi,-}\cap I_{\mu,+}$, because by Lemma \ref{lem:C} we have $w^{-1}\in I^\mu_{\nu,+}$, hence $w^{-1}w_\nu\in I^\mu_{\phi,+}\cap I_{\nu,-}$ and $w_\nu w=(w^{-1}w_\nu)^{-1}\in I_{\mu,+}\cap I^\nu_{\phi,-}$.

Let $\pi_{\mu,\nu}:{}^{w_\nu w}\!A_{\mu,+}\to{}^v\!\bar A^\mu_{\nu,-}$ 
be the composition of the $k$-algebra homomorphism
$T_{\phi,\nu}:{}^v\!\bar A^\mu_{\phi,-}\to{}^v\!\bar A^\mu_{\nu,-}$ in Lemma \ref{lem:gen2}
and of the $k$-algebra isomorphism 
${}^{w_\nu w}\!A_{\mu,+}={}^v\!\bar A^\mu_{\phi,-}$ in \eqref{prop:reg1}.

Note that
${}^{w_\nu w}\! A^\nu_{\mu,+}={}^w\! A^\nu_{\mu,+}$.
We must construct a $k$-algebra isomorphism
$p_{\mu,\nu}:{}^w\! A^\nu_{\mu,+}\to{}^v\!\bar A^\mu_{\nu,-}$
such that $\pi_{\mu,\nu}=p_{\mu,\nu}\circ\tau_{\phi,\nu}$.

Let  $x\in{}^{w_\nu w}\!I_{\mu,+}$. 
Thus $w_\mu x^{-1}\in {}^v\!I^\mu_{\phi,-}$.
By Lemma \ref{lem:gen2}, we have 
$$\pi_{\mu,\nu}(1_x)\neq 0
\iff T_{\phi,\nu}(1_{w_\mu x^{-1}})\neq 0
\iff w_\mu x^{-1}\in I^\mu_{\phi,-}\cap I_{\nu,+}.$$
By Lemma \ref{lem:gen1}, we have 
$$\tau_{\phi,\nu}(1_x)\neq 0
\iff x\in I^\nu_{\mu,+}.$$
Again by Lemma \ref{lem:C}, we have
$$\aligned
w_\mu x^{-1}\in I^\mu_{\phi,-}\cap I_{\nu,+}
\iff xw_\mu\in I^\nu_{\phi,+}\cap I_{\mu,-}
\iff x\in I^\nu_{\mu,+}.
\endaligned$$
Hence, we have
$\tau_{\phi,\nu}(1_x)=0$ if and only if $\pi_{\mu,\nu}(1_x)=0$.
Thus, we have
$\Ker(\tau_{\phi,\nu})\subset\Ker(\pi_{\mu,\nu}),$ because the left hand side
is generated by the $1_x$'s killed by $\tau_{\phi,\nu}$ and the right hand side 
contains the $1_x$'s killed by $T_{\phi,\nu}$.

This proves the existence of a $k$-algebra homomorphism 
$p_{\mu,\nu}$ such that $\pi_{\mu,\nu}=p_{\mu,\nu}\circ\tau_{\phi,\nu}$
and  $p_{\mu,\nu}(1_x)=1_{w_\mu x^{-1}}$ 
for each $x\in{}^w\!I_{\mu,+}^\nu$.
The map $p_{\mu,\nu}$ is surjective, because the map 
$T_{\phi,\nu}$ is surjective by Lemma \ref{lem:gen2}.

Now, we prove that $p_{\mu,\nu}$ is injective.
By Section \ref{sec:tau}, the parabolic inclusion functor 
$i_{\mu,\phi}:{}^v\bfO^\mu_{\nu,-}\to{}^v\bfO_{\nu,-}$
yields a graded $k$-algebra homomorphism 
$i_{\mu,\phi}:{}^v\!\bar A_{\nu,-}^\mu\to{}^v\!\bar A_{\nu,-}$.
Set $z=v^{-1}=w_\nu ww_\mu$. 
The following is proved below.

\vspace{2mm}

\begin{claim}
\label{claim:gen4}
We have the commutative diagram
\begin{equation}
\label{diag1}
\begin{split}
\xymatrix{
{}^{w_\nu w}\!A_{\mu,+}\ar@{=}[r]_-{\eqref{prop:reg1}}\ar[d]^-{T_{\mu,\phi}}
&{}^v\!\bar A_{\phi,-}^\mu
\ar[r]_-{T_{\phi,\nu}}\ar[d]_{i_{\mu,\phi}}
&{}^v\!\bar A_{\nu,-}^\mu
\ar[d]_{i_{\mu,\phi}}\cr
{}^{z}\!A_{\phi,+}\ar@{=}[r]^-{\eqref{prop:reg1}}
&{}^v\!\bar A_{\phi,-}
\ar[r]^-{T_{\phi,\nu}}
&{}^v\!\bar A_{\nu,-}.
}
\end{split}
\end{equation}
\end{claim}

Note that $z\in I_{\mu,-}$, 
hence  the map $T_{\mu,\phi}$ is well-defined.
%Note also that  $v\in I^\mu_{\nu,-}\subset I_{\nu,-}$. 

%Since $d+m>f$ and $e+m>d$, we define
%$T_{\phi,\mu}$, $T_{\phi,\nu}$, $T_{\mu,\phi}$
%as in Proposition \ref{prop:translation2}.

Now, by Proposition \ref{prop:translation2}$(g)$, the translation functor $T_{\mu,\phi}$ yields a map
${}^w\!A_{\mu,+}^\nu\to{}^{ww_\mu}\!A_{\phi,+}^\nu$.
Consider the diagram
\begin{equation}
\label{diag2}
\begin{split}
\xymatrix{
{}^{w_\nu w}\!A_{\mu,+}\ar[r]_-{\tau_{\phi,\nu}}\ar[d]^-{T_{\mu,\phi}}
&{}^w\!A_{\mu,+}^\nu\ar[r]_-{p_{\mu,\nu}}
\ar[d]^-{T_{\mu,\phi}}
&{}^v\!\bar A_{\nu,-}^\mu
\ar[d]_{i_{\mu,\phi}}\cr
{}^z\!A_{\phi,+}\ar[r]^{\tau_{\phi,\nu}}
&{}^{ww_\mu}\!A^\nu_{\phi,+}\ar[r]^-{p_{\phi,\nu}}
&{}^v\!\bar A_{\nu,-}.
}
\end{split}
\end{equation}
%Since $\pi_{\mu,\nu}=p_{\mu,\nu}\circ\tau_{\phi,\nu}$
%and $\pi_{\mu,\phi}=p_{\mu,\phi}\circ\tau_{\phi,\phi},$ 

By Claim \ref{claim:gen4}, the outer rectangle in \eqref{diag1} is commutative.
Thus, the outer rectangle in \eqref{diag2} is commutative.
The left square in \eqref{diag2} is commutative, by
Proposition \ref{prop:translation2}$(c)$.
Thus, since $\tau_{\phi,\nu}$ is surjective, the right square in 
\eqref{diag2} is also commutative.

%The left vertical map in \eqref{diag2} is injective by Remark \ref{rk:fidele}.
%Since $\tau_{\phi,\nu}$ is surjective, we deduce that
%the middle vertical map in \eqref{diag2} is injective.

The middle vertical map in \eqref{diag2} is injective by Remark \ref{rk:fidele}.
Therefore,
to prove that $p_{\mu,\nu}$ is injective it is enough to check that $p_{\phi,\nu}$ is injective.

Now, it is easy to see that the map $p_{\phi,\nu}$ is indeed invertible, beause
 it is surjective by the discussion above (applied to the choice $\nu=\phi$) and
$\dim({}^{ww_\mu}\!A_{\phi,+}^\nu)=\dim({}^v\!\bar A_{\nu,-})$
by Proposition \ref{prop:reg2}.

Finally, to finish the proof of Lemma \ref{lem:gen3} we must check that
$p_{\mu,\nu}(1_x)=1_{y}$ 
for each $x\in{}^w\!I_{\mu,+}^\nu$, where $y=x_-^{-1}$.

To do that, it suffices to observe that $x_-=w_\nu x w_\mu$, and that,
by Proposition \ref{prop:reg1} and  Lemmas \ref{lem:gen1}, \ref{lem:gen2},
the square of maps
in \eqref{square4} gives the following diagram
$$\xymatrix{
1_x\ar@{|->}[r]^-{\eqref{prop:reg1}}\ar@{|->}[d]^-{\tau_{\phi,\nu}}&1_{w_\mu x^{-1}}\ar@{|->}[d]_-{T_{\phi,\nu}}\\
1_x\ar@{|->}[r]^-{p_{\mu,\nu}}&1_{w_\mu x^{-1}w_\nu}.
}$$

Now, we prove Claim \ref{claim:gen4}.
The right square in \eqref{diag1} is commutative, by 
Proposition \ref{prop:translation2}$(c)$.

Let us concentrate on the left square.
We must prove that the isomorphisms 
${}^{w_\nu w}\!A_{\mu,+}={}^v\!\bar A_{\phi,-}^\mu$
and
${}^w\!A_{\phi,+}=
{}^v\!\bar A_{\phi,-}$
in Proposition \ref{prop:reg1} 
yield a commutative square
\begin{equation}\label{4.27}
\begin{split}
\xymatrix{
{}^{w_\nu w}\!A_{\mu,+}\ar@{=}[r]_-{\eqref{prop:reg1}}\ar[d]^-{T_{\mu,\phi}}
&{}^v\!\bar A_{\phi,-}^\mu\ar[d]_-i\cr
{}^z\!A_{\phi,+}\ar@{=}[r]^-{\eqref{prop:reg1}}
&{}^v\!\bar A_{\phi,-}.
}
\end{split}
\end{equation}
By Proposition \ref{prop:translation2}$(g)$, the module 
$T_{\mu,\phi}({}^{w_\nu w}\!P_{\mu,+})$ is a direct summand of
${}^z\!P_{\phi,+}$. 
By Proposition \ref{prop:theta}$(f)$ and Remark \ref{rk:theta/k}, 
the sheaf
$\theta_{\mu,\phi}({}^w\!B_{\mu,+})$ is a direct summand of ${}^w\!B_{\phi,+}.$
Thus, we have the following diagram
\begin{equation*}
%\label{diaga}
\begin{split}
\xymatrix{
{}^{w_\nu w}\!A_{\mu,+}\ar@{=}[r]^-{\bbV_k}\ar[d]_-{T_{\mu,\phi}}
&\End_{{}^{w_\nu w}\!Z_{\mu,+}}({}^{w_\nu w}\!B_{\mu,+})^\op
\ar[d]^-{\theta_{\mu,\phi}}
\cr
{}^z\!A_{\phi,+}\ar@{=}[r]^-{\bbV_k}
&\End_{{}^z\!Z_{\phi,+}}({}^z\!B_{\phi,+})^\op.
}
\end{split}
\end{equation*}
Note that the horizontal maps are invertible by Corollary \ref{cor:speBpm}.
This diagram is commutative by Remark \ref{rk:theta/k}, see also
Proposition \ref{prop:theta}$(a)$.
Next, by Proposition \ref{prop:loc3}$(c)$ and Corollary \ref{cor:loc4}, we have a commutative diagram
\begin{equation*}
%\label{diagb}
\begin{split}
\xymatrix{
\End_{{}^{w_\nu w}\!Z_{\mu,+}}({}^{w_\nu w}\!B_{\mu,+})^\op
\ar[d]_-{\theta_{\mu,\phi}}
&{}^v\!\bar A^\mu_{\phi,-}\ar@{=}[l]_-\bbH
\ar[d]^-i\cr
\End_{{}^z\!Z_{\phi,+}}({}^z\!B_{\phi,+})^\op
&{}^v\!\bar A_{\phi,-}.\ar@{=}[l]_-\bbH
}
\end{split}
\end{equation*}
Finally, the horizontal maps in \eqref{4.27} are equal to the composition of 
$\bbH$ and $\bbV_k$.
\end{proof}

\vspace{2mm}

Finally, we prove our main theorem.

\vspace{2mm}

\begin{proof}[Proof of Theorem \ref{thm:main}]
Since the highest weight categories
${}^w\bfO^\nu_{\mu,+},$
${}^v\bfO^\mu_{\nu,-}$ do not depend on $e,$ $f$ by Remark \ref{rk:level},
we can assume that there is a positive integer $d$
such that $d+N>f$ and $e+N>d$. Thus the hypothesis of
Lemma \ref{lem:gen3} is satisfied.

By Propositions \ref{prop:reg1}, \ref{prop:reg2}
the $k$-algebra
${}^w\!A_{\mu,\pm}$ has a Koszul grading.
Thus, by Lemma \ref{lem:1.2}  and Section \ref{sec:tau}, the $k$-algebra
${}^w\!A^\nu_{\mu,\pm}$ has also a Koszul grading.

Let us equip ${}^w\!A^\nu_{\mu,\pm}$ with this grading.
By Lemma \ref{lem:1.1},  we have
${}^w\!A^{\nu,!}_{\mu,\pm}={}^w\!\bar A^\nu_{\mu,\pm}$
as graded $k$-algebras.
Therefore,  the graded $k$-algebra ${}^w\!\bar A^\nu_{\mu,\pm}$ is Koszul and its Koszul dual is isomorphic to
${}^w\!A^\nu_{\mu,\pm}$ as a $k$-algebra.

We have
${}^w\!\bar A^{\nu,!}_{\mu,+}={}^w\!A^\nu_{\mu,+}$ as a $k$-algebra.
By Lemma \ref{lem:gen3}, we have also a $k$-algebra isomorphism
${}^w\!A^\nu_{\mu,+}={}^v\!\bar A^\mu_{\nu,-}$.
Thus, we have
${}^w\!\bar A^{\nu,!}_{\mu,+}={}^v\!\bar A^\mu_{\nu,-}$
as $k$-algebras. By unicity of the Koszul grading, we deduce that
${}^w\!\bar A^{\nu,!}_{\mu,+}={}^v\!\bar A^\mu_{\nu,-}$
as graded $k$-algebras.

The involutivity of the Koszul duality implies that we have also
${}^w\!\bar A^{\nu,!}_{\mu,-}=
{}^v\!\bar A^\mu_{\nu,+}$ as graded $k$-algebras, and 
${}^w\!A^\nu_{\mu,-}={}^v\!\bar A^\mu_{\nu,+}$ as $k$-algebras. 

Finally, we must check that under the isomorphism
${}^w\!\bar A^{\nu,!}_{\mu,+}={}^v\!\bar A^\mu_{\nu,-}$ we have
$1_x^!=1_{y}$ with $y=x^{-1}_-$ for each $x\in{}^w\!I_{\mu,+}^\nu.$

If $\nu=\phi$ this is Proposition \ref{prop:reg1}.
If $\mu=\phi$ this is Proposition \ref{prop:reg2}.

The isomorphism ${}^w\!\bar A^{\nu,!}_{\mu,+}={}^w\!A^\nu_{\mu,+}$ above,
which is given by Lemma \ref{lem:1.1}, identifies the idempotents $1_x^!$ and $1_x$ for each
$x\in{}^w\!I_{\mu,+}^\nu.$
Thus, by Lemma \ref{lem:gen3}, the isomorphism
${}^w\!\bar A^{\nu,!}_{\mu,+}={}^v\!\bar A^\mu_{\nu,-}$
identifies the idempotents $1_x^!$ and $1_y$, where $y=x_-^{-1}$.
\end{proof}

\vspace{1cm}

\section{Type A and applications to CRDAHA's}\label{app:A}

Fix  integers $e,\ell,N>0$. Let $\frakg=\frakg\frakl(N)$.

\subsection{Koszul duality in the type A case}
Let $\frakb,\frakt\subset\frakg$
be the Borel subalgebra  of upper
triangular matrices and
the maximal
torus  of  diagonal matrices.
Let $(\epsilon_i)$ be the canonical
basis of $\bbC^N$. 
%We identify $\frakt^*=\bbC^N/\bbC\,1^N$,
%$\frakt=\bbC^N(0)$ and $W=\frakS_N$ in the obvious way.
We identify $\frakt^*=\bbC^N$,
$\frakt=\bbC^N$ and $W=\frakS_N$ in the obvious way.
Put $\rho=(0,-1,\dots,1-N)$ and
$\alpha_i=\epsilon_i-\epsilon_{i+1}$ for each $i\in[1,N).$ 

We define the affine Lie algebra $\bfg$ of $\frakg$ as in Section
\ref{sec:affine}.
For any subset $X\subset\bbC^\ell$ and any $d\in\bbC$ let
$X(d)=\bigl\{(x_1,\dots,x_\ell)\in X\,;\,\sum_ix_i=d\bigr\}.$
Set $\scrC^\ell_d=\bbN^\ell(d)$ and fix an element
$\nu\in\scrC^\ell_N$.
Let $\nu$ denote also the parabolic type $\{\alpha_i;\  i\neq \nu_1, \nu_1+\nu_2,...\}$.
Let $\bfp_\nu\subset\bfg$ be the corresponding parabolic subalgebra. 
The Levi subalgebra of $\bfp_\nu$ is the Lie subalgebra $\frakg_\nu=\frakg\frakl(\nu_1)\oplus\dots\oplus\frakg\frakl(\nu_\ell)$ of $\frakg$ 
consisting of block diagonal elements. %trace free elements.

Let $P=\bbZ^N$ be the set of  integral weights of $\frakg$ and let
$P^\nu\subset P$ be the subset of $\nu$-dominant integral weights. 
Fix an element $\mu\in\scrC^e_N$.
Consider the $N$-tuple
$1_\mu=(1^{\mu_1}2^{\mu_2}\cdots e^{\mu_e})$. 
Since $1_\mu\in P+\rho,$ the affine weight
$\oo_{\mu,-}=(1_\mu-\rho)_e$
is an antidominant integral classical affine weight of $\bfg$ of level $-e-N$, see Section
\ref{sec:2.9}.

Recall that  $W_\mu\subset W$ is the parabolic subgroup
generated by the simple affine reflections
$s_i$ with $\alpha_i\notin\mu$. Let $\widetilde W=\frakS_N\ltimes\bbZ^N$ be the extended affine symmetric group.
We define the \emph{$e$-action of  $\widetilde W$ on $\bbZ^N$} to be such that
$\frakS_N$ acts by permutation of the entries of a $N$-tuple, while $\tau\in\bbZ^N$ acts by 
translation by the $N$-tuple $-e\,\tau$.
Let $w\cdot_ex$ be the result of the $e$-action of $w$ on the element $x\in\bbZ^N$. 
The stabilizer of the $N$-tuple $1_\mu$ is equal to $W_\mu$. It is a standard 
parabolic subgroup.

Consider the category $\bfO^\nu_{\mu,-}$
introduced in Section \ref{sec:2.9}. 
It is canonically equivalent to a block of a truncated category O of the affine Lie algebra associated with
$\frak{sl}(N)$, because $\frak{sl}(N)$ differs from  $\frakg$ by a central element.
Hence, all the results  above can be applied to the category $\bfO^\nu_{\mu,-}$.
We'll write $\bfO^\nu_{\mu,-e}=\bfO^\nu_{\mu,-}$ and
$\bfO^\nu_{-e}=\bigoplus_{\mu\in\scrC^e_N}\bfO^\nu_{\mu,-e}.$
%For each $\lambda\in P^\nu,$ we abbreviate
%$V_e^\nu(\lambda)=V^\nu(\lambda_{e})$ and
%$L_e^\nu(\lambda)=L(\lambda_{e})$
%in $\bfO^\nu_{-e}$.
The classes $[V^\nu(\lambda_e)]$ with $\lambda\in P^\nu$ form a $\bbC$-basis of the complexified Grothendieck group $[\bfO^\nu_{-e}]$.
Composing the Koszul equivalence in Theorem \ref{thm:main} and the tilting equivalence, we get an equivalence of 
triangulated 
categories $K=E((\bullet)^\diamond):\bfD^b(\bfO^\nu_{\mu,-e})\to\bfD^b(\bfO^\mu_{\nu,-\ell})$, which
induces a $\bbC$-linear isomorphism 
$K:[\bfO^\nu_{\mu,-e}]\to[\bfO^\mu_{\nu,-\ell}]$.

\vspace{3mm}

\subsection{The level-rank duality}

In thist section, we'll relate the map $K$ above to the level-rank duality.
Consider the Lie algebra $\fraks\frakl(e)$. We identify the set of weights of $\fraks\frakl(e)$  with $\bbC^e/\bbC 1^e$.
Thus, elements of $\bbC^e$ can be viewed as weights of $\fraks\frakl(e),$ or
equivalently, as 
level 0 classical affine weights of the affine Kac-Moody algebra $\widehat{\fraks\frakl}(e).$ 
Let
$\{\varepsilon_{i}\,;\,i\in[1,e]\}$ be the canonical basis of $\bbC^e$. For each $\lambda\in P$ and each $k\in[1,N]$, we decompose the
$k$-th entry of $\lambda+\rho$ as the sum
$\lambda_k+\rho_k=i_k+e\,r_k$ with $i_k\in[1,e]$ and $r_k\in\bbZ$.
Then, we  can view the sum
$\widehat\wt_e(\lambda)=\sum_{k=1}^N\varepsilon_{i_k}+
(\sum_{k=1}^Nr_k)\,\delta$ as a level 0 affine weight of $\widehat{\fraks\frakl}(e)$, and the sum
$\wt_e(\lambda)=\sum_{k=1}^N\varepsilon_{i_k}$ as a weight of ${\fraks\frakl}(e).$

The vector spaces $V_e=\bbC^e\otimes\bbC[z^{-1},z]$ carries an obvious level zero action of $\widehat{\fraks\frakl}(e)$. It induces an action of $\widehat{\fraks\frakl}(e)$ on the space $\bigwedge^\nu(V_e)=\bigotimes_{p=1}^\ell\bigwedge^{\nu_p}(V_e)$ of tensor product of wedge powers.
Write $v_{i+er}=\varepsilon_i\otimes z^r\in V_e$ for each $i\in[1,e]$, $r\in\bbZ$.
Note that a tuple $a\in\bbZ^N$ of the form $a=(a_{1,1},a_{1,2},\dots,a_{\ell,\nu_\ell})$
belongs to $P^\nu+\rho$ if and only if we have
$a_{p,1}>a_{p,2}>\cdots>a_{p,\nu_p}$ for each $p\in[1,\ell]$.
Set
$\wedge^\nu(a)=\bigotimes_{p=1}^\ell(v_{a_{p,1}}\wedge v_{a_{p,2}}
\wedge\dots\wedge v_{a_{p,\nu_p}})$.
Then the set $\{\wedge^\nu(a)\,;\,a\in P^\nu+\rho\}$  is a basis of $\bigwedge^\nu (V_e).$ 
Further, for each $\lambda\in P^\nu$ the element $\wedge^\nu(\lambda+\rho)$ 
has the weight $\widehat\wt_e(\lambda)$.
%e.g., \cite{RSVV} for more details.
Let $\bigwedge^\nu(V_e)_\mu$ be the weight subspace of $\bigwedge^\nu(V_e)$
of weight 
$\bar\mu=\sum_{i=1}^e\mu_i\,\varepsilon_i$ with respect to the action of
$\fraks\frakl(e)$.
The set
$\{\wedge^\nu(a)\,;\,a\in (P^\nu+\rho)\cap(\widetilde W\cdot_e1_\mu)\}$ 
is a basis of $\bigwedge^\nu (V_e)_\mu.$ 
In a similar way, we equip the vector space $V_\ell=\bbC^\ell\otimes\bbC[z^{-1},z]$
with the basis $\{v_{p+\ell r}\,;\,p\in[1,\ell],\,r\in\bbZ\}$. 
Then, we define the $\widehat{\fraks\frakl}(\ell)$-module
$\bigwedge^\mu(V_\ell)$ as above.

\smallskip

Next, consider the $\widehat{\fraks\frakl}(e)\times\widehat{\fraks\frakl}(\ell)$-module
$\bigwedge^N(V_{e,\ell})$ with $V_{e,\ell}=\bbC^e\otimes\bbC^\ell\otimes\bbC[z^{-1},z].$
Let $\bar\nu=\sum_{p=1}^\ell\nu_p\,\varepsilon_p,$ viewed as a weight of ${\fraks\frakl}(\ell)$.
The weight subspace of $\bigwedge^N(V_{e,\ell})$ of weight
$\bar\mu$ for the $\fraks\frakl(e)$-action is isomorphic to
the $\widehat{\fraks\frakl}(\ell)$-module
$\bigwedge^\mu(V_\ell)$, and the weight subspace of weight
$\bar\nu$ for the $\fraks\frakl(\ell)$-action is isomorphic to
the $\widehat{\fraks\frakl}(e)$-module
$\bigwedge^\nu(V_e)$.

Since ${\textstyle\bigwedge^\nu}(V_e)_\mu$ and ${\textstyle\bigwedge^\mu}(V_\ell)_\nu$ are both canonically identified
with the weight $(\bar\mu,\bar\nu)$ subspace $\bigwedge^N(V_{e,\ell})_{\mu,\nu}$ of $\bigwedge^N(V_{e,\ell})$
for the ${\fraks\frakl}(e)\times{\fraks\frakl}(\ell)$-action,
we have a canonical linear isomorphism
$LR:{\textstyle\bigwedge^\nu}(V_e)_\mu\to
{\textstyle\bigwedge^\mu}(V_\ell)_\nu$.
More precisely, we have an isomorphism
$f_{\mu,\nu}:{\textstyle\bigwedge^\nu}(V_e)_\mu\to\bigwedge^N(V_{e,\ell})_{\mu,\nu}$
which takes the element $\wedge^\nu(a)$,
with $a=(a_{p,k})$ in $P^\nu+\rho,$
to the monomial
$\bigwedge_{(p,k)}(v_{i_{p,k}}\otimes v_p\otimes z^{r_{p,k}}).$
Here $i_{p,k}\in[1,e]$, $r_{p,k}\in\bbZ$ are such that $a_{p,k}=i_{p,k}+e\,r_{p,k}$, and
the pair $(p,k)$ runs from $(1,1)$ to $(\ell,\nu_\ell)$ in lexicographic order.
Next, there is an obvious isomorphism of ${\fraks\frakl}(e)\times{\fraks\frakl}(\ell)$-modules
$\tau:\bigwedge^N(V_{e,\ell})\to\bigwedge^N(V_{\ell,e})$
 which exchanges the vector spaces
$\bbC^e$ and $\bbC^\ell$.
Then, we define $LR=(f_{\nu,\mu})^{-1}\circ \tau\circ f_{\mu,\nu}$.
We'll call $LR$ the \emph{level-rank duality}.

Our next result compares the isomorphisms $LR$ and $K$. To do so, we
consider the $\bbC$-linear isomorphism  $\theta:[\bfO^\nu_{-e}]\to\bigwedge^\nu(V_e)$  such that
$\theta([V^\nu(\lambda_e)])=\wedge^\nu(\lambda+\rho)$ for each $\lambda\in P^\nu$. 
It takes $[\bfO^\nu_{\mu,-e}]$ onto the weight subspace
$\bigwedge^\nu(V_e)_\mu$.
We can now prove the following.

\vspace{2mm}

\begin{prop}\label{prop:B1}
We have a commutative square
$$\xymatrix{
[\bfO^\nu_{\mu,-e}]\ar[r]^-{K}\ar[d]_\theta&[\bfO^\mu_{\nu,-\ell}]\ar[d]^\theta\\
{\textstyle\bigwedge^\nu}(V_e)_\mu\ar[r]^-{LR}&
{\textstyle\bigwedge^\mu}(V_\ell)_\nu.
}$$
\end{prop}

\vspace{.5mm}

\begin{proof}
Let $\triv_\mu$, $\sgn_\mu$ be the idempotents of the $\bbC$-algebra of the parabolic subgroup 
$W_\mu\subset\widetilde W$ 
associated with the trivial representation and the signature.

For each tuple $a\in\bbZ^N$ we write $v(a)=\bigotimes_{k=1}^N v_{a_k}$.
Let $(V_e)_\mu^{\otimes N}$ be the subspace of $(V_e)^{\otimes N}$ of weight 
$\bar\mu$ relatively to the $\fraks\frakl(e)$-action.
The element
$v(a)$ of $(V_e)^{\otimes N}$ belongs to 
$(V_e)_\mu^{\otimes N}$ if and only if $a\in\widetilde W\cdot_e1_\mu$.
Therefore, the assignment $x\cdot\triv_\mu\mapsto v(x\cdot_e 1_\mu)$ with $x\in\widetilde W$
yields a $\bbC$-linear isomorphism $\bbC\widetilde W\cdot\triv_\mu\to(V_e)_\mu^{\otimes N}.$

Since $x\,\bullet\,\oo_{\mu,-}=(x\cdot_e1_\mu-\rho)_e$,
we have $x\in I^\nu_{\mu,-}$ if and only if the $N$-tuple $x\cdot_e1_\mu$ lies in $P^\nu+\rho$.
We deduce that the isomorphism above factors to a $\bbC$-linear isomorphism
$B_{\nu,\mu}:\sgn_\nu\!\cdot\,\bbC\widetilde W\cdot\triv_\mu\to\bigwedge^\nu (V_e)_\mu$ such that
$\sgn_\nu\!\cdot\, x\cdot\triv_\mu\mapsto 
\wedge^\nu(x\cdot_e1_\mu).$
Under this isomorphism, the basis $\{\wedge^\nu(a)\,;\,a\in (P^\nu+\rho)\cap(\widetilde W\cdot_e1_\mu)\}$  of 
$\bigwedge^\nu (V_e)_\mu$ is identified with
$\{\sgn_\nu\!\cdot\, x\cdot\triv_\mu\,;\,x\in I^\nu_{\mu,-}\}$.

By Corollary \ref{lem:C2}, we have a bijection $ I^\nu_{\mu,-}\to I^\mu_{\nu,-}$
such that $x\mapsto x^{-1}$.
We claim that, under the isomorphisms $B_{\nu,\mu}$ and $B_{\mu,\nu}$,
the map $LR$ is identified with the $\bbC$-linear map
$\sgn_\nu\!\cdot\,\bbC\widetilde W\cdot\triv_\mu\to\sgn_\mu\!\cdot\,\bbC\widetilde W\cdot\triv_\nu$
such that
$\sgn_\nu\!\cdot\, x\cdot\triv_\mu\mapsto\sgn_\mu\!\cdot\, x^{-1}\cdot\triv_\nu$.
Since, by Remark \ref{rk:2.20}, the isomorphism $K$ takes the element
$[V^\nu(x\bullet\oo_{\mu,-})]$ to $[V^\mu(x^{-1}\bullet\oo_{\nu,-})]$,
this finishes the proof of the proposition.

The claim is a direct consequence of the definition of the map $LR$.
\end{proof}

\vspace{3mm}

\subsection{The CRDAHA} 
\label{sec:CRDAHA}
Let $\Gamma\subset\bbC^\times$ be the group of the $\ell$-th roots of 1.
Fix $\nu\in\bbZ^\ell(N)$ and fix an integer $d>0$. 

Let $\Gamma_d=\frakS_d\ltimes\Gamma^d$. It is a complex reflection group.
Let $\scrP_d$ be the set of {\it partitions of $d$}. Write $|\lambda|=d$ 
and let $l(\lambda)$ be the {\it length} of $\lambda$.
Let $\scrP^\ell_d$ be the set of {\it $\ell$-partitions of $d$}, i.e., the
set of $\ell$-tuples $\lambda=(\lambda_p)$ of partitions with
$\sum_p|\lambda_p|=d$. There is a bijection between the set of irreducible representations of $\Gamma_d$ and $\scrP^\ell_d$, see e.g. \cite[sec.~6]{Ro}.

We set 
$h=1/e$ and
$h_p=\nu_{p+1}/e-p/\ell$ for each $p\in\bbZ/\ell\bbZ.$ 
Let $H^\nu(d)$ 
be the CRDAHA of $\Gamma_d$
with parameters $h$ and $(h_p)$.
It is the quotient
of the smash product of $\bbC \Gamma_{d}$ and the tensor algebra of
$(\bbC^2)^{\oplus d}$ by the relations
$$[y_i,x_i]=1-k\sum_{j\neq i}\sum_{\gamma\in\Gamma}s_{ij}^{\gamma}
-\sum_{\gamma\in\Gamma\setminus\{1\}}c_\gamma\gamma_i,$$
$$[y_i,x_j]=k\sum_{\gamma\in\Gamma}\gamma s_{ij}^{\gamma}
\quad \text{if}\ i\neq j,$$
$$[x_i,x_j]=[y_i,y_j]=0.$$
The
parameters are such that
$k=-h$ and $-c_\gamma=\sum_{p=0}^{\ell-1}\gamma^{-p}(h_{p}-h_{p-1})$ for $\gamma\neq1.$ 

\smallskip

Let $\calO^\nu_{-e}\{d\}$ 
be the  category O of $H^\nu(d)$.
It is a highest weight category with set of standard modules
$\Delta(\calO^\nu_{-e}\{d\})=\{\Delta(\lambda)\,;\,\lambda\in \scrP^\ell_d\},$
see \cite[sec.~3.3, 3.6]{SV}, \cite[sec.~3.6]{VV}.
To avoid confusions we may write
$\Delta_e^\nu(\lambda)=\Delta(\lambda).$
Let $S_e^\nu(\lambda)$ be the top of
$\Delta_e^\nu(\lambda)$.

\smallskip

We have the block decomposition
$\calO^{\nu}_{-e}\{d\}=\bigoplus_{\mu}\calO^{\nu}_{\mu}\{d\},$
where $\mu=(\mu_1,\dots,\mu_e)$ is identified with $\bar\mu=\sum_{i=1}^{e}\mu_i\,\varepsilon_i$
and it runs over the set of all integral weights of $\fraks\frakl(e)$.
By \cite{LM}, these blocks are determined by the following combinatorial rule.

\smallskip

For each $\lambda\in \scrP^\ell_d,$ an {\it $(i,\nu)$-node} in $\lambda$ 
is a triple $(x,y,p)$ with 
$x,y>0$ and $y\leqslant(\lambda_p)_x$ 
such that
$y-x+\nu_p=i$ modulo $e$. 
Let $n_{i}^\nu(\lambda)$ be the number of $(i,\nu)$-nodes 
in $\lambda$.
Then, we have $\Delta^\nu_e(\lambda)\in\calO^{\nu}_{\mu}\{d\}$ if and only if 
$\lambda\in \Lambda^\nu_\mu\{d\}$, where
\begin{equation}
\label{form:1.2}
\Lambda^\nu_\mu\{d\}=\Big\{\lambda\in \scrP^\ell_d\,;\,\sum_{p=1}^{\ell}\omega_{\nu_p}-
\sum_{i=1}^{e-1}\bigl(n_i^\nu(\lambda)-n_0^\nu(\lambda)\bigr)\,\alpha_i=\bar\mu\Big\}.
\end{equation}
The expression above should be regarded as
an equality of integral weights of $\fraks\frakl(e)$. More precisely,
the symbols $\omega_1,\omega_2,\dots,\omega_{e-1}$ and
$\alpha_1,\alpha_2,\dots,\alpha_{e-1}$ are respectively the fundamental weights and
the simple roots of $\fraks\frakl(e)$, and the subscript $\nu_p$ in \eqref{form:1.2} should be 
viewed as the residue class of $\nu_p$ in $[1,e)$.
See \cite[lem.~5.16]{SV} for details.

By \cite[rem.~4.5]{RSVV}, the condition \eqref{form:1.2} is equivalent to
$$\Lambda^\nu_\mu\{d\}=\Big\{\lambda\in \scrP^\ell_d\,;\,\wt_e(\lambda+\rho_\nu-\rho)=
\bar\mu\Big\}.$$
Note that an integral weight of $\fraks\frakl(e)$ can be represented
by an element of $\bbZ^e(k)$ if and only if it lies in $\omega_k+e\bbZ\Pi$.
Thus, since $\nu\in\bbZ^\ell(N)$, if $\Lambda^\nu_\mu\{d\}\neq\emptyset$ then
$\mu$ can be represented by an $e$-tuple in $\bbZ^e(N)$.
Indeed, it is not difficult to check that if $\Lambda^\nu_\mu\{d\}\neq\emptyset$ then
$\mu\in\scrC^e_N.$

\smallskip

For each $\mu\in\scrC^e_N$
and $a\in\bbN,$ we write
$\calO_{\mu}^\nu[a]=\calO^\nu_\mu\{d\}$ and
$\Lambda^\nu_{\mu}[a]=\Lambda^\nu_\mu\{d\}$, where
$d=a\,e+\bigl\langle \sum_{p=1}^\ell\omega_{\nu_p}-\mu:\rho\bigr\rangle.$
For each $\lambda\in\scrP^\ell$, we have
$\lambda\in\Lambda^{\nu}_{\mu}[a]$
if and only if $n_0^\nu(\lambda)=a$ and $\wt_e(\lambda+\rho_\nu-\rho)=\bar\mu$.

\smallskip
We are interested by the following conjecture \cite[conj.~6]{CM}.

\vspace{2mm}

\begin{conj} \label{conj:A.0}
The blocks $\calO^{\nu}_{\mu}[a]$ and 
$\calO^{\mu}_{\nu}[a]$ have a (standard) Koszul grading.
The Koszul dual of  $\calO^{\nu}_{\mu}[a]$ is equivalent to the Ringel dual of
$\calO^{\mu}_{\nu}[a]$. 
\end{conj}

\vspace{2mm}

\subsection{The Schur category }
Fix a composition $\nu\in\scrC^\ell_N$. 
Let $\scrP^\nu_d=\{\lambda\in \scrP^\ell_d\,;\,l(\lambda_p)\leqslant\nu_p\}$.
There is an  inclusion
$\scrP^\nu_d\subset P^\nu\subset \bbZ^N$ such that $$\lambda\mapsto
\bigl(\lambda_10^{\nu_1-l(\lambda_1)}, \lambda_20^{\nu_2-l(\lambda_2)},\dots,
\lambda_\ell 0^{\nu_\ell-l(\lambda_\ell)}\bigr),$$
where the partition $\lambda_p$ is viewed as an element in $\bbZ^{\l(\lambda_p)}$.

Consider the $\nu$-dominant weight  $\rho_\nu=(\nu_1,\nu_1-1,\dots,1,\nu_2,\nu_2-1,\dots,\nu_\ell,\dots 1)\in P^\nu.$
For each $\lambda\in \scrP^\nu,$ we abbreviate
%$\varpi(\lambda)=\lambda+\rho_\nu-\rho\in P^\nu.$
%We abbreviate
$V_e^\nu(\lambda)=V^\nu((\lambda+\rho_\nu-\rho)_e)$ and
$L_e^\nu(\lambda)=L^\nu((\lambda+\rho_\nu-\rho)_e)$.

Following \cite{VV}, let $\bfA^{\nu}_{-e}\{d\}\subset\bfO^{\nu}_{-e}$ be the Serre subcategory 
consisting of the
finite length $\bfg$-modules of level $-e-N$ 
whose constituents belong to the set
$\{L_e^\nu(\lambda)\,;\,\lambda\in \scrP^\nu_d\}.$
By \cite{VV}, it is a highest weight category with
the set of standard modules
$\Delta(\bfA^{\nu}_{-e}\{d\})=\{V_e^\nu(\lambda)\,;\,\lambda\in \scrP^\nu_d\}.$

We have the following conjecture \cite[conj.~8.8]{VV}.

\vspace{2mm}

\begin{conj}
\label{conj:A.1}
There is a quotient functor 
$\calO^\nu_{-e}\{d\}\to\bfA^{\nu}_{-e}\{d\}$ taking
$S_e^\nu(\lambda)$ to $L_e^\nu(\lambda)$ if $\lambda\in \scrP^\nu_d$
and to 0 else. If $\nu_p\geqslant d$ for each $p$, then we have $\scrP^\nu_d=\scrP_d$ and the functor above is an equivalence
of highest weight categories.
\end{conj}

A proof of Conjecture \ref{conj:A.1} is given in \cite{RSVV}.

\vspace{2mm}

\subsection{Koszul duality of the Schur category} 
Fix an integer $d\geqslant 0$ and compositions $\mu\in\scrC^e_N$, $\nu\in\scrC^\ell_N$.
Let $\bfA^{\nu}_{\mu}\{d\},$ $\bfA^\nu_{\mu}[a]$
be the Serre subcategories of $\bfA^{\nu}_{-e}=\bigoplus_{d\in\bbN}\bfA^{\nu}_{-e}\{d\}$
generated by the modules $L^\nu_e(\lambda)$ such that
$\lambda\in\Lambda^{\nu}_{\mu}\{d\}$, $\Lambda^{\nu}_{\mu}[a]$ respectively.
Write $\bfA^{\nu}_{\mu,-e}=\bigoplus_{d\in\bbN}\bfA^{\nu}_{\mu,-e}\{d\}.$

In view of Conjecture \ref{conj:A.1}, the following claim can be regarded
as an analogue of Conjecture \ref{conj:A.0}.

\vspace{2mm}

\begin{thm}\label{thm:1.3}
The category $\bfA^{\nu}_{\mu}[a]$ has a (standard) Koszul grading.
The Koszul dual of  $\bfA^{\nu}_{\mu}[a]$ is equivalent to the Ringel dual of
$\bfA^{\mu}_{\nu}[a]$. 
\end{thm}

\vspace{.5mm}

\begin{proof}
Recall the $\bbC$-linear map
$K:[\bfO^\nu_{\mu,-e}]\to[\bfO^\mu_{\nu,-\ell}]$.
First, we claim that
$K([\bfA^\nu_{\mu,-e}])=[\bfA^\mu_{\nu,-\ell}]$.
By Proposition \ref{prop:B1} and the definition of the Schur category,
we must compute the element
$LR(\wedge^\nu(a)),$ for each $a=\lambda+\rho_\nu$ such that $\lambda\in\scrP^\nu$ and
$\wedge^\nu(a)\in{\textstyle\bigwedge^\nu}(V_e)_\mu$.

First, we consider the $\bbC$-vector space
$LR\big({\textstyle\bigwedge^\nu}(V_e)_\mu\big)$.
The map $f_{\mu,\nu}$ takes the monomial $\wedge^\nu(a),$ 
with $a=(a_{p,k})\in P^\nu+\rho,$ to the monomial
$\bigwedge_{(p,k)}(v_{i_{p,k}}\otimes v_p\otimes z^{r_{p,k}})$ such that
$a_{p,k}=i_{p,k}+e\,r_{p,k}$. The weight subspace
${\textstyle\bigwedge^\nu}(V_e)_\mu$ is spanned by the monomials
$\wedge^\nu(a)$ as above such that 
$\mu_i=\sharp\{(p,k)\,;\,i_{p,k}=i\}$ for each $i\in[1,e]$.
Therefore, we have
$f_{\mu,\nu}\big({\textstyle\bigwedge^\nu}(V_e)_\mu\big)=\bigwedge^N(V_{e,\ell})_{\mu,\nu},$ 
which is the subspace of
$\bigwedge^N(V_{e,\ell})$ spanned by all
the monomials
$\wedge^N(x)=\bigwedge_{(i',p',r')}(v_{i'}\otimes v_{p'}\otimes z^{r'})$ 
where $(i',p',r')$ runs over the entries of an $N$-tuple
$x\in([1,e]\times[1,\ell]\times\bbZ)^N$
such that
$\mu_i=\sharp\{(i',p',r')\in x\,;\,i'=i\}$ and
$\nu_p=\sharp\{(i',p',r')\in x\,;\,p'=p\}$ for each $i\in[1,e],$ $p\in[1,\ell]$.
%This implies in particular that
%$LR\big({\textstyle\bigwedge^\nu}(V_e)_\mu\big)={\textstyle\bigwedge^\mu}(V_\ell)_\nu.$

Next, we consider the subspace
$LR\circ \theta\big([\bfA^\nu_{\mu,-e}]\big)$.
The subspace $\theta\big([\bfA^\nu_{\mu,-e}]\big)$ of ${\textstyle\bigwedge^\nu}(V_e)_\mu$ 
is spanned by the monomials $\wedge^\nu(a)$
such that there is an $\ell$-partition $\lambda\in\Lambda^\nu_\mu$
with $a_{p,k}=\lambda_{p,k}+\nu_p+1-k$ for each $p,k$. Here 
$\lambda_{p,k}$ is the 
$k$-th part of the $p$-th partition $\lambda_p$ of $\lambda$.
Thus, the map $f_{\mu,\nu}$ takes 
$\theta\big([\bfA^\nu_{\mu,-e}]\big)$ to
the subspace of
$\bigwedge^N(V_{e,\ell})_{\mu,\nu}$ spanned by all
the monomials
$\wedge^N(x)$ as above such that 
$x\in([1,e]\times[1,\ell]\times\bbN)^N$.
This implies in particular that 
$LR\circ \theta\big([\bfA^\nu_{\mu,-e}]\big)=\theta\big([\bfA^\mu_{\nu,-\ell}]\big)$, proving the claim.

Now, to finish the proof of the theorem, it is enough to observe that
$\theta([\bfA^\nu_{-e}[a]])$ is the sum, over all affine weights
$\hat\mu\in P+a\delta,$ of the intersection of
of $\theta([\bfA^\nu_{-e}])$ with the subspace of 
${\textstyle\bigwedge^\nu}(V_e)$ of weight $\hat\mu$
for the $\widehat{\fraks\frakl}(e)$-action. 
\end{proof}

\vspace{2mm}

The following is obvious.

\vspace{2mm}

\begin{cor}
Conjecture \ref{conj:A.1} implies Conjecture \ref{conj:A.0}. 
\qed
\end{cor}

\vspace{2mm}

\subsection{Koszulity of the $q$-Schur algebra}\label{sec:schurkoszul}
Now, we consider the case $\ell=1$.

By the Kazhdan-Lusztig equivalence \cite[thm.~38.1]{KL}, 
the category $\bfA^{(N)}_{-e}\{d\}$ is equivalent to the module category of the $q$-Schur algebra.
This equivalence is an equivalence of highest-weight categories. This follows from the fact that
an equivalence of abelian categories $\bfC_1\to\bfC_2$ such that $\bfC_1$, $\bfC_2$
are of highest weight
and that the induced bijection $\Irr(\bfC_1)\to\Irr(\bfC_2)$ is increasing is an equivalence of highest weight 
categories, i.e., it induces a bijection $\Delta(\bfC_1)\to\Delta(\bfC_2)$.
See also \cite[thm.~A.5.1]{VV} for a more detailed proof.

Therefore, Theorem \ref{thm:1.3} implies the following.

\vspace{2mm}

\begin{cor} The $q$-Schur algebra is Morita
equivalent to a Koszul algebra which is balanced.
\qed
\end{cor}

\vspace{2mm}

See also \cite{CM} for a different approach to this result.
Note that our proof gives
an explicit description of the Koszul dual of the $q$-Schur algebra
in term of affine Lie algebras.

\appendix

\vspace{1cm}

\section{Finite codimensional affine Schubert varieties}\label{app:B}

By a  scheme we always mean a scheme  over the field $k$, and by a variety we'll mean a reduced scheme of finite type
which is quasi-projective.
Let $T$ be a torus.
A {\it $T$-scheme}  is a scheme  with an algebraic $T$-action.

Fix a contractible topological $T$-space $ET$ with a topologically free $T$-action.
For any $T$-scheme $X,$ we set $X_T=X\times_TET$. There are obvious projections
$p:X\times ET\to X$ and $q:X\times ET\to X_T.$

We'll  identify $S$ with the $T$-equivariant cohomology $k$-space $H_T(\bullet)$ of a point.

\subsection{Equivariant perverse sheaves on finite dimensional varieties}
\label{sec:HT}

Fix a $T$-variety $X.$
Let $\bfD_T^b(X)$ be the $T$-equivariant bounded derived category.
It is the full subcategory of the bounded derived category $\bfD^b(X_T)$
of sheaves of $k$-vector spaces on $X_T$ (for the analytic topology on $X$), 
consisting of the complexes of sheaves $\calF$ with an isomorphism
$q^*\calF\simeq p^*\calF_X$ for some $\calF_X\in\bfD^b(X)$. 

The cohomology of an object $\calF\in\bfD^b_{T}(X)$ is the graded $S$-module
$
H(\calF)=\bigoplus_{i\in\bbZ}
\Hom_{\bfD^b_T(X)}\bigl(k_X,\calF[i]\bigr).
$
For each $\calE,\calF\in\bfD^b_T(X)$ we have an object
$\calR\calH om(\calE,\calF)$ in $\bfD^b(X_T)$ such that
$\Ext_{\bfD^b_T(X)}(\calE,\calF)=H(\calR\calH om(\calE,\calF)).$

If $Y\subset X$ is a 
$T$-equivariant embedding, let $\bar Y$ be its closure in $X$.
Let $IC_T(\overline{Y})$ be the minimal extension of 
$k_{Y}[\dim Y]$ (the shifted equivariant constant sheaf). 
It is a perverse sheaf on $X$ supported on 
$\bar Y$.

Let $IH_T(\bar Y)$ be the $k$-space of equivariant intersection cohomology of $\bar Y$
and let $H_T(\bar Y)$ be its equivariant cohomology.
We have
$IH_T(\bar{Y})=H(IC_T(\bar Y))$ and
$H_T(\bar Y)=H(k_{\bar Y})$.
Forgetting the $T$-action we define $IC(\bar Y)$, $IH(\bar Y)$ and $H(\bar Y)$ in the same way.

We'll say that $X$ is {\it good}
if the following hold
\begin{itemize}
\item $X$ has a  Whitney stratification
$X=\bigsqcup_xX_x$ by $T$-stable subvarieties,
\item $X_x=\bbA^{l(x)}$ with a linear $T$-action, for some integer $l(x)\in\bbN$,
%\item $X$ is covered by $T$-stable open affine  subset
%with an attractive fixed point,
%\item $X$ has a finite number of one-dimensional orbits and the closure of each of them is smooth,
\item there are integers
$n_{x,y,i}\in\bbN$ such that
\begin{equation}
\label{parity1}
j_y^*IC_T(\bar X_x)=\bigoplus_pk_{X_y}[l(y)][l(x)-l(y)-2p]^{\oplus n_{x,y,p}}
\end{equation}
where $j_x$ is the inclusion $X_x\subset X$.
Equivalently, we have
\begin{equation}
\label{parity2}
j_y^!IC_T(\bar X_x)=
\bigoplus_{q}k_{X_y}[l(y)][l(y)-l(x)+2q]^{\bigoplus n_{x,y,q}}.
\end{equation}
\end{itemize}
We call the third property the {\it parity vanishing}.

\vspace{2mm}

\begin{prop}
\label{prop:IC/ICT/IH}
If $X$ is a good $T$-variety, then the following hold

(a) 
$\dim\Ext_{\bfD^b(X)}^i
\bigl(IC(\bar X_x),IC(\bar X_y)\bigr)=
\sum_{z,p,q}n_{x,z,p}n_{y,z,q}$ where $z,p,q$ runs over the set of triples such that
$2l(z)-l(y)-l(x)+2p+2q=i,$

(b) $\Ext_{\bfD^b(X)}\bigl(IC(\bar X_x),
IC(\bar X_y)\bigr)=
k\Ext_{\bfD^b_T(X)}\bigl(IC_T(\bar X_x),
IC_T(\bar X_y)\bigr),$

(c) $\Ext_{\bfD^b_T(X)}
\bigl(IC_T(\bar X_x),IC_T(\bar X_y)\bigr)
=\Hom_{H_T(X)}\bigl(IH_T(\bar X_x),IH_T(\bar X_y)\bigr),$

(d) $\Ext_{\bfD^b(X)}
\bigl(IC(\bar X_x),IC(\bar X_y)\bigr)
=\Hom_{H(X)}\bigl(IH(\bar X_x),IH(\bar X_y)\bigr),$

(e) $IH(\bar X_x)$ vanishes in degrees $\not\equiv l(x)$ modulo 2 and
$IH(\bar X_x)=kIH_T(\bar X_x).$
\end{prop}

\vspace{.5mm}

\begin{proof}
Part $(a)$  is  \cite[thm.~3.4.1]{BGS}.
We sketch briefly the proof  for the comfort of the reader.
Set $X_p=\bigsqcup_{p=l(z)}X_z$.
Let $j_p$ be the inclusion of $X_p$ into $X$.
For each $\calF\in\bfD^b(X),$
there is a spectral sequence
$E_1^{p,q}=H^{p+q}(j_p^!\calF)\Rightarrow H^{p+q}(\calF).$
Therefore, if
$\calF=\calR\calH om\bigl(IC(\bar X_x),IC(\bar X_y)\bigr),$
we get a spectral sequence
$$\gathered
E_1^{p,q}
=\bigoplus_{z}
H^{p+q}
\calR\calH om\bigl(j_z^*IC(\bar X_x),j_z^!IC(\bar X_y)\bigr)\Rightarrow 
\Ext_{\bfD^b(X)}^{p+q}
\bigl(IC(\bar X_x),IC(\bar X_y)\bigr),
\endgathered$$
where $z$ runs over the set of elements with $l(z)=p$.
By \eqref{parity1}, \eqref{parity2}
the spectral sequence degenerates at $E_1$,
and we get
$$\dim\Ext_{\bfD^b(X)}^i
\bigl(IC(\bar X_x),IC(\bar X_y)\bigr)=
\sum_{z,p,q}n_{x,z,p}n_{y,z,q},$$
where $z,p,q$ are as in $(a)$.

Now, we prove $(b)$.
For each $\calF\in\bfD^b_T(X),$
there is a spectral sequence 
$E_2^{p,q}=S^p\otimes H^q(\calF_X)\Rightarrow H^{p+q}(\calF)$, see \cite[sec.~5.5]{GKM}.
Thus, if
$\calF=\calR\calH om\bigl(IC_T(\bar X_x),IC_T(\bar X_y)\bigr),$
we get a spectral sequence
$$E_2^{p,q}=S^p\otimes \Ext^q_{\bfD^b(X)}(IC(\bar X_x),IC(\bar X_y))
\Rightarrow\Ext^{p+q}_{\bfD^b_T(X)}(IC_T(\bar X_x),IC_T(\bar X_y)).$$
Now, we have $\Ext^q_{\bfD^b(X)}(IC(\bar X_x),IC(\bar X_y))=0$ for
$q\not\equiv l(x)+l(y)$ modulo 2 by $(a)$. 
Therefore, since $S$ vanishes in odd degrees, 
the spectral sequence degenerates at $E_2$.
Thus, we have
$$S\otimes \Ext_{\bfD^b(X)}(IC(\bar X_x),IC(\bar X_y))
=\Ext_{\bfD^b_T(X)}(IC_T(\bar X_x),IC_T(\bar X_y)).$$

The proof  of $(d)$ is as in \cite[sec.~1]{G}, \cite[sec.~ 3.3]{BY}.
Since our setting is slightly different we sketch briefly the main arguments.
Fix a partial order on the set of strata such that
$\bar X_x=X_{\leqslant x}=\bigsqcup_{y\leqslant x}X_y.$
Consider the obvious inclusions
$j_{\leqslant x}:X_{\leqslant x}\to X$,
$i=j_{<x}:X_{<x}\to X_{\leqslant x}$ and $j=j_{x}:X_{x}\to X_{\leqslant x}.$
For each $y\leqslant x,$ we set
$\calF_1=j_{\leqslant x}^*IC(\bar X_x)$ and $\calF_2=j_{\leqslant x}^*IC(\bar X_y)$.
For $s=1,2$ and $p\in\bbN$, there are integers $d_s$ and $d_{s,y,p}$ such that
\begin{equation}
\label{parity}
j_y^*\calF_s=\bigoplus_pk_{X_y}[d_s-2p]^{\oplus d_{s,y,p}},\quad
j_y^!\calF_s=\bigoplus_{q}k_{X_y}[2l(y)-d_s+2q]^{\bigoplus d_{s,y,q}}.
\end{equation}
Consider the diagram of graded $k$-vector spaces  given by
$$\xymatrix{
\Ext_{\bfD^b(X_{<x})}(i^*\calF_1,i^!\calF_2)\ar[r]\ar[d]&
\Hom_{H(X_{<x})}(H(i^*\calF_1),H(i^!\calF_2))\ar[d]^-a
\\
\Ext_{\bfD^b(X_{\leqslant x})}(\calF_1,\calF_2)\ar[r]\ar[d]&
\Hom_{H(X_{\leqslant x})}(H(\calF_1),H(\calF_2))\ar[d]^-b
\\
\Ext_{\bfD^b(X_x)}(j^*\calF_1,j^*\calF_2)\ar[r]&\Hom_{H(X_{ x})}(H(j^*\calF_1),H(j^*\calF_2)).
}$$
In the right side sequence, the Hom's are the $k$-spaces of graded module homomorphisms over  the graded
$k$-algebras $H(X_{<x})$, $H(X_{\leqslant x})$, $H(X_{x})$ respectively, which are computed in the category
of non-graded modules. The horizontal maps are given by taking hyper-cohomology.

We'll prove that the middle map is invertible by induction on $x$. This proves part $(d)$.
The short exact sequence associated with the left side is exact by \eqref{parity}, see e.g., \cite[lem.~ 5.3]{FW} and the references there.
The lower map is obviously invertible, because $j^*\calF_s$ are constant sheaves on $X_x$ for each $s=1,2$.
The complexes $i^*\calF_s$ and $i^!\calF_s$ on $X_{<x}$ satisfy again \eqref{parity} for 
any stratum $X_y\subset X_{<x}$.
Thus, by induction, we may assume that the upper 
map is invertible. Then, to prove that the middle map is invertible it is enough to check that
$a$ is injective and that $\Im(a)=\Ker(b)$. 

By \eqref{parity}, the distinguished triangles
$i_*i^!\to1\to j_*j^*{\buildrel +1\over\to}$  and $j_!j^*\to 1\to i_*i^*{\buildrel +1\over\to}$
yield exact sequences of $H(X_{\leqslant x})$-modules
\begin{equation*}\gathered
0\to H(i^!\calF_s){\buildrel\alpha_s\over\to} H(\calF_s){\buildrel\beta_s\over\to} H(j^*\calF_s)\to 0,\\
0\to H_c(j^*\calF_s){\buildrel\gamma_s\over\to} H(\calF_s) {\buildrel\delta_s\over\to}H(i^*\calF_s)\to 0,
\endgathered\end{equation*}
where $H_c(j^*\calF_s)=H(j_!j^*\calF_s)$ is the cohomology with compact supports of $j^*\calF_s$.
%See e.g., \cite[lem.~ 5.3]{FW}.
Thus, the map $a$ is injective, because $a(\phi)=\alpha_2\circ\phi\circ\delta_1$ for each $\phi$, $\alpha_2$ is injective and $\delta_1$ is surjective. 

Next, fix an element $\phi\in\Ker(b)$.
Since $b(\phi)\circ\beta_1=\beta_2\circ\phi$ by construction and $\beta_1$ is surjective, we have 
$\beta_2\circ\phi=0$. Hence, there is an $H(X_{\leqslant x})$-linear map $\phi':H(\calF_1)\to H(i^!\calF_2)$ such that $\phi=\alpha_2\circ\phi'$.
We must check that $\phi'\circ\gamma_1=0$.
%Taking $\calF_1=k_{X_{\leqslant x}}$ in the exact sequence above, we get the exact sequence 
%$$0\to H_c(X_x)\to H(X_{\leqslant x})\to H(X_{<x})\to 0.$$
%Note that $H_c(X_{\leqslant x})=H(X_{\leqslant x})$ and $H_c(X_{<x})=H(X_{< x})$.

To do that, set $d=\dim X_x$. Since $X_x\simeq k^d$, the fundamental class $\omega$ of $X_x$
belongs to $H_c^{2d}(X_x)$.
The cup product with $\omega$ restricts to 0 on $X_{<x}$, because there is an exact sequence
$0\to H_c(X_x)\to H(X_{\leqslant x})\to H(X_{<x})\to 0.$
Thus, it yields a map
$g_s:H(\calF_s)\to H_c(j^*\calF_s)[2d]$ which vanishes on $\Im(\alpha_s)$.
We deduce that $g_2\circ\phi=0.$ 
Since $H_c(X_x)\subset H_c(X_{\leqslant x})$
and $\phi$ is $H_c(X_{\leqslant x})$-linear,
we have also $\phi\circ g_1=0$.
Since $g_1$ is surjective, we deduce that
$\phi\circ\gamma_1=0.$ 
Hence, we have $\phi'\circ\gamma_1=0$, because $\alpha_2$ is injective.

Part $(c)$ is proved as part $(d)$, using equivariant cohomology. 

Finally, we prove $(e)$. By parity vanishing, the spectral sequence
$$E_1^{p,q}=H^{p+q}(j_p^!IC(\bar X_x))\Rightarrow IH^{p+q}(\bar X_x)$$
degenerates at $E_1$. This yields the first claim, as in part $(a)$.
The second one follows from the first one, because the spectral sequence 
$E_2^{p,q}=S^p\otimes IH^{q}(\bar X_x)\Rightarrow IH^{p+q}(\bar X_x)$
degenerates at $E_2$.
\end{proof}

\vspace{2mm}

\subsection{Equivariant perverse sheaves on infinite dimensional varieties}
This section is a reminder on equivariant perverse sheaves on infinite dimensional varieties,
to be used in the next section.

Let $X$ be an {\it essentially smooth} $T$-scheme, in the sense of
\cite[sec.~1.6]{KT3}. In \cite[sec.~2]{KT3} 
the derived category of constructible complexes on
$X$ and perverse sheaves on $X$ (for the analytic topology) are defined.
Note that the convention for perverse sheaves here 
differs from the convention for perverse sheaves on varieties (as in Section \ref{sec:HT}) :
indeed $k_Y[-\codim Y]$ is perverse if $Y$ is 
an essentially smooth $T$-scheme, while $k_Y[\dim Y]$
is perverse for a smooth variety $Y$.
We'll use the same terminology as in loc.~cit. and 
we formulate our results in the $T$-equivariant setting. 
The equivariant version of the constructions in \cite{KT3} 
is left to the reader. 

Let $\bfD_{T}(X)$ be the $T$-equivariant derived category on $X$.
If $X,Y$ are  essentially smooth and $Y\to X$ is a 
$T$-equivariant embedding of finite presentation,
let $IC_T(\overline{Y})$ be the minimal extension of  the equivariant constant shifted sheaf
$k_{Y}[-\codim Y]$ on $Y$. 
It is a perverse sheaf on $X$ supported on the Zariski closure
$\bar Y$ in $X$.
If $\calF\in\bfD_{T}(X),$ we set
$H(\calF)=\bigoplus_{i\in\bbZ}
\Hom_{\bfD_T(X)}\bigl(k_X,\calF[i]\bigr),$
a graded $S$-module.
We abbreviate 
$
IH_T(\bar{Y})=H(IC_T(\bar Y))$
and
$H_T(Y)=H(k_Y).$
We call $H_T(Y)$ the {\it equivariant cohomology of $Y$}.
It is an $H_T(X)$-module.

Recall  that for a morphism $f:Z\to X$ of essentially smooth $T$-schemes 
there is a  functor
$f^*:\bfD_T(X)\to\bfD_T(Z)$, see \cite[sec.~3.7]{KT3}.
If $f$ is the inclusion of a $T$-stable subscheme we write $\calF_Z=i^*\calF$ and
$IH_T(\overline{Y})_Z=H(IC_T(\overline{Y})_Z). $

Now, assume that $X$ is the limit of a projective system of smooth schemes
\begin{equation}\label{projsystem}
X=\pro_nX_n\end{equation} 
as in \cite[sec.~1.3]{KT3}.
Let $p_n:X\to X_n$ be the projection and set $Y_n=p_n(Y)$.
Assume that $Y_n$ is locally closed in $X_n$.
Since $k_Y=p_n^*k_{Y_n}$
we have an obvious map
$$\aligned
H_T(Y_n)
&=\bigoplus_{i\in\bbZ}
\Hom_{\bfD_T(X_n)}\bigl(k_{X_n},k_{Y_n}[i]\bigr)\\
&\to\bigoplus_{i\in\bbZ}
\Hom_{\bfD_T(X)}\bigl(p_n^*k_{X_n},p_n^*k_{Y_n}[i]\bigr)\\
&\to\bigoplus_{i\in\bbZ}
\Hom_{\bfD_T(X)}\bigl(k_{X},k_{Y}[i]\bigr)\\
&\to H_T(Y).
\endaligned$$
It yields an isomorphism
$H_T(Y)=\ind_nH_T(Y_n).$

\vspace{2mm}

\subsection{Moment graphs and the Kashiwara flag manifold}
In this section we consider the localization on Kashiwara's flag manifold.
The main motivation is Corollary \ref{cor:B2} below that we used in the rest of text, in particular in
Corollary \ref{cor:speBpm} to prove that the functor
$\bbV_{\!k}$ is fully-faithful on projectives in ${}^w\bfO_{\mu,-}^\Delta$.

Let $P_\mu$ be the parabolic subgroup corresponding to the Lie algebra
$\bfp_\mu$.
Let $X=X_\mu=\calG/P_\mu$ be the Kashiwara partial
flag manifold associated with $\bfg$ and $\bfp_\mu$, see \cite{K}.
Here $\calG$ is the schematic analogue of $G(k((t)))$ defined in \cite{K}, which has a 
locally free right action of the group-$k$-scheme $P_\mu$ and a 
locally free left action of the group-$k$-scheme $B^-$, 
the Borel subgroup
opposit to $B$.
Recall that $X$ is an essentially smooth, not quasi-compact,  $T$-scheme,
which is covered by $T$-stable quasi-compact open subsets 
isomorphic to $\bbA^\infty=\Spec k[x_k\,;\,k\in\bbN]$.

Let $e_X=P_\mu/P_\mu$ be the origin of $X$. For each $x\in I_{\mu,+},$ we set 
$X^x=B^-xe_X=B^-xP_\mu/P_\mu.$
Note that $X^x$ is a locally closed $T$-stable subscheme of $X$ 
of codimension $l(x)$ which is isomorphic to $\bbA^\infty$.
Consider the $T$-stable subschemes
$$\overline{X^x}=X^{\geqslant x}=\bigsqcup_{y\geqslant x}X^y,\quad
X^{\leqslant x}=\bigsqcup_{y\leqslant x}X^y,\quad
X^{<x}=\bigsqcup_{y<x}X^y.$$
We call $\overline{X^x}$
a {\it finite-codimensional affine Schubert variety.}
We call $X^{\leqslant x}$ an
{\it admissible open set.}
%It is quasi-compact.
%If $f:Z\to X$ is the inclusion of a finite dimensional $T$-scheme which is transverse to
%$X^y$ for each $y\geqslant x$, there are 
%isomorphisms 
%$$\gathered
%IC_T(\overline{X^x})_Z=IC_T(\overline{X^x}\cap Z),\quad
%IH_T(\overline{X^x})_Z=IH_T(\overline{X^x}\cap Z).
%\endgathered$$

If $\Omega$ is an admissible open set, there are canonical 
isomorphisms 
$$
IC_T(\overline{X^x})_\Omega=IC_T(\overline{X^x}\cap \Omega),\quad
IH_T(\overline{X^x})_\Omega=IH_T(\overline{X^x}\cap \Omega).
$$
We can view $\Omega$ as the limit of a projective system of smooth schemes 
$(\Omega_n)$ as in  \cite[lem.~4.4.3]{KT3}. So,
the projection $p_n:\Omega\to \Omega_n$ is a good quotient by a congruence subgroup
$B^-_{n}$ of $B^-$. 
Let $n$ be large enough.
Then, we have $IC_T(\overline{X^x}\cap \Omega)=p_n^*IC_T(\overline{X^x_{\Omega,n}})$
with $\overline{X^x_{\Omega,n}}=p_n(\overline{X^x}\cap \Omega)$
by \cite[sec.~2.6]{KT3}.
Thus we have a map
$$\aligned
IH_T(\overline{X^x_{\Omega,n}})
&=\bigoplus_{i\in\bbZ}
\Hom_{\bfD_T(\Omega_n)}\bigl(k_{\Omega_n},IC_T(\overline{X^x_{\Omega,n}})[i]\bigr)\\
&\to\bigoplus_{i\in\bbZ}
\Hom_{\bfD_T(\Omega)}\bigl(k_{\Omega},p_n^*IC_T(\overline{X^x_{\Omega,n}})[i]\bigr)\\
&\to IH_T(\overline{X^x}\cap \Omega)
\endaligned$$
which yields an isomorphism
$$IH_T(\overline{X^x}\cap \Omega)=\ind_nIH_T(\overline{X^x_{\Omega,n}}).$$

For $\Omega=X^{\leqslant w}$ and $x\leqslant w$ we abbreviate
$X^{[x,w]}=\overline{X^x}\cap X^{\leqslant w}$ and
$X^{[x,w]}_n=p_n(X^{[x,w]}).$
Since $p_n$ is a good quotient by  $B^-_n$
and since $X^{x}$ is $B^-_n$-stable,
we have an algebraic stratification 
$X^{\leqslant w}_n=\bigsqcup_{x\leqslant w}X^{x}_n,$
where $X^{x}_n$ is an affine space whose Zarisky closure 
is $X^{[x,w]}_n$.

\vspace{2mm}

\begin{lemma}
\label{lem:good}
(a) The $T$-variety $X^{\leqslant w}_n$ is smooth and good. 

(b) It is covered by $T$-stable open affine  subsets
with an attractive fixed point.
The fixed points subset is naturally identified with ${}^w\!I_{\mu,+}$.

(c) There is a finite number of one-dimensional orbits. The closure of each of them is smooth.
Two fixed points are joined by a one-dimensional orbit if and only if the corresponding 
points in ${}^w\!I_{\mu,+}$ are joined by an edge in ${}^w\calG_\mu$.
\end{lemma}

\vspace{.5mm}

\begin{proof}
The $T$-variety $X^{\leqslant w}_n$ is smooth by \cite{KT3}, because
 $X^{\leqslant w}$ is smooth and $p_n$ is a $B^-_n$-torsor for $n$ large enough.
We claim that  it is also quasi-projective.

Let $X_0$ be the stack of $G$-bundles on $\bbP^1$. 
We may assume that  $P_\mu$ is maximal parabolic.
Then, by the Drinfeld-Simpson theorem, a $k$-point of $X$ is the same as a 
$k$-point of $X_0$ with a trivialization of its pullback to $\Spec(k[[t]])$. 
Here $t$ is regarded
as a local coordinate at $\infty\in\bbP^1$ and we identify $B^-$ with the Iwahori subgroup in
$G(k[[t]])$. We may choose
$B^-_n$ to be the kernel of the restriction $G(k[[t]])\to G(k[t]/t^n).$
Then, a $k$-point of
$X_n=X/B_n^-$ is the same as a $k$-point of $X_0$ with a trivialization 
of its pullback to $\Spec(k[t]/t^n)$. 
We'll prove that there is an increasing system of open subsets
$\calU_m\subset X_0$ such that for each $m$ and for $n\gg0$ 
the fiber product $X_n\times _{X_0}\calU_m$ is representable by a quasi-projective variety.
This implies our claim.

Choosing a faithful representation $G\subset SL_r$
we can assume that $G=SL_r$.
So a $k$-point of $X_0$ is the same as a rank $r$ vector bundle
on  $\bbP^1$ of degree 0. 
For an integer $m>0$  let
$\calU_m(k)$ be the set of  $V$ in $X_0(k)$ with
$H^1(\bbP^1,V\otimes\calO(m))=0$ which are generated by global sections. It is the set of $k$-points of an open substack $\calU_m$ of $X_0$.
Note that $\calU_m\subset\calU_{m+1}$ and $X_0=\bigcup_m\calU_m$.
Now, the set
$\calY_m(k)$ of pairs $(V,b)$ where $V\in \calU_m(k)$  and $b$ is a  basis of 
$H^0(\bbP^1,V\otimes\calO(m))$
is the set of $k$-points of a quasi-projective variety $\calY_m$ by 
the Grothendieck theory of Quot-schemes.
Further, there is a canonical $GL_{r(m+1)}$-action on $\calY_m$ such that the morphism
$\calY_m\to\calU_m$, $(V,b)\mapsto V$ is a $GL_{r(m+1)}$-bundle. 
%So we have $\calU_m=[\calY_m/GL_{r(m+1)}]$. 
Now,  for $n\gg 0$
the fiber product $X_n\times _{X_0}\calU_m$ is representable by a quasi-projective variety,
see e.g., \cite[thm.~5.0.14]{W}.

Next, note that  $X^{\leqslant w}_n$ is recovered by the open subsets 
$V^x_n=p_n(V^x)$ with $x\leqslant w.$ Each of them contains a unique fixed point 
under the $T$-action %which is equal to $p_n(x_X)$, 
and finitely many one-dimensional orbits.

Finally,  the parity vanishing holds :
since $IC(X^{[x,w]})=p_n^*IC(X^{[x,w]}_{n})$ we have 
$$IC(X^{[x,w]}_n)_{X^y_n}=\bigoplus_ik_{X^y_n}[-l(y)][l(y)-l(x)-2i]^{\oplus Q^{\mu,-1}_{x,y,i}}$$
by \cite[thm.~1.3]{KT2}.
The change in the degrees with respect
to Section \ref{sec:HT} is due to the change of convention for perverse sheaves mentioned above.
\end{proof}

\vspace{2mm}

Now we set $V=\bft^*$ and we consider the moment graph ${}^w\calG^\vee$.

\vspace{2mm}

\begin{prop}
\label{prop:IH}
We have

(a) $H_T(X^{\leqslant w})={}^w\!\bar Z^\vee_{S,\mu,-}$ 
and $H(X^{\leqslant w})={}^w\!\bar Z^\vee_{\mu,-}$ as graded $k$-algebras,

(b) $IH_T(X^{[x,w]})={}^w\!\bar B_{S,\mu,-}^\vee(x)$ as a graded
${}^w\!\bar Z_{S,\mu,-}$-module,

(c) $IH(X^{[x,w]})={}^w\!\bar B_{\mu,-}^\vee(x)$ as a graded
${}^w\!\bar Z_{\mu,-}^\vee$-module.
\end{prop}

\vspace{.5mm}

\begin{proof}
Assuming $n$ to be large enough we may assume that
$H_T(X^{\leqslant w})=H_T(X^{\leqslant w}_n),$
$IH_T(X^{[x,w]})=IH_T(X^{[x,w]}_n)$, etc.
By Lemma \ref{lem:good} the $S$-module $H_T(X^{\leqslant w}_n)$ is free.
Thus we can apply the localization theorem \cite[thm.~6.3]{GKM}, which
proves  $(a)$.

Now, we concentrate on $(b)$. 
The graded $k$-module $IH(X^{[x,w]}_n)$ vanishes in odd degree 
by Proposition \ref{prop:IC/ICT/IH} and Lemma \ref{lem:good}. 
Thus, applying \cite{BM} to
$X_n^{[x,w]},$ we get a graded ${}^w\bar Z_{S,\mu,-}$-module isomorphism 
$IH^*_T(X^{[x,w]}_n)={}^w\!\bar B^\vee_{S,\mu,-}(x).$

Part $(c)$ follows from $(b)$, Proposition \ref{prop:IC/ICT/IH} and Lemma \ref{lem:good}.
\end{proof}

\vspace{2mm}

\begin{cor}
\label{cor:B1}
We have a graded $S$-module isomorphism
$${}^w\!\bar B_{S,\mu,-}^\vee(x)_{y}=
\bigoplus_{i\geqslant 0}(S\langle -l(x)-2i\rangle)^{\oplus Q^{\mu,-1}_{x,y,i}}.$$
\end{cor}

\vspace{.5mm}

\begin{proof}
Apply Proposition \ref{prop:IH} and \cite[thm.~1.3(i)]{KT2}.
\end{proof}

\vspace{2mm}

\begin{prop} 
\label{prop:IC}
For each $x,y\leqslant w$ we have

(a) $\sum_{i\in\bbZ}t^i\dim k\Ext^i_{\bfD_T(X^{\leqslant w})}
\bigl(IC_T(X^{[x,w]}),IC_T(X^{[y,w]})\bigr)=
\sum_zQ_\mu(t)_{x,z}Q_\mu(t)_{y,z},$

(b) $\Ext_{\bfD_T(X^{\leqslant w})}
\bigl(IC_T(X^{[x,w]}),IC_T(X^{[y,w]})\bigr)
=\Hom_{H_T(X^{\leqslant w})}\bigl(IH_T(X^{[x,w]}),IH_T(X^{[y,w]})\bigr),$

(c) $\Ext_{\bfD(X^{\leqslant w})}
\bigl(IC(X^{[x,w]}),IC(X^{[y,w]})\bigr)=\Hom_{H(X^{\leqslant w})}\bigl(IH(X^{[x,w]}),IH(X^{[y,w]})\bigr),$

(d) $\Ext_{\bfD(X^{\leqslant w})}
\bigl(IC(X^{[x,w]}),IC(X^{[y,w]})\bigr)=k\Ext_{\bfD_T(X^{\leqslant w})}
\bigl(IC_T(X^{[x,w]}),IC_T(X^{[y,w]})\bigr).$
\end{prop}

\vspace{.5mm}

\begin{proof}
Apply Proposition \ref{prop:IC/ICT/IH} and Lemma \ref{lem:good}.
\end{proof}

\vspace{2mm}

Finally, we obtain the following.

\vspace{2mm}

\begin{cor}
\label{cor:B2}
We have a graded $k$-algebra isomorphism
$
{}^w\!\bar \scrA_{\mu,-}^\vee=
\End_{{}^w\!Z^\vee_{\mu}}\bigl({}^w\!\bar B^\vee_{\mu,-}\bigr)^\op.
$ 
\end{cor}

\vspace{.5mm}

\begin{proof}
Apply Propositions \ref{prop:IH},  \ref{prop:IC}.
\end{proof}

\end{document}